\documentclass[12pt,a4paper]{amsart}
\usepackage{amssymb,eucal}
\usepackage[cmtip,arrow]{xy}
\usepackage{pb-diagram,pb-xy}
\usepackage{tikz-cd}
\usepackage{hyperref}

\calclayout

\renewcommand{\le}{\leqslant}
\renewcommand{\ge}{\geqslant}

\newcommand{\+}{\nobreakdash-}
\renewcommand{\:}{\colon}
\renewcommand{\.}{\mskip.5\thinmuskip}

\newcommand{\ot}{\otimes}

\newcommand{\rarrow}{\longrightarrow}

\newcommand{\bu}{{\text{\smaller\smaller$\scriptstyle\bullet$}}}
\newcommand{\lrarrow}{\.\relbar\joinrel\relbar\joinrel\rightarrow\.}

\DeclareMathOperator{\Hom}{Hom}
\DeclareMathOperator{\Ext}{Ext}
\DeclareMathOperator{\Tor}{Tor}
\DeclareMathOperator{\Id}{Id}

\newcommand{\B}{\mathcal B}
\newcommand{\C}{\mathcal C}
\newcommand{\D}{\mathcal D}

\newcommand{\sA}{\mathsf A}
\newcommand{\sB}{\mathsf B}
\newcommand{\sC}{\mathsf C}
\newcommand{\sD}{\mathsf D}
\newcommand{\sE}{\mathsf E}
\newcommand{\sF}{\mathsf F}
\newcommand{\sG}{\mathsf G}
\newcommand{\sH}{\mathsf H}
\newcommand{\sJ}{\mathsf J}
\newcommand{\sK}{\mathsf K}
\newcommand{\sP}{\mathsf P}
\newcommand{\sS}{\mathsf S}

\newcommand{\DG}{\mathsf{DG}}
\newcommand{\Hot}{\mathsf{Hot}}

\newcommand{\boR}{\mathbb R}
\newcommand{\boL}{\mathbb L}
\newcommand{\boZ}{\mathbb Z}
\newcommand{\boQ}{\mathbb Q}

\renewcommand{\b}{{\mathsf{b}}}
\newcommand{\co}{{\mathsf{co}}}
\newcommand{\ctr}{{\mathsf{ctr}}}
\newcommand{\abs}{{\mathsf{abs}}}
\newcommand{\sico}{{\mathsf{sico}}}
\newcommand{\sictr}{{\mathsf{sictr}}}

\newcommand{\inj}{{\mathsf{inj}}}
\newcommand{\proj}{{\mathsf{proj}}}
\renewcommand{\flat}{{\mathsf{flat}}}
\newcommand{\sfpin}{{\mathsf{sfpin}}}
\newcommand{\sfpfl}{{\mathsf{sfpfl}}}
\newcommand{\spl}{{\mathsf{spl}}}

\newcommand{\rop}{{\mathrm{op}}}

\newcommand{\tors}{{\operatorname{\mathsf{-tors}}}}
\newcommand{\ctra}{{\operatorname{\mathsf{-ctra}}}}

\newcommand{\modl}{{\operatorname{\mathsf{--mod}}}}
\newcommand{\modr}{{\operatorname{\mathsf{mod--}}}}
\newcommand{\bimod}{{\operatorname{\mathsf{--mod--}}}}
\newcommand{\comodl}{{\operatorname{\mathsf{--comod}}}}

\newcommand{\contra}{{\operatorname{\mathsf{--contra}}}}

\newcommand{\Section}[1]{\bigskip\section{#1}\medskip}

\theoremstyle{plain}
\newtheorem{thm}{Theorem}[section]

\newtheorem{lem}[thm]{Lemma}
\newtheorem{prop}[thm]{Proposition}
\newtheorem{cor}[thm]{Corollary}
\theoremstyle{definition}

\newtheorem{exs}[thm]{Examples}
\newtheorem{rem}[thm]{Remark}

\begin{document}

\title{Pseudo-dualizing complexes \\ and pseudo-derived categories}
\author{Leonid Positselski}

\address{Institute of Mathematics, Czech Academy of Sciences,
\v Zitn\'a~25, 115~67 Prague~1, Czech Republic; and
\newline\indent Laboratory of Algebraic Geometry, National Research
University Higher School of Economics, Moscow 119048; and
\newline\indent Sector of Algebra and Number Theory, Institute for
Information Transmission Problems, Moscow 127051, Russia; and
\newline\indent Department of Mathematics, Faculty of Natural Sciences,
University of Haifa, Mount Carmel, Haifa 31905, Israel}
\email{posic@mccme.ru}

\begin{abstract}
 The definition of a pseudo-dualizing complex is obtained from that 
of a dualizing complex by dropping the injective dimension 
condition, while retaining the finite generatedness and homothety
isomorphism conditions.
 In the specific setting of a pair of associative rings, we show that
the datum of a pseudo-dualizing complex induces a triangulated
equivalence between a pseudo-coderived category and
a pseudo-contraderived category.
 The latter terms mean triangulated categories standing ``in between''
the conventional derived category and the coderived or
the contraderived category.
 The constructions of these triangulated categories use appropriate
versions of the Auslander and Bass classes of modules.
 The constructions of derived functors providing the triangulated
equivalence are based on a generalization of a technique developed
in our previous paper~\cite{Pmgm}.
\end{abstract}

\maketitle

\tableofcontents

\section*{Introduction}
\medskip

\subsection{{}} \label{pure-types-subsecn}
 According to the philosophy elaborated in the introduction
to~\cite{Pmgm}, the choice of a dualizing complex induces
a triangulated equivalence between the coderived category of
(co)modules and the contraderived category of (contra)modules,
while in order to construct an equivalence between the conventional
derived categories of (co)modules and (contra)modules one needs
a dedualizing complex.
 In particular, an associative ring $A$ is a dedualizing complex of
bimodules over itself, while a coassociative coalgebra $\C$ over
a field~$k$ is a dualizing complex of bicomodules over itself.
 The former assertion refers to the identity equivalence
\begin{equation} \label{taut}
 \sD(A\modl)=\sD(A\modl),
\end{equation}
while the latter one points to the natural triangulated equivalence
between the coderived category of comodules and the contraderived
category of contramodules
\begin{equation} \label{co-contra}
 \sD^\co(\C\comodl)\simeq\sD^\ctr(\C\contra),
\end{equation}
known as the \emph{derived comodule-contramodule
correspondence}~\cite[Sections~0.2.6\+-7 and~5.4]{Psemi},
\cite[Sections~4.4 and~5.2]{Pkoszul}.

 Given a left coherent ring $A$ and a right coherent ring $B$,
the choice of a dualizing complex of $A$\+$B$\+bimodules $D^\bu$
induces a triangulated equivalence between the coderived and
the contraderived category~\cite[Theorem~4.8]{IK},
\cite[Theorem~4.5]{Pfp}
\begin{equation} \label{iyengar-krause}
 \sD^\co(A\modl)\simeq\sD^\ctr(B\modl).
\end{equation}
 Given a left cocoherent coalgebra $\C$ and a right cocoherent
coalgebra $\D$ over a field~$k$, the choice of a dedualizing
complex of $\C$\+$\D$\+bicomodules $\B^\bu$ induces
a triangulated equivalence between the conventional derived
categories of comodules and contramodules~\cite[Theorem~3.6]{Pmc}
\begin{equation} \label{mgmcoalg}
 \sD(\C\comodl)\simeq\sD(\D\contra).
\end{equation}

\subsection{{}}
 The equivalences~(\ref{taut}\+-\ref{mgmcoalg}) of
Section~\ref{pure-types-subsecn} are the ``pure types''.
 The more complicated and interesting triangulated equivalences
of the ``broadly understood co-contra correspondence'' kind
are obtained by mixing these pure types, or maybe rather building
these elementary blocks on top of one another.

 In particular, let $R$ be a commutative ring and $I\subset R$ be
an ideal.
 An $R$\+module $M$ is said to be \emph{$I$\+torsion} if
$$
 R[s^{-1}]\ot_RM=0 \qquad\text{for all \,$s\in I$}.
$$
 Clearly, it suffices to check this condition for a set of generators
$\{s_j\}$ of the ideal~$I$.
 An $R$\+module $P$ is said to be an \emph{$I$\+contramodule} if
$$
 \Hom_R(R[s^{-1}],P)=0=\Ext^1_R(R[s^{-1}],P)
 \qquad\text{for all \,$s\in I$}.
$$
 Once again, it suffices to check these conditions for a set of
generators $\{s_j\}$ of the ideal~$I$ \cite[Theorem~5 and
Lemma~7(1)]{Yek}, \cite[Theorem~5.1]{Pcta}.
 Both the full subcategory of $I$\+torsion $R$\+modules $R\modl_{I\tors}
\subset R\modl$ and the full subcategory of $I$\+contramodule
$R$\+modules $R\modl_{I\ctra}\subset R\modl$ are abelian categories.

 Assume for simplicity that $R$ is a Noetherian ring.
 Then, using what the paper~\cite{Pmgm} calls a \emph{dedualizing
complex} $B^\bu$ for the ring $R$ with the ideal $I\subset R$,
one can construct a triangulated equivalence between the conventional
derived categories of the abelian categories of $I$\+torsion and
$I$\+contramodule $R$\+modules
\begin{equation} \label{matlis-greenlees-may}
 \sD(R\modl_{I\tors})\simeq\sD(R\modl_{I\ctra}).
\end{equation}
 This result can be generalized to the so-called weakly proregular
finitely generated ideals $I$ in the sense of~\cite{Sch,PSY} in
not necessarily Noetherian commutative rings~$R$
\cite[Corollary~3.5 or Theorem~5.10]{Pmgm}.

 Using what is generally known as a \emph{t\+dualizing}
(torsion-dualizing) \emph{complex}~\cite[Definition~2.5.1]{AJL},
\cite[Section~3]{PSY2} (called a ``dualizing complex for a ring $R$
with an ideal $I\subset R$'' in~\cite{Pcosh}), one can construct
a triangulated equvalence between the coderived category of
$I$\+torsion $R$\+modules and the contraderived category of
$I$\+contramodule $R$\+modules~\cite[Theorem~C.1.4]{Pcosh}
(see also~\cite[Theorem~C.5.10]{Pcosh})
\begin{equation} \label{covariant-serre-grothendieck}
 \sD^\co(R\modl_{I\tors})\simeq\sD^\ctr(R\modl_{I\ctra}).
\end{equation}
 This result can be generalized from affine Noetherian formal schemes
to ind-affine ind-Noetherian or ind-coherent ind-schemes with
a dualizing complex~\cite[Theorem~D.2.7]{Pcosh}
(see also~\cite[Remark~4.10]{Pmgm}).

 Informally, one can view the $I$\+adic completion of a ring $R$ as
``a ring in the direction of $R/I$ and a coalgebra in the transversal
direction of $R$ relative to $R/I$''.
 In this sense, one can say that (the formulation of) the triangulated
equivalence~\eqref{matlis-greenlees-may} is obtained by
building~\eqref{mgmcoalg} on top of~\eqref{taut}, while (the idea of)
the triangulated equivalence~\eqref{covariant-serre-grothendieck}
is the result of bulding~\eqref{co-contra} on top
of~\eqref{iyengar-krause}.

\subsection{{}}
 A number of other triangulated equivalences appearing in
the present author's work can be described as mixtures of
some of the equivalences~(\ref{taut}\+-\ref{mgmcoalg}).
 In particular, the equivalence between the coderived category
of comodules and the contraderived category of contramodules
over a pair of corings over associative rings
in~\cite[Corollaries~B.4.6 and~B.4.10]{Pcosh}
is another way of building~\eqref{co-contra} on top
of~\eqref{iyengar-krause}.

 The equivalence between the conventional derived
categories of semimodules and semicontramodules
in~\cite[Theorem~4.3]{Pmc} is obtained by
building~\eqref{taut} on top of~\eqref{mgmcoalg}.
 The equivalence between the semicoderived and
the semicontraderived categories of modules
in~\cite[Theorem~5.6]{Pfp} is the result of
building~\eqref{taut} on top of~\eqref{iyengar-krause}.

 The most deep and difficult in this series of triangulated
equivalences is the \emph{derived semimodule-semicontramodule
correspondence} of~\cite[Section~0.3.7]{Psemi} (see the proof
in a greater generality in~\cite[Section~6.3]{Psemi}).
  The application of this triangulated equivalence to the categories
$\mathsf O$ and $\mathsf O^\ctr$ over Tate Lie algebras
in~\cite[Corollary~D.3.1]{Psemi} is of particular importance.
 This is the main result of the book~\cite{Psemi}.
 It can be understood as obtainable by building~\eqref{taut} on top
of~\eqref{co-contra}.
 
 Note that all the expressions like ``can be obtained by'' or
``is the result of'' above refer, at best, to the \emph{formulations}
of the mentioned theorems, rather than to their \emph{proofs}.
 For example, the derived semimodule-semicontramodule
correspondence, even in the generality of~\cite[Section~0.3.7]{Psemi},
is a difficult theorem.
 There is no way to deduce it from the easy~\eqref{co-contra} and
the trivial~\eqref{taut}.
 The formulations of~\eqref{co-contra} and~\eqref{taut} serve as
an inspiration and the guiding heuristics for arriving to
the formulation of the derived semimodule-semicontramodule
correspondence.
 Subsequently, one has to develop appropriate techniques leading
to a proof.

\subsection{{}}
 More generally, beyond building things on top of one another, one
may wish to develop notions providing a kind of ``smooth
interpolation'' between various concepts.
 In particular, the notion of a discrete module over a topological ring
can be viewed as interpolating between those of a module over
a ring and a comodule over a coalgebra over a field, while the notion
of a contramodule over a topological ring (see~\cite[Remark~A.3]{Psemi}
or~\cite{PR}) interpolates between those of a module over a ring and
a contramodule over a coalgebra over a field.

 The notion of a \emph{pseudo-dualizing complex} (known as
a ``semi-dualizing complex'' in the literature) interpolates between
those of a dualizing and a dedualizing complex.
 Similarly, the notions of a \emph{pseudo-coderived} and
a \emph{pseudo-contraderived} category interpolate between those
of the conventional derived category and the co- or contraderived
category.
 The aim of this paper is to construct the related interpolation
between the triangulated equivalences~\eqref{taut}
and~\eqref{iyengar-krause}.

 Let us mention that a family of ``intermediate'' model structures
between conventional derived ones (``of the first kind'') and
the coderived ones (``of the second kind'') was constructed, in
the case of DG\+coalgebras and DG\+comodules, in
the paper~\cite{DCH}.
 There is some vague similarity between our construction and
the one in~\cite{DCH}.
 The differences are that we start from a  pseudo-dualizing complex
and obtain a triangulated equivalence for our intermediate triangulated
categories in the context of the comodule-contramodule correspondence,
while the authors of~\cite{DCH} start from a twisting cochain and
obtain a Quillen adjunction in the context of Koszul duality.

\subsection{{}} \label{two-associative-rings-introd}
 Let $A$ and $B$ be associative rings.
 A \emph{pseudo-dualizing complex} $L^\bu$ for the rings $A$ and $B$
is a finite complex of $A$\+$B$\+bimodules satisfying the following
two conditions:
\begin{enumerate}
\renewcommand{\theenumi}{\roman{enumi}}
\setcounter{enumi}{1}
\item as a complex of left $A$\+modules, $L^\bu$ is quasi-isomorphic
to a bounded above complex of finitely generated projective
$A$\+modules, and similarly, as a complex of right $B$\+modules,
$L^\bu$ is quasi-isomorphic to a bounded above complex of finitely
generated projective $B$\+modules;
\item the homothety maps $A\rarrow\Hom_{\sD^\b(\modr B)}(L^\bu,L^\bu[*])$
and $B^\rop\rarrow\Hom_{\sD^\b(A\modl)}(L^\bu,L^\bu[*])$ are
isomorphisms of graded rings. \emergencystretch=3em\hbadness=2900
\end{enumerate}
 This definition is obtained by dropping the injectivity (or
finite injective dimension, or fp\+injectivity, etc.)\ condition~(i)
from the definition of a \emph{dualizing} or (``cotilting'') complex
of $A$\+$B$\+bimodules $D^\bu$ in the papers~\cite{Miy,YZ,CFH,Pfp},
removing the Noetherianness/coherence conditions on the rings $A$
and $B$, and rewriting the finite generatedness/presentability
condition~(ii) accordingly.

 For example, when the rings $A$ and $B$ coincide, the one-term
complex $L^\bu=A=B$ becomes the simplest example of
a pseudo-dualizing complex.
 This is what can be called a \emph{dedualizing complex} in this
context.
 More generally, a ``dedualizing complex of $A$\+$B$\+bimodules'' is
the same thing as a ``(two-sided) tilting complex'' $T^\bu$ in
the sense of Rickard's derived Morita theory~\cite{Ric,Ric2,DS}.

 What in our terminology would be called ``pseudo-dualizing complexes
of modules over commutative Noetherian rings'' were studied in
the paper~\cite{Chr} and the references therein under some other names,
such as ``semi-dualizing complexes''.
 What the authors call ``semidualizing bimodules'' for pairs of
associative rings were considered in the paper~\cite{HW}.
 We use this other terminology of our own in this paper, because in
the context of the present author's work the prefix ``semi'' means
something related but different and more narrow (as in~\cite{Psemi}
and~\cite[Sections~5\+-6]{Pfp}).

 The main result of this paper provides the following commutative
diagram of triangulated functors associated with a pseudo-dualizing
complex of $A$\+$B$\+bimodules~$L^\bu$:
\begin{equation} \label{two-associative-rings-triangulated-diagram}
\begin{diagram}
\node{\sD^\co(A\modl)}\arrow{s,A}
\node[2]{\sD^\ctr(B\modl)}\arrow{s,A} \\
\node{\sD_\prime^{L^\bu}(A\modl)}\arrow{s,A}\arrow[2]{e,=}
\node[2]{\sD_{\prime\prime}^{L^\bu}(B\modl)}\arrow{s,A} \\
\node{\sD'_{L^\bu}(A\modl)}\arrow{s,A}\arrow[2]{e,=}
\node[2]{\sD''_{L^\bu}(B\modl)}\arrow{s,A} \\
\node{\sD(A\modl)}\node[2]{\sD(B\modl)}
\end{diagram}
\end{equation}
 Here the vertical arrows are Verdier quotient functors, while
the horizontal double lines are triangulated equivalences.

 Thus $\sD_\prime^{L^\bu}(A\modl)$ and $\sD'_{L^\bu}(A\modl)$ are
certain intermediate triangulated categories between the coderived
category of left $A$\+modules $\sD^\co(A\modl)$ and their
conventional unbounded derived category $\sD(A\modl)$.
 Similarly, $\sD_{\prime\prime}^{L^\bu}(B\modl)$ and
$\sD''_{L^\bu}(B\modl)$ are certain intermediate triangulated categories
between the contraderived category of left $B$\+modules
$\sD^\ctr(B\modl)$ and their conventional unbounded derived category
$\sD(B\modl)$.
 These intermediate triangulated quotient categories depend on, and
are determined by, the choice of a pseudo-dualizing complex $L^\bu$
for a pair of associative rings $A$ and~$B$.

 The triangulated category $\sD'_{L^\bu}(A\modl)$ is called
the \emph{lower pseudo-coderived category} of left $A$\+modules
corresponding to the pseudo-dualizing complex~$L^\bu$.
 The triangulated category $\sD''_{L^\bu}(B\modl)$ is called
the \emph{lower pseudo-contraderived category} of left $B$\+modules
corresponding to the pseudo-dualizing complex~$L^\bu$.
 The triangulated category $\sD_\prime^{L^\bu}(A\modl)$ is called
the \emph{upper pseudo-coderived category} of left $A$\+modules
corresponding to~$L^\bu$.
 The triangulated category $\sD_{\prime\prime}^{L^\bu}(B\modl)$ is
called the \emph{upper pseudo-contraderived category} of left
$B$\+modules corresponding to~$L^\bu$.
 The choice of a pseudo-dualizing complex $L^\bu$ also induces
triangulated equivalences $\sD'_{L^\bu}(A\modl)\simeq
\sD''_{L^\bu}(B\modl)$ and $\sD_\prime^{L^\bu}(A\modl)\simeq
\sD_{\prime\prime}^{L^\bu}(A\modl)$ forming the commutative
diagram~\eqref{two-associative-rings-triangulated-diagram}.

 In particular, when $L^\bu=D^\bu$ is a dualizing complex, i.~e.,
the condition~(i) of~\cite[Section~4]{Pfp} is satisfied, assuming
additionally that all fp\+injective left $A$\+modules have finite
injective dimensions, one has $\sD_\prime^{L^\bu}(A\modl)=
\sD^\co(A\modl)$ and $\sD_{\prime\prime}^{L^\bu}(B\modl)=
\sD^\ctr(B\modl)$, that is the upper two vertical arrows in
the diagram~\eqref{two-associative-rings-triangulated-diagram}
are isomorphisms of triangulated categories.
 The upper triangulated equivalence in
the diagram~\eqref{two-associative-rings-triangulated-diagram}
coincides with the one provided by~\cite[Theorem~4.5]{Pfp}
in this case.

 When $L^\bu=A=B$, one has $\sD'_{L^\bu}(A\modl)=\sD(A\modl)$ and
$\sD''_{L^\bu}(B\modl)=\sD(B\modl)$, that is the lower two vertical
arrows in 
the diagram~\eqref{two-associative-rings-triangulated-diagram}
are isomorphisms of triangulated categories.
 The lower triangulated equivalence in
the diagram~\eqref{two-associative-rings-triangulated-diagram} is
just the identity isomorphism $\sD(A\modl)=\sD(B\modl)$ is this case.
 More generally, the lower triangulated equivalence in
the diagram~\eqref{two-associative-rings-triangulated-diagram}
corresponding to a tilting complex $L^\bu=T^\bu$ recovers
Rickard's derived Morita equivalence~\cite[Theorem~6.4]{Ric},
\cite[Theorem~3.3 and Proposition~5.1]{Ric2} (see
also~\cite[Theorem~4.2]{DS}).

\subsection{{}}
 A delicate point is that when $A=B=R$ is, e.~g., a Gorenstein
Noetherian commutative ring of finite Krull dimension, the ring $R$
itself can be chosen as a dualizing complex of $R$\+$R$\+bimodules.
 So we are in both of the above-described situations at the same time.
 Still, the derived category of $R$\+modules $\sD(R\modl)$,
the coderived category $\sD^\co(R\modl)$, and the contraderived
category $\sD^\ctr(R\modl)$ are three quite different quotient
categories of the homotopy category of (complexes of) $R$\+modules
$\Hot(R\modl)$.
 In this case, the commutative
diagram~\eqref{two-associative-rings-triangulated-diagram} takes
the form
$$
\begin{diagram}
\node{\sD^\co(R\modl)}\arrow{s,A}\arrow[2]{e,=}
\node[2]{\sD^\ctr(R\modl)}\arrow{s,A} \\
\node{\sD(R\modl)}\arrow[2]{e,=}\node[2]{\sD(R\modl)}
\end{diagram}
$$

 More precisely, the two Verdier quotient functors $\Hot(R\modl)\rarrow
\sD^\co(R\modl)$ and $\Hot(R\modl)\allowbreak\rarrow\sD^\ctr(R\modl)$
both factorize naturally through the Verdier quotient functor
$\Hot(R\modl)\rarrow\sD^\abs(R\modl)$ from the homotopy category onto
the absolute derived category of $R$\+modules $\sD^\abs(R\modl)$.
 But the two resulting Verdier quotient functors
$\sD^\abs(R\modl)\rarrow\sD^\co(R\modl)$ and $\sD^\abs(R\modl)\rarrow
\sD^\ctr(R\modl)$ do \emph{not} form a commutative triangle with
the equivalence $\sD^\co(R\modl)\simeq\sD^\ctr(R\modl)$.
 Rather, they are the two adjoint functors on the two sides to
the fully faithful embedding of a certain (one and the same)
triangulated subcategory in $\sD^\abs(R\modl)$
\cite[proof of Theorem~3.9]{Pkoszul}.

 This example shows that one cannot hope to have a procedure recovering
the conventional derived category $\sD(A\modl)=\sD(B\modl)$ from
the dedualizing complex $L^\bu=A=B$, and at the same time recovering
the coderived category $\sD^\co(A\modl)$ and the contraderived category
$\sD^\ctr(B\modl)$ from a dualizing complex $L^\bu=D^\bu$.
 Thus the distinction between the lower and and the upper
pseudo-co/contraderived category constructions is in some sense
inevitable.

\subsection{{}}
 Before we finish this introduction, let us say a few words about
where the pseudo-coderived and pseudo-contraderived categories
come from in Section~\ref{two-associative-rings-introd}.
 We use ``pseudo-derived categories'' as a generic term for
the pseudo-coderived and pseudo-contraderived categories.
 The constructions of such triangulated categories that we use in
this paper were originally introduced for the purposes of
the infinitely generated Wakamatsu tilting theory~\cite[Section~4]{PS}
(an even more general approach to exotic derived categories is
suggested in~\cite[Section~5]{PS}).

 The pseudo-derived categories are constructed as the conventional
unbounded derived categories of certain exact subcategories
$\sE_\prime\subset\sE'\subset\sA$ and $\sF_{\prime\prime}\subset\sF''
\subset\sB$ in the abelian categories $\sA=A\modl$ and $\sB=B\modl$.
 The idea is that shrinking an abelian (or exact) category to its
exact subcategory leads, under certain assumptions, to a bigger
derived category, as complexes in the exact subcategory are
considered up to a finer equivalence relation in the derived
category construction.

 In the situation at hand, the larger subcategories $\sE'$ and $\sF''$
are our versions of what are called the \emph{Auslander and Bass
classes} in the literature~\cite{Chr,FJ,EJLR,CFH,HW}.
 Specifically, $\sF''$ is the Auslander class and $\sE'$ is
the Bass class.
 The two full subcategories $\sE_\prime$ and $\sF_{\prime\prime}$ are
certain natural smaller classes.
 One can say, in some approximate sense, that $\sE'$ and $\sF''$
are the \emph{maximal corresponding classes}, while $\sE_\prime$ and
$\sF_{\prime\prime}$ are the \emph{minimal corresponding classes} in
the categories $\sA$ and~$\sB$.

 More precisely, there is a natural single way to define the full
subcategories $\sE'\subset\sA$ and $\sF''\subset\sB$ when
the pseudo-dualizing complex $L^\bu$ is a one-term complex.
 In the general case, we have two sequences of embedded subcategories
$\sE_{d_1}\subset\sE_{d_1+1}\subset\sE_{d_1+2}\subset\dotsb\subset\sA$
and $\sF_{d_1}\subset\sF_{d_1+1}\subset\sF_{d_1+2}\subset\dotsb\subset
\sB$ indexed by large enough integers.
 All the subcategories $\sE_{l_1}$ with varying index $l_1=d_1$,
$d_1+1$, $d_1+2$,~\dots\ are ``the same up to finite homological
dimension'', and so are all the subcategories~$\sF_{l_1}$.
 Hence the triangulated functors $\sD(\sE_{l_1})\rarrow\sD(\sE_{l_1+1})$
and $\sD(\sF_{l_1})\rarrow\sD(\sF_{l_1+1})$ induced by the exact
embeddings $\sE_{l_1}\rarrow\sE_{l_1+1}$ and $\sF_{l_1}\rarrow
\sF_{l_1+1}$ are triangulated equivalences, so the pseudo-derived
categories $\sD'_{L^\bu}(A\modl)=\sD(\sE_{l_1})$ and
$\sD''_{L^\bu}(B\modl)=\sD(\sF_{l_1})$ do not depend on the choice of
a number~$l_1$.

 The idea of the construction of the triangulated equivalence between
the two lower pseudo-derived categories is that the functor
$\sD'_{L^\bu}(\sA)\rarrow\sD''_{L^\bu}(\sB)$ should be a version of
$\boR\Hom_A(L^\bu,{-})$, while the inverse functor $\sD''_{L^\bu}(\sB)
\rarrow\sD'_{L^\bu}(\sA)$ is a version of derived tensor product
$L^\bu\ot_B^\boL{-}$.
 The full subcategories $\sE_{l_1}\subset\sA$ and $\sF_{l_1}\subset\sB$
are defined by the conditions of uniform boundedness of cohomology
of such Hom and tensor product complexes (hence dependence on
a fixed bound~$l_1$) and the composition of the two operations leading
back to the original object.

 The point is that the two functors $\boR\Hom_A(L^\bu,{-})$ and
$L^\bu\ot_B^\boL{-}$ are mutually inverse when viewed as acting between
the pseudo-derived categories $\sD(\sE)$ and $\sD(\sF)$, but objects
of the pseudo-derived categories are complexes viewed up to a more
delicate equivalence relation than in the conventional derived
categories $\sD(\sA)$ and $\sD(\sB)$.
 When this subtlety is ignored, the two functors cease to be mutually
inverse, generally speaking, and such mutual inverseness needs to be
enforced as an additional adjustness restriction on the objects one
is working with.

 Similarly, there is a natural single way to define the full
subcategories $\sE_\prime\subset\sA$ and $\sF_{\prime\prime}\subset\sF$
when the pseudo-dualizing complex $L^\bu$ is a one-term complex.
 In the general case, we have two sequences of embedded subcategories
$\sE^{d_2}\supset\sE^{d_2+1}\supset\sE^{d_2+2}\supset\dotsb$ in $\sA$
and $\sF^{d_2}\supset\sF^{d_2+1}\supset\sF^{d_2+2}\supset\dotsb$
in $\sB$, indexed by large enough integers.
 As above, all the subcategories $\sE^{l_2}$ with varying $l_2=d_2$,
$d_2+1$, $d_2+2$,~\dots\ are ``the same up to finite homological
dimension'', and so are all the subcategories~$\sF^{l_2}$.
 Hence the triangulated functors $\sD(\sE^{l_2+1})\rarrow\sD(\sE^{l_2})$
and $\sD(\sF^{l_2+1})\rarrow\sD(\sF^{l_2})$ induced by the exact
embeddings $\sE^{l_2+1}\rarrow\sE^{l_2}$ and $\sF^{l_2+1}\rarrow
\sF^{l_2}$ are triangulated equivalences, so the pseudo-derived
categories $\sD_\prime^{L^\bu}(A\modl)=\sD(\sE^{l_2})$ and
$\sD_{\prime\prime}^{L^\bu}(B\modl)=\sD(\sF^{l_2})$ do not depend on
the choice of a number~$l_2$.

 The triangulated equivalence between the two upper pseudo-derived
categories is also provided by some versions of derived functors
$\boR\Hom_A(L^\bu,{-})$ and $L^\bu\ot_B^\boL{-}$.
 The full subcategories $\sE^{l_2}\subset\sA$ and $\sF^{l_2}\subset\sB$
are produced by a kind of generation process.
 One starts from declaring that all the injectives in $\sA$ belong
to $\sE^{l_2}$ and all the projectives in $\sB$ belong to~$\sF^{l_2}$.
 Then one proceeds with generating further objects of $\sF^{l_2}$ by
applying $\boR\Hom_A(L^\bu,{-})$ to objects of $\sE^{l_2}$, and further
objects of $\sE^{l_2}$ by applying $L^\bu\ot_B^\boL{-}$ to objects
of~$\sF^{l_2}$.
 One needs to resolve the complexes so obtained to produce objects
of the abelian module categories, and the number~$l_2$ indicates
the length of the resolutions used.
 More objects are added to $\sE^{l_2}$ and $\sF^{l_2}$ to make these
full subcategories closed under certain operations.

 We refer to the main body of the paper for further details.

\subsection{{}} \label{introd-adjoints}
 Inspired by the recent paper~\cite{BT}, in the last three sections of
the present paper we address the question of existence of left and
right adjoint functors to the Verdier quotient functors in
the diagram~\eqref{two-associative-rings-triangulated-diagram}
from Section~\ref{two-associative-rings-introd}.
 More precisely, for any pseudo-dualizing complex of $A$\+$B$\+bimodules
$L^\bu$ the adjoint functors shown by curvilinear arrows on
the following diagram exist: {\hfuzz=2.2pt
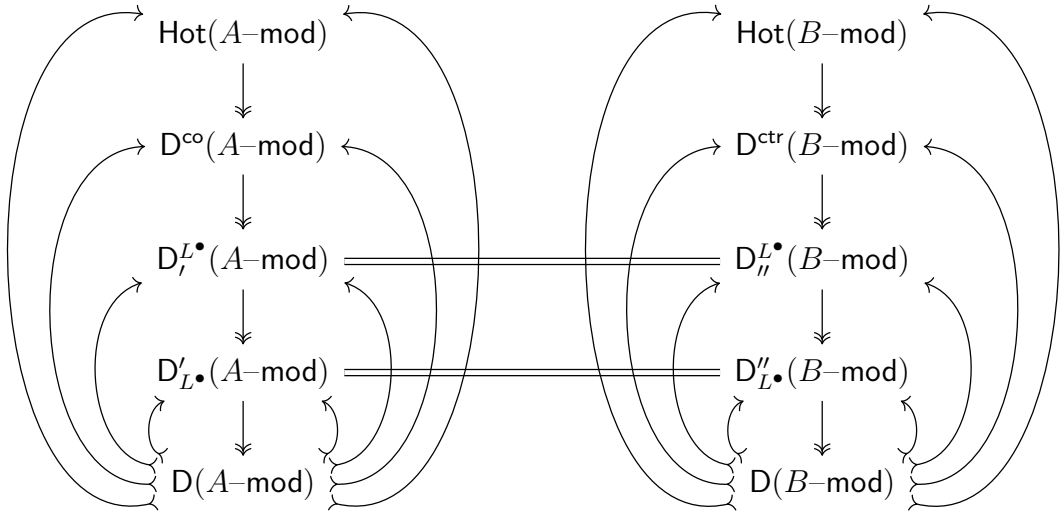
\begin{equation} \label{easy-adjoints-diagram}
\begin{tikzcd}
\Hot(A\modl) \arrow[d, two heads] &&&&&
\Hot(B\modl) \arrow[d, two heads] \\
\sD^\co(A\modl) \arrow[d, two heads] &&&&&
\sD^\ctr(B\modl) \arrow[d, two heads] \\
\sD^{L^\bu}_{\prime}(A\modl) \arrow[d, two heads]
\arrow[rrrrr, Leftrightarrow, no head, no tail] &&&&&
\sD^{L^\bu}_{\prime\prime}(B\modl) \arrow[d, two heads] \\
\sD'_{L^\bu}(A\modl) \arrow[d, two heads]
\arrow[rrrrr, Leftrightarrow, no head, no tail] &&&&&
\sD''_{L^\bu}(B\modl) \arrow[d, two heads] \\
\sD(A\modl) \arrow[uuuu, tail, bend left=102]
\arrow[uuuu, tail, bend right=102]
\arrow[uuu, tail, bend left=90] \arrow[uuu, tail, bend right=90]
\arrow[uu, tail, bend left=78] \arrow[uu, tail, bend right=78]
\arrow[u, tail, bend left=70] \arrow[u, tail, bend right=70]
&&&&& \sD(B\modl) \arrow[uuuu, tail, bend left=102]
\arrow[uuuu, tail, bend right=102]
\arrow[uuu, tail, bend left=90] \arrow[uuu, tail, bend right=90]
\arrow[uu, tail, bend left=78] \arrow[uu, tail, bend right=78]
\arrow[u, tail, bend left=70] \arrow[u, tail, bend right=70]
\end{tikzcd}
\end{equation}
 The functors} shown by curvilinear arrows, being adjoints to Verdier
quotient functors, are fully faithful.
 Their existence follows straightforwardly from the existence of 
left and right adjoint functors to the natural Verdier quotient functor
$\Hot(R\modl)\rarrow\sD(R\modl)$ for any associative ring~$R$.

 Moreover, assuming that the ring $A$ is left coherent, fp\+injective
left $A$\+modules have finite injective dimensions, the ring $B$ is
right coherent, and flat left $B$\+modules have finite projective
dimensions, the adjoint functors shown by the curvilinear arrows
on the following diagram also exist:
\begin{equation} \label{hard-adjoints-diagram}
\begin{tikzcd}
\sD^\co(A\modl) \arrow[d, two heads] &&
\sD^\ctr(B\modl) \arrow[d, two heads] \\
\sD'_{L^\bu}(A\modl) \arrow[d, two heads]
\arrow[rr, Leftrightarrow, no head, no tail] 
\arrow[u, tail, bend left=68] \arrow[u, tail, bend right=68] &&
\sD''_{L^\bu}(B\modl) \arrow[d, two heads] 
\arrow[u, tail, bend left=68] \arrow[u, tail, bend right=68] \\
\sD(A\modl) \arrow[uu, tail, bend left=80]
\arrow[uu, tail, bend right=80]
\arrow[u, tail, bend left=68] \arrow[u, tail, bend right=68]
&& \sD(B\modl) \arrow[uu, tail, bend left=80]
\arrow[uu, tail, bend right=80]
\arrow[u, tail, bend left=68] \arrow[u, tail, bend right=68]
\end{tikzcd}
\end{equation}
 As compared to the previous diagram, the nontrivial additional part
here is the existence of the left and right adjoints to the natural
Verdier quotient functors
$\sD^\co(A\modl)\rarrow\sD'_{L^\bu}(A\modl)$ and
$\sD^\ctr(B\modl)\rarrow\sD''_{L^\bu}(B\modl)$ between
the co/contraderived categories and the lower pseudo-derived
categories of modules over $A$ and~$B$.
 In order to prove that these functors exist, we show that
the Auslander and Bass classes of modules $\sF_{l_1}(L^\bu)
\subset B\modl$ and $\sE_{l_1}(L^\bu)\subset A\modl$ 
are \emph{deconstructible}, deduce from this the \emph{existence
of Hom sets} in the lower pseudo-derived categories, recall from
the literature that the co/contraderived categories are
\emph{compactly generated} in the above assumptions, and
observe that the Verdier quotient functors in question preserve
infinite direct sums and products.

\subsection{{}}
 To end, let us say a few words about how the results of this paper
compare to the theory developed in the papers~\cite{PS0,PS}.
 Both~\cite{PS0,PS} and this paper can be viewed as generalizations
of the classical finitely generated $n$\+tilting theory, as
developed, e.~g., in~\cite{Mi} and~\cite{CPS}, albeit in different
subsets of directions.

 Basically, using the notion of a finitely generated $n$\+tilting
module as the starting point, one can
\begin{enumerate}
\renewcommand{\theenumi}{\alph{enumi}}
\item drop the condition that the tilting module is finitely generated;
and/or
\item drop the condition that the tilting module has finite projective
dimension~$n$; and/or
\item replace a tilting module by a tilting complex; and/or
\item replace the category of modules over an associative ring by
an abelian category of more general nature (like, e.~g., a Grothendieck
category).
\end{enumerate}

 All these four lines of thought have been explored by many authors
over the years; we refer to the introduction to~\cite{PS} for
an (admittedly very possibly incomplete) collection of references.
 In particular, it seems that (a)~was first done in~\cite{CT},
(b)~was initiated in~\cite{Wak1,Wak2}, (c)~was introduced
in~\cite{Ric,Ric2}, and (d)~was first attempted in~\cite{Co}.

 Our paper~\cite{PS0} does (a) and~(d).
 In particular, the generalization to infinitely generated tilting
modules/objects is the reason why contramodule categories appear
(as the tilting hearts) in~\cite{PS0}.
 The paper~\cite{PS} does (a), (b), and~(d).
 In particular, the generalization to tilting modules/objects of
infinite projective dimension is the reason why pseudo-derived
categories were needed in~\cite[Section~4]{PS} (and an even more general
approach using t\+structures was developed in~\cite[Section~5]{PS}).

 The present paper does~(b) and~(c).
 The generalization from $\infty$\+tilting \emph{modules} to
pseudo-dualizing \emph{complexes} in this paper (cf.~\cite{HW}
and~\cite[Example~6.6]{PS}) is what makes many aspects of our exposition
here so much more complicated technically than in~\cite{PS}.
 In fact, the triangulated equivalences $\sD^\star(\sE)\simeq
\sD^\star(\sF)$ come for free in~\cite{PS0,PS}, being immediate
consequences of the equivalences of exact categories $\sE\simeq\sF$.
 In this paper, a whole (rather long and complicated) appendix is
needed just to construct the pair of adjoint derived functors providing
this triangulated equivalence.
 On the other hand, our pseudo-dualizing complexes are presumed to be
strongly finitely generated (on each side and up to quasi-isomorphism)
in this paper.
 That is why the underlying abelian categories on both sides of our
equivalences are the conventional categories of modules over
associative rings, and no contramodules appear.

\subsection{{}} \label{introd-hom-sets}
 One terminological disclaimer is in order.
 With the exception of the final
Sections~\ref{deconstructibility-secn}\+-%
\ref{existence-of-adjoints-secn}, throughout this paper we adopt
the policy of benign neglect of the issue of \emph{existence} of
Verdier quotients (cf.~\cite[Set-Theoretic Remark~10.3.3]{Wei}).

 Generally speaking, given a triangulated subcategory in a triangulated
category, their Verdier quotient (such as, e.~g., the derived category of
an abelian or exact category) may not exist as a category, in the sense
that morphisms between a fixed pair of objects in the Verdier quotient
may not form a set.
 So we tacitly presume that a Grothendieck universe has been chosen, and
all our sets (rings, modules, etc.) and categories are sets and categories
in this universe.
 If the need arises, the problem of existence of a Verdier quotient category
can be then resolved by considering such Verdier quotient as
a category in a larger universe.

 When (in the final two sections) we will need to emphasize that
a particular Verdier quotient category \emph{does} exist (in the fixed,
unchanged universe), we will say that such a Verdier quotient
\emph{has Hom sets}.

\medskip

\textbf{Acknowledgement.} 
 I~am grateful to Vladimir Hinich, Jan Trlifaj,
Jan \v St\!'ov\'\i\v cek, Hanno Becker, Amnon Yekutieli,
and Marco Tarantino for helpful discussions.
 I~would like also to thank the anonymous referee for several helpful
suggestions on the improvement of the exposition.
 The author's research is supported by research plan RVO:~67985840,
by the Israel Science Foundation grant~\#\,446/15, and by
the Grant Agency of the Czech Republic under the grant P201/12/G028.

\Section{Pseudo-Coderived and Pseudo-Contraderived Categories}
\label{pseudo-derived-introd}

 For any additive category $\sA$, we denote by $\Hot^\star(\sA)$, where
$\star=\b$, $+$, $-$, or~$\varnothing$, the categories of
(respectively bounded or unbounded) cochain complexes in $\sA$ with
the morphisms considered up to the cochain homotopy.

 Let $\sA$ be an exact category (in Quillen's sense).
 A complex in $\sA$ is said to be \emph{exact} if it is obtained by
splicing short exact sequences in~$\sA$.
 A complex in $\sA$ is \emph{acyclic} if it is homotopy equivalent to
an exact complex, or equivalently, if it is a direct summand of
an exact complex.
 The (bounded or unbounded) conventional derived category
$\sD^\star(\sA)$ with the symbol $\star=\b$, $+$, $-$, or~$\varnothing$
is defined as the quotient category of $\Hot^\star(\sA)$ by the thick
subcategory of (respectively bounded or unbounded) acyclic complexes
(see the paper~\cite{Neem0} and the overviews~\cite{Kel,Bueh}).

 A short exact sequence of complexes in $\sA$ can be viewed as
a bicomplex with three rows.
 Taking the total complex of such a bicomplex, one obtains an exact
complex in~$\sA$.
 A $\star$\+bounded complex in $\sA$ is said to be \emph{absolutely
acyclic} if it belongs to the thick subcategory of $\Hot^\star(\sA)$
generated by the totalizations of short exact sequences of
$\star$\+bounded complexes in~$\sA$.
 In fact, a $\star$\+bounded complex in $\sA$ is absolutely acyclic if
and only if it is absolutely acyclic as an unbounded
complex~\cite[Lemma~A.1.1]{Pcosh}.
 Any bounded acyclic complex is absolutely acyclic (as a bounded
complex, i.~e., for $\star=\b$).
 The \emph{absolute derived categories} $\sD^{\abs+}(\sA)$, \
$\sD^{\abs-}(\sA)$, and $\sD^\abs(\sA)$ are defined as the quotient
categories of the respective homotopy categories $\Hot^+(\sA)$, \
$\Hot^-(\sA)$, and $\Hot(\sA)$ by their thick subcategories of
absolutely acyclic complexes (see~\cite[Section~A.1]{Pcosh}
or~\cite[Appendix~A]{Pmgm} for a further discussion).

 We will say that a full subcategory $\sE\subset\sA$ is
\emph{coresolving} if $\sE$ is closed under extensions and
the passages to the cokernels of admissible monomorphisms in $\sA$,
and every object of $\sA$ is the source of an admissible monomorphism
into an object of~$\sE$.
 This definition slightly differs from that in~\cite[Section~2]{St} in
that we do not require $\sE$ to be closed under direct summands
(cf.~\cite[Section~A.3]{Pcosh}).
 Obviously, any coresolving subcategory $\sE$ inherits an exact
category structure from the ambient exact category~$\sA$.

 Let $\sA$ be an exact category in which the functors of infinite
direct sum are everywhere defined and exact.
 A complex in $\sA$ is said to be \emph{coacyclic} if it belongs to
the minimal triangulated subcategory of the homotopy category
$\Hot(\sA)$ containing the totalizations of short exact sequences of
complexes in $\sA$ and closed under infinite direct sums.
 The \emph{coderived category} $\sD^\co(\sA)$ is defined as
the Verdier quotient category of $\Hot(\sA)$ by the thick subcategory
of coacyclic complexes (see~\cite[Section~2.1]{Psemi},
\cite[Section~A.1]{Pcosh}, or~\cite[Appendix~A]{Pmgm}).

 The conventional unbounded derived category $\sD(\sA)$ is naturally
a quotient category of $\sD^\co(\sA)$.
 A triangulated category $\sD'$ is called a \emph{pseudo-coderived}
category of $\sA$ if triangulated Verdier quotient functors
$\sD^\co(\sA)\rarrow\sD'\rarrow\sD(\sA)$ are given forming
a commutative triangle with the canonical Verdier quotient functor
$\sD^\co(\sA)\rarrow\sD(\sA)$ between the coderived and
the conventional derived category of the exact category~$\sA$.

 Let $\sE\subset\sA$ be a coresolving subcategory closed under
infinite direct sums.
 According to the dual version of~\cite[Proposition~A.3.1(b)]{Pcosh}
(formulated explicitly in~\cite[Proposition~2.1]{Pfp}),
the triangulated functor between the coderived categories
$\sD^\co(\sE)\rarrow\sD^\co(\sA)$ induced by the direct sum-preserving
embedding of exact categories $\sE\rarrow\sA$ is an equivalence of
triangulated categories.
 From the commutative diagram of triangulated functors
$$
\begin{diagram}
\node{\sD^\co(\sE)}\arrow{s,A}\arrow[2]{e,=}
\node[2]{\sD^\co(\sA)}\arrow{s,A} \\
\node{\sD(\sE)}\arrow[2]{e}
\node[2]{\sD(\sA)}
\end{diagram} 
$$
one can see that the lower horizontal arrow is a Verdier quotient
functor.
 Thus $\sD'=\sD(\sE)$ is a pseudo-coderived category of~$\sA$
\cite[Proposition~4.2(a)]{PS}.

 Furthermore, let $\sE_\prime\subset\sE'\subset\sA$ be two embedded
coresolving subcategories, both closed under infinite direct sums
in~$\sA$.
 Then the canonical Verdier quotient functor $\sD^\co(\sA)\rarrow
\sD(\sA)$ decomposes into a sequence of Verdier quotient functors
$$
 \sD^\co(\sA)\lrarrow\sD(\sE_\prime)\lrarrow\sD(\sE')\lrarrow\sD(\sA).
$$
 In other words, when the full subcategory $\sE\subset\sA$ is
expanded, the related pseudo-coderived category $\sD(\sE)$ gets
deflated \cite[Remark~4.3]{PS}.

 Notice that, as a coresolving subcategory closed under infinite
direct sums $\sE\subset\sA$ varies, its conventional derived category
behaves in quite different ways depending on the boundedness
conditions.
 The functor $\sD^\b(\sE_\prime)\rarrow\sD^\b(\sE')$ induced by
the embedding $\sE_\prime\rarrow\sE'$ is fully faithful, the functor
$\sD^+(\sE_\prime)\rarrow\sD^+(\sE')$ is a triangulated equivalence
(by the assertion dual to~\cite[Proposition~A.3.1(a)]{Pcosh}),
and the functor $\sD(\sE_\prime)\rarrow\sD(\sE')$ is a Verdier
quotient functor.

\medskip

 Let $\sB$ be another exact category.
 We will say that a full subcategory $\sF\subset\sB$ is \emph{resolving}
if $\sF$ is closed under extensions and the passages to the kernels of
admissible epimorphisms, and every object of $\sB$ is the target of
an admissible epimorphism from an object of~$\sF$.
 Obviously, a resolving subcategory $\sF$ inherits an exact category
structure from the ambient exact category~$\sB$.

 Let $\sB$ be an exact category in which the functors of infinite
product are everywhere defined and exact.
 A complex in $\sB$ is said to be \emph{contraacyclic} if it belongs
to the minimal triangulated subcategory of $\Hot(\sB)$ containing
the totalizations of short exact sequences of complexes in $\sB$ and
closed under infinite products.
 The \emph{contraderived category} $\sD^\ctr(\sB)$ is defined as
the Verdier quotient category of $\Hot(\sB)$ by the thick subcategory
of contraacyclic complexes (see~\cite[Section~4.1]{Psemi},
\cite[Section~A.1]{Pcosh}, or~\cite[Appendix~A]{Pmgm}).

 The conventional unbounded derived category $\sD(\sB)$ is naturally
a quotient category of $\sD^\ctr(\sB)$.
 A triangulated category $\sD''$ is called a \emph{pseudo-contraderived}
category of $\sB$ if Verdier quotient functors $\sD^\ctr(\sB)\rarrow
\sD''\rarrow\sD(\sB)$ are given forming a commutative triangle with
the canonical Verdier quotient functor $\sD^\ctr(\sB)\rarrow\sD(\sB)$
between the contraderived and the conventional derived categories
of the exact category~$\sB$.

 Let $\sF\subset\sB$ be a resolving subcategory closed under infinite
products.
 According to~\cite[Proposition~A.3.1(b)]{Pcosh}, the triangulated
functor between the contraderived categories $\sD^\ctr(\sF)\rarrow
\sD^\ctr(\sB)$ induced by the product-preserving embedding of exact
categories $\sF\rarrow\sB$ is an equivalence of triangulated categories.
 From the commutative diagram of triangulated functors
$$
\begin{diagram}
\node{\sD^\ctr(\sF)}\arrow{s,A}\arrow[2]{e,=}
\node[2]{\sD^\ctr(\sB)}\arrow{s,A} \\
\node{\sD(\sF)}\arrow[2]{e}
\node[2]{\sD(\sB)}
\end{diagram} 
$$
one can see that the lower horizontal arrow is a Verdier quotient
functor.
 Thus $\sD''=\sD(\sF)$ is a pseudo-contraderived category of~$\sB$
\cite[Proposition~4.2(b)]{PS}.

 Let $\sF_{\prime\prime}\subset\sF''\subset\sB$ be two embedded resolving
subcategories, both closed under infinite products in~$\sF$.
 Then the canonical Verdier quotient functor $\sD^\ctr(\sB)\rarrow
\sD(\sB)$ decomposes into a sequence of Verdier quotient functors
$$
 \sD^\ctr(\sB)\lrarrow\sD(\sF_{\prime\prime})\lrarrow\sD(\sF'')\lrarrow
 \sD(\sB).
$$
 In other words, when the full subcategory $\sF\subset\sB$ is expanded,
the related pseudo-contraderived category $\sD(\sF)$ gets deflated
\cite[Remark~4.3]{PS}.

 Once again, we notice that, as a resolving subcategory closed under
infinite products $\sF\subset\sB$ varies, the behavior of its
conventional derived category depends on the boundedness conditions.
 The functor $\sD^\b(\sF_{\prime\prime})\rarrow\sD^\b(\sF'')$ is fully
faithful, the functor $\sD^-(\sF_{\prime\prime})\rarrow\sD^-(\sF'')$
is a triangulated equivalence~\cite[Proposition~A.3.1(a)]{Pcosh}, and
the functor $\sD(\sF_{\prime\prime})\rarrow\sD(\sF'')$ is a Verdier
quotient functor.

\medskip

 Some of the simplest examples of coresolving subcategories $\sE$
closed under infinite direct sums and resolving subcategories $\sF$
closed under infinite products in the abelian categories of modules
over associative rings will be given in
Examples~\ref{pseudo-coderived-examples}\+-%
\ref{pseudo-contraderived-examples}; and more complicated examples
will be discussed in Sections~\ref{two-rings-auslander-bass-subsecn},
\ref{minimal-classes-secn}, and~\ref{two-rings-base-change-subsecn}.

\Section{Strongly Finitely Presented Modules}
\label{sfp-modules-subsecn}

 Let $A$ be an associative ring.
 We denote by $A\modl$ the abelian category of left $A$\+modules
and by $\modr A$ the abelian category of right $A$\+modules.

 We say that an $A$\+module is \emph{strongly finitely presented}
if it has a projective resolution consisting of finitely
generated projective $A$\+modules.

\begin{rem}
 In the traditional terminology, such modules are said to be
\emph{pseudo-coherent}~\cite{Il,SP} or
\emph{of type FP$_\infty$} \cite{Bier,BGH,BP}.
 More generally, an $A$\+module $M$ is said to be 
\emph{$n$\+pseudo-coherent} or \emph{of type $FP_n$} if it has
a projective resolution $\dotsb\rarrow P_2\rarrow P_1\rarrow P_0
\rarrow M\rarrow 0$ such that the $A$\+modules $P_i$ are
finitely generated for all $0\le i\le n$.
 We prefer our (admittedly less conventional) terminology for
several reasons, one of them being that we would like to speak also
about strongly finitely presented \emph{complexes} of modules
(see below in this section), and the term ``chain complex of type
FP$_\infty$'' already means something else~\cite[Section~2.1]{BG}.
 On the other hand, the prefix ``pseudo-'' is used for quite
different purposes in this paper.
 Besides, ``strongly finitely presented'' abbreviates to
a convenient prefix ``sfp\+'', upon which we build our
terminological system.
\end{rem}

\begin{lem} \label{strongly-finitely-presented-modules}
 Let\/ $0\rarrow K\rarrow L\rarrow M\rarrow0$ be a short exact
sequence of $A$\+modules.
 Then whenever two of the three modules $K$, $L$, $M$ are strongly
finitely presented, so is the third one.
\end{lem}

\begin{proof}
 This result, in a more general version for modules of type FP$_{n-1}$
and FP$_n$, goes back to~\cite[Exercise~I.2.6]{Bour}.
 For a different proof, see~\cite[Proposition~1.4]{Bier}, and for
a discussion with further references, \cite[Section~1]{BP}.
 Here is yet another proof.

 If $P_\bu\rarrow K$ and $R_\bu\rarrow M$ are projective resolutions of
the $A$\+modules $K$ and $M$, then there is a projective resolution
$Q_\bu\rarrow L$ of the $A$\+module $L$ with the terms
$Q_i\simeq P_i\oplus R_i$.
 If $P_\bu\rarrow K$ and $Q_\bu\rarrow L$ are projective resolutions of
the $A$\+modules $K$ and $L$, then there exists a morphism of complexes
of $A$\+modules $P_\bu\rarrow Q_\bu$ inducing the given morphism
$K\rarrow L$ on the homology modules.
 The cone $R_\bu$ of the morphism of complexes $P_\bu\rarrow Q_\bu$ is
a projective resolution of the $A$\+module $M$ with the terms
$R_i\simeq Q_i\oplus P_{i-1}$.

 If $Q_\bu\rarrow L$ and $R_\bu\rarrow M$ are projective resolutions of
the $A$\+modules $L$ and $M$, then there exists a morphism of complexes
of $A$\+modules $Q_\bu\rarrow R_\bu$ inducing the given morphism
$L\rarrow M$ on the homology modules.
 The cocone $P'_\bu$ of the morphism of complexes $Q_\bu\rarrow R_\bu$
is a bounded above complex of $R$\+modules with the terms
$P'_i=Q_i\oplus R_{i+1}$ and the only nonzero cohomology module
$H_0(P'_\bu)\simeq K$.
 Still, the complex $P'_\bu$ is not yet literally a projective
resolution of $K$, as its term $P'_{-1}\simeq R_0$ does not vanish.
 Setting $P_{-1}=0$, \ $P_0=\ker(P'_0\to P'_{-1})$, and $P'_i=P_i$
for $i\ge 1$, one obtains a subcomplex $P_\bu\subset P'_\bu$ with
a termwise split embedding $P_\bu\rarrow P'_\bu$ such that $P_\bu$
is a projective resolution of the $R$\+module~$K$.
\end{proof}

 Abusing terminology, we will say that a bounded above complex of
$A$\+modules $M^\bu$ is \emph{strongly finitely presented} if it is
quasi-isomorphic to a bounded above complex of finitely generated
projective $A$\+modules.
 (Such complexes are called ``pseudo-coherent'' in~\cite{Il,SP}.)
 Clearly, the class of all strongly finitely presented complexes
is closed under shifts and cones in $\sD^-(A\modl)$.

\begin{lem} \label{sfp-complexes-lemma}
\textup{(a)} Any bounded above complex of strongly finitely
presented $A$\+modules is strongly finitely presented. \par
\textup{(b)} Let $M^\bu$ be a complex of $A$\+modules concentrated in
the cohomological degrees~$\le\nobreak n$, where $n$~is a fixed integer.
 Then $M^\bu$ is strongly finitely presented if and only if it is
quasi-isomorphic to a complex of finitely generated projective
$A$\+modules concentrated in the cohomological
degrees~$\le\nobreak n$. \par
\textup{(c)} Let $M^\bu$ be a finite complex of $A$\+modules
concentrated in the cohomological degrees $n_1\le m\le n_2$.
 Then $M^\bu$ is strongly finitely presented if and only if it is
quasi-isomorphic to a complex of $A$\+modules $R^\bu$ concentrated
in the cohomological degrees $n_1\le m\le n_2$ such that
the $A$\+modules $R^m$ are finitely generated and projective for
all $n_1+1\le m\le n_2$, while the $A$\+module $R^{n_1}$ is strongly
finitely presented.
\end{lem}

\begin{proof}
 Part~(a) is provable using
Lemma~\ref{strongly-finitely-presented-modules}.
 Part~(b) holds, because the kernel of a surjective morphism of
finitely generated projective $A$\+modules is a finitely generated
projective $A$\+module; and part~(c) is also easy.
\end{proof}

 Let $A$ and $B$ be associative rings.
 A left $A$\+module $J$ is said to be \emph{sfp\+injective} if
$\Ext^1_A(M,J)=0$ for all strongly finitely presented left $A$\+modules
$M$, or equivalently, $\Ext^n_A(M,J)=0$ for all strongly finitely
presented left $A$\+modules $M$ and all $n>0$.
 A left $B$\+module $P$ is said to be \emph{sfp\+flat} if
$\Tor^B_1(N,P)=0$ for all strongly finitely presented right
$B$\+modules $N$, or equivalently, $\Tor^B_n(N,P)=0$ for all strongly
finitely presented right $B$\+modules $N$ and all $n>0$.

 What we call ``sfp\+injective modules'' are otherwise known as
``FP$_\infty$\+injective'', and what we call ``sfp\+flat'' modules
would be usually called ``FP$_\infty$\+flat'' \cite[Section~3]{BP}.
 In the terminology of the papers~\cite[Section~2]{BGH} and~\cite{BG},
the former modules are also called ``absolutely clean'', and the latter
ones are called ``level''.

\begin{lem} \label{sfp-injective-flat-modules-are-closed-under}
\textup{(a)} The class of all sfp\+injective left $A$\+modules is
closed under extensions, the cokernels of injective morphisms, filtered
inductive limits, infinite direct sums, and infinite products. \par
\textup{(b)} The class of all sfp\+flat left $B$\+modules is
closed under extensions, the kernels of surjective morphisms, filtered
inductive limits, infinite direct sums, and infinite products.
\end{lem}

\begin{proof}
 This is~\cite[Propositions~2.7 and~2.10]{BGH}.
 For a generalization to FP$_n$\+injective and FP$_n$\+flat modules,
$n\ge2$, see~\cite[Propositions~3.10 and~3.11]{BP}.
\end{proof}

\begin{exs} \label{pseudo-coderived-examples}
 (1)~The following construction using strongly finitely presented
modules provides some examples of pseudo-coderived categories of
modules over an associative ring in the sense of
Section~\ref{pseudo-derived-introd}.
 Let $A$ be an associative ring and $\sS$ be a set of strongly
finitely presented left $A$\+modules.
 Denote by $\sE\subset\sA=A\modl$ the full subcategory formed by all
the left $A$\+modules $E$ such that $\Ext^i_A(S,E)=0$ for all
$S\in\sS$ and all $i>0$.
 Then the full subcategory $\sE\subset A\modl$ is a coresolving
subcategory closed under infinite direct sums (and products).
 So the induced triangulated functor between the two coderived
categories $\sD^\co(\sE)\rarrow\sD^\co(A\modl)$ is a triangulated
equivalence by the dual version of~\cite[Proposition~A.3.1(b)]{Pcosh}
(cf.~\cite[Proposition~2.1]{Pfp}).
 Thus the derived category $\sD(\sE)$ of the exact category $\sE$ is
a pseudo-coderived category of the abelian category $A\modl$, that is
an intermediate quotient category between the coderived category
$\sD^\co(A\modl)$ and the derived category $\sD(A\modl)$.

\smallskip

 (2)~In particular, if $\sS=\varnothing$, then one has $\sE=A\modl$.
 On the other hand, if $\sS$ is the set of all strongly finitely
presented left $A$\+modules, then the full subcategory
$\sE\subset A\modl$ consists of all the sfp\+injective modules.
 When the ring $A$ is left coherent, all the finitely presented left
$A$\+modules are strongly finitely presented, and objects of the class
$\sE$ are called \emph{fp\+injective} left $A$\+modules.
 In this case, the derived category $\sD(\sE)$ of the exact category
$\sE$ is equivalent to the homotopy $\Hot(A\modl_\inj)$ of the additive
category of injective left $A$\+modules~\cite[Theorem~6.12]{St2}.

\smallskip

 (3)~More generally, for any associative ring $A$, the category
$\Hot(A\modl_\inj)$ can be called the \emph{coderived category in
the sense of Becker}~\cite{Bec} of the category of left $A$\+modules.
 A complex of left $A$\+modules $X^\bu$ is called \emph{coacyclic in
the sense of Becker} if the complex of abelian groups
$\Hom_A(X^\bu,J^\bu)$ is acyclic for any complex of injective left
$A$\+modules~$J^\bu$.
 According to~\cite[Proposition~1.3.6(2)]{Bec}, the full subcategories
of complexes of injective modules and coacyclic complexes in
the sense of Becker form a semiorthogonal decomposition of
the homotopy category of left $A$\+modules $\Hot(A\modl)$.
 According to~\cite[Theorem~3.5(a)]{Pkoszul}, any coacyclic complex of
left $A$\+modules in the sense of~\cite[Section~2.1]{Psemi},
\cite[Appendix~A]{Pmgm} is also coacyclic in the sense of Becker.
 Thus $\Hot(A\modl_\inj)$ occurs as an intermediate triangulated
quotient category between $\sD^\co(A\modl)$ and $\sD(A\modl)$.
 So the coderived category in the sense of Becker is
a pseudo-coderived category in our sense.  {\hfuzz=2.2pt\par}

 We do \emph{not} know whether the Verdier quotient functor
$\sD^\co(A\modl)\rarrow\Hot(A\modl_\inj)$ is a triangulated equivalence
(or, which is the same, the natural fully faithful triangulated functor
$\Hot(A\modl_\inj)\rarrow\sD^\co(A\modl)$ is a triangulated equivalence)
for an arbitrary associative ring~$A$.
 Partial results in this direction are provided
by~\cite[Theorem~3.7]{Pkoszul}
and~\cite[Theorem~2.4]{Pfp} (see also
Proposition~\ref{sfp-inj-homotopy-coderived-comparison} below).
\hbadness=1600
\end{exs}

\begin{exs} \label{pseudo-contraderived-examples}
 (1)~The following dual version of
Example~\ref{pseudo-coderived-examples}(1) provides some examples
of pseudo-contraderived categories of modules.
 Let $B$ be an associative ring and $\sS$ be a set of strongly
finitely presented right $B$\+modules.
 Denote by $\sF\subset\sB=B\modl$ the full subcategory formed by all
the left $B$\+modules $F$ such that $\Tor^B_i(S,F)=0$ for all
$S\in\sS$ and $i>0$.
 Then the full subcategory $\sF\subset B\modl$ is a resolving
subcategory closed under infinite products (and direct sums).
 So the induced triangulated functor between the two contraderived
categories $\sD^\ctr(\sF)\rarrow\sD^\ctr(B\modl)$ is a triangulated
equivalence by~\cite[Proposition~A.3.1(b)]{Pcosh}.
 Thus the derived category $\sD(\sF)$ of the exact category $\sF$ is
a pseudo-contraderived category of the abelian category $B\modl$,
that is an intermediate quotient category between the contraderived
category $\sD^\ctr(B\modl)$ and the derived category $\sD(B\modl)$,
as it was explained in Section~\ref{pseudo-derived-introd}.

\smallskip

 (2)~In particular, if $\sS=\varnothing$, then one has $\sF=B\modl$.
 On the other hand, if $\sS$ is the set of all strongly finitely
presented right $B$\+modules, then the full subcategory
$\sF\subset B\modl$ consists of all the sfp\+flat modules.
 When the ring $B$ is right coherent, all the sfp\+flat left
$B$\+modules are flat and $\sF$ is the full subcategory of all flat
left $B$\+modules.
 For any associative ring $B$, the derived category of the exact
category of flat left $B$\+modules is equivalent to the homotopy
category $\Hot(B\modl_\proj)$ of the additive category of projective
left $B$\+modules~\cite[Proposition~8.1 and Theorem~8.6]{Neem}.

\smallskip

 (3)~For any associative ring $B$, the category $\Hot(B\modl_\proj)$
can be called the \emph{contraderived category in the sense of Becker}
of the category of left $B$\+modules.
  A complex of left $B$\+modules $Y^\bu$ is called \emph{contraacyclic
in the sense of Becker} if the complex of abelian groups
$\Hom_B(P^\bu,Y^\bu)$ is acyclic for any complex of projective left
$B$\+modules~$P^\bu$.
 According to~\cite[Proposition~1.3.6(1)]{Bec}, the full subcategories
of contraacyclic complexes in the sense of Becker and complexes of
projective modules form a semiorthogonal decomposition of
the homotopy category of left $B$\+modules $\Hot(B\modl)$.
 According to~\cite[Theorem~3.5(b)]{Pkoszul}, any contraacyclic complex
of left $B$\+modules in the sense of~\cite[Section~4.1]{Psemi},
\cite[Appendix~A]{Pmgm} is also contraacyclic in the sense of Becker.
 Thus $\Hot(B\modl_\proj)$ occurs as an intermediate triangulated
quotient category between $\sD^\ctr(B\modl)$ and $\sD(B\modl)$.
 So the contraderived category in the sense of Becker is
a pseudo-contraderived category in our sense.

 We do \emph{not} know whether the Verdier quotient functor
$\sD^\ctr(B\modl)\rarrow\Hot(B\modl_\proj)$ is a triangulated
equivalence (or, which is the same, the natural fully faithful
triangulated functor $\Hot(B\modl_\proj)\rarrow\sD^\ctr(B\modl)$ is
a triangulated equivalence) for an arbitrary associative ring~$B$.
 A partial result in this direction is provided
by~\cite[Theorem~3.8]{Pkoszul} (cf.~\cite[Theorem~4.4]{Pfp};
see also Proposition~\ref{sfp-flat-homotopy-contraderived-comparison}
below).  \hbadness=1375
\end{exs}

\Section{Auslander and Bass Classes}
\label{two-rings-auslander-bass-subsecn}

 We recall the definition of a pseudo-dualizing complex of bimodules
from Section~\ref{two-associative-rings-introd} of the Introduction.
 Let $A$ and $B$ be associative rings.

 A \emph{pseudo-dualizing complex} $L^\bu$ for the rings $A$ and $B$
is a finite complex of $A$\+$B$\+bimodules satisfying the following
two conditions:
\begin{enumerate}
\renewcommand{\theenumi}{\roman{enumi}}
\setcounter{enumi}{1}
\item the complex $L^\bu$ is strongly finitely presented as
a complex of left $A$\+modules and as a complex of right $B$\+modules;
\item the homothety maps $A\rarrow\Hom_{\sD^\b(\modr B)}(L^\bu,L^\bu[*])$ 
and $B^\rop\rarrow\Hom_{\sD^\b(A\modl)}(L^\bu,L^\bu[*])$ are
isomorphisms of graded rings. \emergencystretch=3em\hbadness=2900
\end{enumerate}
 Here the condition~(ii) refers to the definition of a strongly
finitely presented complex of modules in
Section~\ref{sfp-modules-subsecn}.
 The complex $L^\bu$ is viewed as an object of the bounded derived
category of $A$\+$B$\+bimodules $\sD^\b(A\bimod B)$.

 We will use the following simplified notation: given two complexes
of left $A$\+mod\-ules $M^\bu$ and $N^\bu$, we denote by
$\Ext^n_A(M^\bu,N^\bu)$ the groups $H^n\boR\Hom_A(M^\bu,N^\bu)=
\Hom_{\sD(A\modl)}(M^\bu,N^\bu[n])$.
 Given a complex of right $B$\+modules $N^\bu$ and a complex of
left $B$\+modules $M^\bu$, we denote by $\Tor_n^B(N^\bu,M^\bu)$
the groups $H^{-n}(N^\bu\ot_B^\boL M^\bu)$.

 The tensor product functor $L^\bu\ot_B{-}\:\Hot(B\modl)\rarrow
\Hot(A\modl)$ acting between the unbounded homotopy categories
of left $B$\+modules and left $A$\+modules is left adjoint to
the functor $\Hom_A(L^\bu,{-})\:\Hot(A\modl)\rarrow\Hot(B\modl)$.
 Using homotopy flat and homotopy injective resolutions of
the second arguments, one constructs the derived functors
$L^\bu\ot_B^\boL{-}\:\sD(B\modl)\rarrow\sD(A\modl)$ and
$\boR\Hom_A(L^\bu,{-})\:\sD(A\modl)\rarrow\sD(B\modl)$ acting
between the (conventional) unbounded derived categories of
left $A$\+modules and left $B$\+modules. 
 As always with the left and right derived functors (e.~g., in
the sense of Deligne~\cite[1.2.1\+-2]{Del}), the functor
$L^\bu\ot_B^\boL{-}$ is left adjoint to the functor
$\boR\Hom_A(L^\bu,{-})$ \cite[Lemma~8.3]{Psemi}.

 Suppose that the finite complex $L^\bu$ is situated in
the cohomological degrees $-d_1\le m\le d_2$.
 Then one has $\Ext_A^n(L^\bu,J)=0$ for all $n>d_1$ and all
sfp\+injective left $A$\+modules~$J$.
 Similarly, one has $\Tor_n^B(L^\bu,P)=0$ for all $n>d_1$ and all
sfp\+flat left $B$\+modules~$P$.
 Choose an integer $l_1\ge d_1$ and consider the following full
subcategories in the abelian categories of left $A$\+modules
and left $B$\+modules:
\begin{itemize}
\item $\sE_{l_1}=\sE_{l_1}(L^\bu)\subset A\modl$ is the full subcategory
consisting of all the $A$\+modules $E$ such that $\Ext^n_A(L^\bu,E)=0$
for all $n>l_1$ and the adjunction morphism $L^\bu\ot_B^\boL
\boR\Hom_A(L^\bu,E)\rarrow E$ is an isomorphism in $\sD^-(A\modl)$;
\item $\sF_{l_1}=\sF_{l_1}(L^\bu)\subset B\modl$ is the full subcategory
consisting of all the $B$\+modules $F$ such that $\Tor^B_n(L^\bu,F)=0$
for all $n>l_1$ and the adjunction morphism $F\rarrow\boR\Hom_A(L^\bu,\>
L^\bu\ot_B^\boL F)$ is an isomorphism in $\sD^+(B\modl)$.
\end{itemize}
 Clearly, for any $l_1''\ge l_1'\ge d_1$, one has $\sE_{l_1'}\subset
\sE_{l_1''}\subset A\modl$ and $\sF_{l_1'}\subset\sF_{l_1''}\subset
B\modl$.

 The category $\sF_{l_1}$ is our version of what is called
the \emph{Auslander class} in~\cite{Chr,FJ,EJLR,CFH,HW}, while
the category $\sE_{l_1}$ is our version of the \emph{Bass class}.
 The definition of such classes of modules goes back to
Foxby~\cite[Section~1]{Fo}.

 The next three Lemmas~\ref{two-rings-E-F-closed-kernels-cokernels}\+-%
\ref{two-rings-closed-sums-products} are our versions
of~\cite[Lemma~4.1, Proposition~4.2, and Theorem~6.2]{HW}.

\begin{lem} \label{two-rings-E-F-closed-kernels-cokernels}
\textup{(a)} The full subcategory\/ $\sE_{l_1}\subset A\modl$ is closed
under the cokernels of injective morphisms, extensions, and direct
summands. \par
\textup{(b)} The full subcategory\/ $\sF_{l_1}\subset B\modl$ is closed
under the kernels of surjective morphisms, extensions, and direct
summands. \qed
\end{lem}

\begin{lem} \label{two-rings-contains-injectives-flats}
\textup{(a)} The full subcategory\/ $\sE_{l_1}\subset A\modl$ contains
all the injective left $A$\+modules. \par
\textup{(b)} The full subcategory\/ $\sF_{l_1}\subset B\modl$ contains
all the flat left $B$\+modules.
\end{lem}

\begin{proof}
 Part~(a): let ${}'\!\.L^\bu$ be a bounded above complex of finitely
generated projective right $B$\+modules endowed with
a quasi-isomorphism of complexes of right $B$\+modules
${}'\!\.L^\bu\rarrow L^\bu$.
 Then the complex ${}'\!\.L^\bu\ot_B\Hom_A(L^\bu,J)$ computes
$L^\bu\ot_B^\boL\boR\Hom_A(L^\bu,J)$ as an object of the derived
category of abelian groups for any injective left $A$\+module~$J$.
 Now we have an isomorphism of complexes of abelian groups
${}'\!\.L^\bu\ot_B\Hom_A(L^\bu,J)\simeq
\Hom_A(\Hom_{B^\rop}({}'\!\.L^\bu,L^\bu),J)$ and a quasi-isomorphism of
complexes of left $A$\+modules $A\rarrow\Hom_{B^\rop}({}'\!\.L^\bu,
L^\bu)$, implying that the natural morphism $L^\bu\ot_B^\boL
\boR\Hom_A(L^\bu,J)\rarrow J$ is an isomorphism in the derived
category of abelian groups, hence also in the derived category
of left $A$\+modules.

 Part~(b): let ${}''\!\.L^\bu$ be a bounded above complex of finitely
generated projective left $A$\+modules endowed with
a quasi-isomorphism of complexes of left $A$\+modules
${}''\!\.L^\bu\rarrow L^\bu$.
 Then the complex $\Hom_A({}''\!\.L^\bu,\>L^\bu\ot_BP)$ represents
the derived category object $\boR\Hom_A(L^\bu,\>L^\bu\ot_B^\boL P)$
for any flat left $B$\+module~$P$.
 Now we have an isomorphism of complexes of abelian groups
$\Hom_A({}''\!\.L^\bu,\>L^\bu\ot_BP)\simeq\Hom_A({}''\!\.L^\bu,L^\bu)
\ot_BP$ and a quasi-isomorphism of complexes of right $B$\+modules
$B\rarrow\Hom_A({}''\!\.L^\bu,L^\bu)$.
\end{proof}

 If $L^\bu$ is finite complex of $A$\+$B$\+bimodules that are strongly
finitely presented as left $A$\+modules and as right $B$\+modules,
then the class $\sE_{l_1}$ contains also all the sfp\+injective left
$A$\+modules and the class $\sF_{l_1}$ contains all the sfp\+flat
left $B$\+modules.

\begin{lem} \label{two-rings-closed-sums-products}
\textup{(a)} The full subcategory\/ $\sE_{l_1}\subset A\modl$ is closed
under infinite direct sums and products. \par
\textup{(b)} The full subcategory\/ $\sF_{l_1}\subset B\modl$ is closed
under infinite direct sums and products.
\end{lem}

\begin{proof}
 The functor $\boR\Hom_A(L^\bu,{-})\:\sD(A\modl)\rarrow\sD(B\modl)$
preserves infinite direct sums of uniformly bounded below families
of complexes and infinite products of arbitrary families of
complexes.
 The functor $L^\bu\ot_B^\boL{-}\:\sD(B\modl)\rarrow\sD(A\modl)$
preserves infinite products of uniformly bounded above families of
complexes and infinite direct sums of arbitrary families of complexes.
 These observations imply both the assertions~(a) and~(b).
\end{proof}

 For a further result in the direction of the above three lemmas, see
Corollary~\ref{two-rings-closed-filtered-colimits} below (claiming
that the classes of modules $\sE_{l_1}$ and $\sF_{l_1}$ are also closed
under filtered inductive limits).

 The full subcategories $\sE_{l_1}\subset A\modl$ and $\sF_{l_1}\subset
B\modl$ inherit exact category structures from the abelian categories
$A\modl$ and $B\modl$.
 It follows from Lemmas~\ref{two-rings-E-F-closed-kernels-cokernels}
and~\ref{two-rings-contains-injectives-flats} that the induced
triangulated functors $\sD^\b(\sE_{l_1})\rarrow \sD^\b(A\modl)$ and
$\sD^\b(\sF_{l_1})\rarrow\sD^\b(B\modl)$ are fully faithful.
 The following lemma describes their essential images.

\begin{lem} \label{d-b-e-to-a-f-to-b-essential-images}
\textup{(a)} Let $M^\bu$ be a complex of left $A$\+modules concentrated
in the cohomological degrees $-n_1\le m\le n_2$.
 Then $M^\bu$ is quasi-isomorphic to a complex of left $A$\+modules
concentrated in the cohomological degrees $-n_1\le m\le n_2$ with
the terms belonging to the full subcategory\/ $\sE_{l_1}\subset A\modl$
if and only if\/ $\Ext^n_A(L^\bu,M^\bu)=0$ for $n>n_2+l_1$ and
the adjunction morphism $L^\bu\ot_B^\boL\boR\Hom_A(L^\bu,M^\bu)\rarrow
M^\bu$ is an isomorphism in\/ $\sD^-(A\modl)$. \par
\textup{(b)} Let $N^\bu$ be a complex of left $B$\+modules concentrated
in the cohomological degrees $-n_1\le m\le n_2$.
 Then $N^\bu$ is quasi-isomorphic to a complex of left $B$\+modules
concentrated in the cohomological degrees $-n_1\le m\le n_2$ with
the terms belonging to the full subcategory\/ $\sF_{l_1}\subset B\modl$
if and only if\/ $\Tor_n^B(L^\bu,N^\bu)=0$ for $n>n_1+l_1$ and
the adjunction morphism\/ $N^\bu\rarrow\boR\Hom_A(L^\bu,\>
L^\bu\ot_B^\boL N^\bu)$ is an isomorphism in\/ $\sD^+(B\modl)$.
\end{lem}

\begin{proof}
 Part~(a): the ``only if'' part is obvious.
 To prove the ``if'', replace $M^\bu$ by a quasi-isomorphic complex
${}'\!M^\bu$ concentrated in the same cohomological degrees
$-n_1\le m\le n_2$ with ${}'\!M^m\in A\modl_\inj$ for $-n_1\le m < n_2$,
and use Lemma~\ref{two-rings-contains-injectives-flats}(a) in order to
check that ${}'\!M^m\in\sE_{l_1}$ for all $-n_1\le m\le n_2$.
 Part~(b): to prove the ``if'', replace $N^\bu$ by a quasi-isomorphic
complex ${}'\!N^\bu$ concentrated in the same cohomological degrees
$-n_1\le m\le n_2$ with ${}'\!N^m\in B\modl_\proj$ for $-n_1<m\le n_2$,
and use Lemma~\ref{two-rings-contains-injectives-flats}(b) in order to
check that ${}'\!N^m\in\sF_{l_1}$ for all $-n_1\le m\le n_2$.
\end{proof}

 Thus the full subcategory $\sD^\b(\sE_{l_1})\subset\sD(A\modl)$
consists of all the complexes of left $A$\+modules $M^\bu$ with bounded
cohomology such that the complex $\boR\Hom_A(L^\bu,M^\bu)$ also has
bounded cohomology and the adjunction morphism
$L^\bu\ot_B^\boL\boR\Hom_A(L^\bu,M^\bu)\rarrow M^\bu$ is an isomorphism.
 Similarly, the full subcategory $\sD^\b(\sF_{l_1})\subset\sD(B\modl)$
consists of all the complexes of left $B$\+modules $N^\bu$ with
bounded cohomology such that the complex $L^\bu\ot_B^\boL N^\bu$ also
has bounded cohomology and the adjunction morphism
$N^\bu\rarrow\boR\Hom_A(L^\bu,\>L^\bu\ot_B^\boL N^\bu)$ is an isomorphism.
{\hbadness=1500\par}

 These full subcategories are usually called the \emph{derived Bass}
and \emph{Auslander classes}.
 As any pair of adjoint functors restricts to an equivalence between
the full subcategories of all objects whose adjunction morphisms are
isomorphisms~\cite[Theorem~1.1]{FJ}, the functors
$\boR\Hom_A(L^\bu,{-})$ and $L^\bu\ot_B^\boL{-}$ restrict to
a triangulated equivalence between the derived Bass and Auslander
classes~\cite[Theorem~3.2]{AF}, \cite[Theorem~4.6]{Chr}
\begin{equation} \label{bounded-derived-bass-auslander-equiv}
 \sD^\b(\sE_{l_1})\simeq\sD^\b(\sF_{l_1}).
\end{equation}

\begin{lem} \label{E-l-F-l-taken-into-each-other-up-to-finite}
\textup{(a)} For any $A$\+module $E\in\sE_{l_1}$, the object\/
$\boR\Hom_A(L^\bu,E)\in\sD^\b(B\modl)$ can be represented by
a complex of $B$\+modules concentrated in the cohomological degrees
$-d_2\le m\le l_1$ with the terms belonging to\/~$\sF_{l_1}$. \par
\textup{(b)} For any $B$\+module $F\in\sF_{l_1}$, the object\/
$L^\bu\ot_B^\boL F\in\sD^\b(A\modl)$ can be represented by
a complex of $A$\+modules concentrated in the cohomological degrees
$-l_1\le m\le d_2$ with the terms belonging to\/~$\sE_{l_1}$.
\end{lem}

\begin{proof}
 Part~(a) follows from
Lemma~\ref{d-b-e-to-a-f-to-b-essential-images}(b), as
the derived category object $L^\bu\ot_B^\boL\boR\Hom_A(L^\bu,E)
\simeq E$ has no cohomology in the cohomological degrees
$-n<-d_2-l_1\le-d_2-d_1\le0$.
 Part~(b) follows from
Lemma~\ref{d-b-e-to-a-f-to-b-essential-images}(a), as
the derived category object $\boR\Hom_A(L^\bu,\>L^\bu\ot_B^\boL F)
\simeq F$ has no cohomology in the cohomological degrees
$n>d_2+l_1\ge d_2+d_1\ge0$.
\end{proof}

 We refer to~\cite[Section~2]{St} and~\cite[Section~A.5]{Pcosh} for
discussions of the \emph{coresolution dimension} of objects of
an exact category $\sA$ with respect to its coresolving subcategory
$\sE$ and the \emph{resolution dimension} of objects of an exact
category $\sB$ with respect to its resolving subcategory $\sF$
(called the \emph{right\/ $\sE$\+homological dimension} and
the \emph{left\/ $\sF$\+homological dimension} in~\cite{Pcosh}).

\begin{lem} \label{two-rings-maximal-classes-co-resolution-dimension}
\textup{(a)} For any integers $l_1''\ge l_l'\ge d_1$, the full
subcategory\/ $\sE_{l_1''}\subset A\modl$ consists precisely of
all the left $A$\+modules whose $\sE_{l_1'}$\+coresolution dimension
does not exceed $l_1''-l_1'$. \par
\textup{(b)} For any integers $l_1''\ge l_1'\ge d_1$, the full
subcategory\/ $\sF_{l_1''}\subset B\modl$ consists precisely of
all the left $B$\+modules whose $\sF_{l_1'}$\+resolution dimension
does not exceed $l_1''-l_1'$.
\end{lem}

\begin{proof}
 Part~(a) is obtained by applying
Lemma~\ref{d-b-e-to-a-f-to-b-essential-images}(a) to
a one-term complex $M^\bu=E$, concentrated in the cohomological
degree~$0$, with $n_1=0$, \ $n_2=l_1''-l_1'$, and $l_1=l_1'$.
 Part~(b) is obtained by applying
Lemma~\ref{d-b-e-to-a-f-to-b-essential-images}(b) to
a one-term complex $N^\bu=F$, concentrated in the cohomological
degree~$0$, with $n_2=0$, \ $n_1=l_1''-l_1'$, and $l_1=l_1'$.
\end{proof}

\begin{rem} \label{finite-injective-flat-dim-bass-auslander-remark}
 In particular, it follows from
Lemmas~\ref{two-rings-contains-injectives-flats}
and~\ref{two-rings-maximal-classes-co-resolution-dimension} that,
for any $n\ge0$, all the left $A$\+modules of injective dimension
not exceeding~$n$ belong to $\sE_{d_1+n}$ and all the left $B$\+modules
of flat dimension not exceeding~$n$ belong to~$\sF_{d_1+n}$.
\end{rem}

 We refer to~\cite[Section~A.1]{Pcosh}, \cite[Appendix~A]{Pmgm},
or Section~\ref{pseudo-derived-introd} for the definitions of exotic
derived categories appearing in the next proposition.

\begin{prop} \label{two-rings-maximal-classes-essentially-the-same}
 For any $l_1''\ge l_1'\ge d_1$ and any conventional or exotic derived
category symbol\/ $\star=\b$, $+$, $-$, $\varnothing$, $\abs+$,
$\abs-$, $\co$, $\ctr$, or\/~$\abs$, the exact embedding functors\/
$\sE_{l_1'}\rarrow\sE_{l_1''}$ and\/ $\sF_{l_1'}\rarrow\sF_{l_1''}$ induce
triangulated equivalences
$$
 \sD^\star(\sE_{l_1'})\simeq\sD^\star(\sE_{l_1''})\quad\text{and}\quad
 \sD^\star(\sF_{l_1'})\simeq\sD^\star(\sF_{l_1''}).
$$
\end{prop}

\begin{proof}
 In view of~\cite[Proposition~A.5.6]{Pcosh}, the assertions follow
from Lemma~\ref{two-rings-maximal-classes-co-resolution-dimension}.
\end{proof}

 In particular, the unbounded derived category $\sD(\sE_{l_1})$ is
the same for all $l_1\ge d_1$ and the unbounded derived category
$\sD(\sF_{l_1})$ is the same for all $l_1\ge d_1$.

 As it was explained in Section~\ref{pseudo-derived-introd}, it follows
from Lemmas~\ref{two-rings-E-F-closed-kernels-cokernels}\+-%
\ref{two-rings-closed-sums-products} by virtue
of~\cite[Proposition~A.3.1(b)]{Pcosh} that the natural Verdier
quotient functor $\sD^\co(A\modl)\rarrow\sD(A\modl)$ factorizes
into two Verdier quotient functors
$$
 \sD^\co(A\modl)\lrarrow\sD(\sE_{l_1})\lrarrow\sD(A\modl),
$$
and the natural Verdier quotient functor $\sD^\ctr(B\modl)\rarrow
\sD(B\modl)$ factorizes into two Verdier quotient functors
$$
 \sD^\ctr(B\modl)\lrarrow\sD(\sF_{l_1})\lrarrow\sD(B\modl).
$$
 In other words, the triangulated category $\sD(\sE_{l_1})$ is
a pseudo-coderived category of the abelian category of left
$A$\+modules and the triangulated category $\sD(\sF_{l_1})$ is
a pseudo-contraderived category of the abelian category of
left $B$\+modules.

 These are called the \emph{lower pseudo-coderived category} of
left $A$\+modules and the \emph{lower pseudo-contraderived category} of
left $B$\+modules corresponding to the pseudo-dualizing complex~$L^\bu$.
 The notation is
$$
 \sD'_{L^\bu}(A\modl)=\sD(\sE_{l_1}) \quad\text{and}\quad
 \sD''_{L^\bu}(B\modl)=\sD(\sF_{l_1}).
$$
 The next theorem provides, in particular, a triangulated
equivalence between the lower pseudo-coderived and the lower
pseudo-contraderived category,
$$
 \sD'_{L^\bu}(A\modl)=\sD(\sE_{l_1})\.\simeq\.
 \sD(\sF_{l_1})=\sD''_{L^\bu}(B\modl). 
$$

\begin{thm} \label{two-rings-lower-main-thm}
 For any symbol\/ $\star=\b$, $+$, $-$, $\varnothing$, $\abs+$,
$\abs-$, $\co$, $\ctr$, or\/~$\abs$, there is a triangulated
equivalence\/ $\sD^\star(\sE_{l_1})\simeq\sD^\star(\sF_{l_1})$
provided by (appropriately defined) mutually inverse functors\/
$\boR\Hom_A(L^\bu,{-})$ and $L^\bu\ot_B^\boL{-}$.
\end{thm}

\begin{proof}
 This is a particular case of
Theorem~\ref{two-rings-generalized-main-thm} below.
 The construction of the derived functors $\boR\Hom_A(L^\bu,{-})$ and
$L^\bu\ot_B^\boL{-}$ is explained in the proof of
Theorem~\ref{two-rings-generalized-main-thm} and in the appendix.
\end{proof}

\Section{Abstract Corresponding Classes}
\label{two-rings-abstract-classes}

 More generally, suppose that we are given two full subcategories
$\sE\subset A\modl$ and $\sF\subset B\modl$ satisfying the following
conditions (for some fixed integers $l_1$ and~$l_2$):
\begin{enumerate}
\renewcommand{\theenumi}{\Roman{enumi}}
\item the full subcategory $\sE\subset A\modl$ is closed under
extensions and the cokernels of injective morphisms, and contains
all the injective left $A$\+modules;
\item the full subcategory $\sF\subset B\modl$ is closed under
extensions and the kernels of surjective morphisms, and contains
all the projective left $B$\+modules;
\item for any $A$\+module $E\in\sE$, the object $\boR\Hom_A(L^\bu,E)
\in\sD^+(B\modl)$ can be represented by a complex of $B$\+modules
concentrated in the cohomological degrees $-l_2\le m\le l_1$ with
the terms belonging to~$\sF$;
\item for any $B$\+module $F\in\sF$, the object $L^\bu\ot_B^\boL F
\in\sD^-(A\modl)$ can be represented by a complex of $A$\+modules
concentrated in the cohomological degrees $-l_1\le m\le l_2$ with
the terms belonging to~$\sE$.
\end{enumerate}

 One can see from the conditions~(I) and~(III), or~(II) and~(IV),
that $l_1\ge d_1$ and $l_2\ge d_2$ if $H^{-d_1}(L^\bu)\ne0
\ne H^{d_2}(L^\bu)$.
 According to
Lemmas~\ref{two-rings-E-F-closed-kernels-cokernels},
\ref{two-rings-contains-injectives-flats}, and
\ref{E-l-F-l-taken-into-each-other-up-to-finite}, the two classes
$\sE=\sE_{l_1}$ and $\sF=\sF_{l_1}$ satisfy the conditions~(I\+-IV)
with $l_2=d_2$.

 The following lemma, providing a kind of converse implication, can be
obtained as a byproduct of the proof of
Theorem~\ref{two-rings-generalized-main-thm} below (based on
the arguments in the appendix).
 It is somewhat counterintuitive, claiming that the adjunction
isomorphism conditions on the modules in the classes $\sE$ and
$\sF$, which were necessary in the context of the previous
Section~\ref{two-rings-auslander-bass-subsecn}, follow from
the conditions~(I\+-IV) in our present context.
 So we prefer to present a separate explicit proof. 

\begin{lem} \label{two-rings-adjunction-isomorphisms-follow-lemma}
\textup{(a)} For any $A$\+module $E\in\sE$, the adjunction morphism
$L^\bu\ot_B^\boL\boR\Hom_A(L^\bu,E)\rarrow E$ is an isomorphism
in $\sD^\b(A\modl)$. \par
\textup{(b)} For any $B$\+module $F\in\sF$, the adjunction morphism
$F\rarrow\boR\Hom_A(L^\bu,\>\allowbreak L^\bu\ot_B^\boL\nobreak F)$
is an isomorphism in $\sD^\b(B\modl)$.
\end{lem}

\begin{proof}
 We will prove part~(a); the proof of part~(b) is similar.
 Specifically, let $0\rarrow E\rarrow K^0\rarrow K^1\rarrow\dotsb$
be an exact sequence of left $A$\+modules with $E\in\sE$ and
$K^i\in\sE$ for all $i\ge0$.
 Suppose that the adjunction morphisms
$L^\bu\ot_B^\boL\boR\Hom_A(L^\bu,K^i)\rarrow K^i$ are isomorphisms
in $\sD^\b(A\modl)$ for all $i\ge0$.
 We will show that the adjunction morphism
$L^\bu\ot_B^\boL\boR\Hom_A(L^\bu,E)\rarrow E$ is also an isomorphism
in this case.
 As injective left $A$\+modules belong to~$\sE$ by the condition~(I),
the desired assertion will then follow from
Lemma~\ref{two-rings-contains-injectives-flats}(a).

 Let $Z^i$ denote the kernel of the differential $K^i\rarrow K^{i+1}$;
in particular, $Z^0=E$.
 The key observation is that, according to the condition~(I), one has
$Z^i\in\sE$ for all $i\ge0$.
 For every~$i\ge0$, choose a coresolution of the short exact
sequence $0\rarrow Z^i\rarrow K^i\rarrow Z^{i+1}$ by short exact
sequences of injective left $A$\+modules $0\rarrow Y^{i,j}\rarrow
J^{i,j}\rarrow Y^{i+1,j}\rarrow0$, \ $j\ge0$.
 Applying the functor $\Hom_A(L^\bu,{-})$ to the complexes of
left $A$\+modules $J^{i,\bu}$ and $Y^{i,\bu}$, we obtain short exact
sequences of complexes of left $B$\+modules
$0\rarrow G^{i,\bu}\rarrow F^{i,\bu}\rarrow G^{i+1,\bu}\rarrow0$,
where $G^{i,\bu}=\Hom_A(L^\bu,Y^{i,\bu})$ and $F^{i,\bu}=
\Hom_A(L^\bu,J^{i,\bu})$.
 According to the condition~(III), each complex $G^{i,\bu}$ and
$F^{i,\bu}$ is quasi-isomorphic to a complex of left $B$\+modules
concentrated in the cohomological degrees $-l_2\le m\le l_1$ with
the terms belonging to~$\sF$.

 For every $i\ge0$, choose complexes of projective left $B$\+modules
$Q^{i,\bu}$ and $P^{i,\bu}$, concentrated in the cohomological
degrees~$\le l_1$ and endowed with quasi-isomorphisms of complexes
of left $B$\+modules $Q^{i,\bu}\rarrow G^{i,\bu}$ and $P^{i,\bu}\rarrow
F^{i,\bu}$ so that there are short exact sequences of complexes
of left $B$\+modules $0\rarrow Q^{i,\bu}\rarrow P^{i,\bu}\rarrow
Q^{i+1,\bu}\rarrow0$ and the whole diagram is commutative.
 Applying the functor $L^\bu\ot_B{-}$ to the complexes of left
$B$\+modules $P^{i,\bu}$ and $Q^{i,\bu}$, we obtain short exact
sequences of complexes of left $A$\+modules $0\rarrow N^{i,\bu}
\rarrow M^{i,\bu}\rarrow N^{i+1,\bu}\rarrow0$, where
$N^{i,\bu}=L^\bu\ot_B Q^{i,\bu}$ and $M^{i,\bu}=L^\bu\ot_B P^{i,\bu}$.
 It follows from the condition~(IV) that each complex $M^{i,\bu}$
and $N^{i,\bu}$ is quasi-isomorphic to a complex of left $A$\+modules
concentrated in the cohomological degrees $-l_1-l_2\le m\le l_1+l_2$.
 In particular, the cohomology modules of the complexes
$M^{i,\bu}$ and $N^{i,\bu}$ are concentrated in the degrees
$-l_1-l_2\le m\le l_1+l_2$.

 Applying the functors of two-sided canonical truncation
$\tau_{\ge-l_1-l_2}\tau_{\le l_1+l_2}$ to the complexes $M^{i,\bu}$ and
$N^{i,\bu}$, we obtain short exact sequences $0\rarrow
{}'\!N^{i,\bu}\rarrow{}'\!M^{i,\bu}\rarrow{}'\!N^{i+1,\bu}\rarrow0$
of complexes whose terms are concentrated in the cohomological
degrees $-l_1-l_2\le m\le l_1+l_2$.
 Similarly, applying the functors of canonical truncation
$\tau_{\le l_1+l_2}$ to the complexes $J^{i,\bu}$ and $Y^{i,\bu}$, we
obtain short exact sequences
$0\rarrow{}'Y^{i,\bu}\rarrow{}'\!J^{i,\bu}\rarrow{}'Y^{i+1,\bu}\rarrow0$
of complexes whose terms are concentrated in the cohomological
degrees $0\le m\le l_1+l_2$.
 Now we have two morphisms of bicomplexes ${}'\!M^{i,j}\rarrow
{}'\!J^{i,j}$ and $K^i\rarrow{}'\!J^{i,j}$, which are both
quasi-isomorphisms of finite complexes along the grading~$j$
at every fixed degree~$i$, by assumption.
 Furthermore, we have two morphisms of bicomplexes
${}'\!N^{0,j}\rarrow{}'\!M^{i,j}$ and ${}'Y^{0,j}\rarrow{}'\!J^{i,j}$,
which are both quasi-isomorphisms along the grading~$i$ at
every fixed degree~$j$, by construction.
 We also have a quasi-isomorphism $E\rarrow{}'Y^{0,\bu}$.

 Passing to the total complexes, we see that the morphism of complexes
${}'\!N^{0,\bu}\rarrow{}'Y^{0,\bu}$ is a quasi-isomorphism, because
so are the morphisms ${}'\!N^{0,\bu}\rarrow{}'\!M^{\bu,\bu}$,
\ ${}'\!M^{\bu,\bu}\rarrow{}'\!J^{\bu,\bu}$, and
${}'Y^{0,\bu}\rarrow{}'\!J^{\bu,\bu}$ in a commutative square.
 This proves that the adjunction morphism
$L^\bu\ot_B^\boL\boR\Hom_A(L^\bu,E)\rarrow E$ is an isomorphism
in the derived category.
\end{proof}

 Assuming that $l_1\ge d_1$ and $l_2\ge d_2$, it is now clear from
the conditions~(III\+-IV) and
Lemma~\ref{two-rings-adjunction-isomorphisms-follow-lemma}
that one has $\sE\subset\sE_{l_1}$ and
$\sF\subset\sF_{l_1}$ for any two classes of objects $\sE\subset
A\modl$ and $\sF\subset B\modl$ satisfying~(I\+-IV).
 Furthermore, it follows from the conditions~(I\+-II) that
the triangulated functors $\sD^\b(\sE)\rarrow\sD^\b(A\modl)$ and
$\sD^\b(\sF)\rarrow\sD^\b(B\modl)$ induced by the exact embeddings
$\sE\rarrow A\modl$ and $\sF\rarrow B\modl$ are fully faithful.
 Hence so are the triangulated functors $\sD^\b(\sE)\rarrow
\sD^\b(\sE_{l_1})$ and $\sD^\b(\sF)\rarrow\sD^\b(\sF_{l_1})$.
 In view of the conditions~(III\+-IV), we can conclude that
equivalence~\eqref{bounded-derived-bass-auslander-equiv} restricts
to a triangulated equivalence
\begin{equation} \label{two-rings-bounded-derived-abstract-equiv}
 \sD^\b(\sE)\simeq\sD^\b(\sF).
\end{equation}

 The following theorem is the main result of this paper.

\begin{thm} \label{two-rings-generalized-main-thm}
 Let\/ $\sE\subset A\modl$ and $\sF\subset B\modl$ be a pair of
full subcategories of modules satisfying the conditions~(I\+-IV) for
a pseudo-dualizing complex of $A$\+$B$\+bimodules~$L^\bu$.
 Then for any symbol\/ $\star=\b$, $+$, $-$, $\varnothing$, $\abs+$,
$\abs-$, $\co$, $\ctr$, or\/~$\abs$, there is a triangulated
equivalence\/ $\sD^\star(\sE)\simeq\sD^\star(\sF)$ provided
by (appropriately defined) mutually inverse functors\/
$\boR\Hom_A(L^\bu,{-})$ and $L^\bu\ot_B^\boL{-}$.

 Here, in the case\/ $\star=\co$ it is assumed that the full
subcategories\/ $\sE\subset A\modl$ and\/ $\sF\subset B\modl$ are
closed under infinite direct sums, while in the case\/ $\star=\ctr$
it is assumed that these two full subcategories are closed under
infinite products.
\end{thm}

\begin{proof}
 The words ``appropriately defined'' here mean ``as defined or
constructed in the appendix''.
 Specifically, given a complex of left $A$\+modules $E^\bu$ with
the terms $E^i\in\sE$, we choose its termwise injective coresolution,
which is a bounded below complex of complexes of injective left
$A$\+modules $J^{\bu,\bu}$.
 Applying the functor $\Hom_A(L^\bu,{-})$ to every bounded below
complex of injective left $A$\+modules $J^{i,\bu}$ coresolving
the left $A$\+module $E^i$, we obtain a bicomplex of left
$B$\+modules $N^{\bu,\bu}$, which can be replaced by a quasi-isomorphic
(in the sense explained in the appendix, see~\eqref{R-Psi-C-D-b})
bounded complex of complexes of left $B$\+modules $F^{\bu,\bu}$ with
the terms $F^{i,j}\in\sF$.
 Totalizing the bicomplex $F^{\bu,\bu}$, we obtain the desired
complex of left $B$\+modules $\boR\Hom_A(L^\bu,E^\bu)$.

 Similarly, given a complex of left $B$\+modules $F^\bu$ with
the terms $F^i\in\sF$, we choose its termwise projective resolution,
which is a bounded above complex of complexes of projective left
$B$\+modules $P^{\bu,\bu}$.
 Applying the functor $L^\bu\ot_B{-}$ to every bounded above complex
of projective left $B$\+modules $P^{i,\bu}$ resolving the left
$B$\+module $F^i$, we obtain a bicomplex of left $A$\+modules
$M^{\bu,\bu}$, which can be replaced by a quasi-isomorphic bounded
complex of complexes of left $A$\+modules $E^{\bu,\bu}$ with
the terms $E^{i,j}\in\sE$.
 Totalizing the bicomplex $E^{\bu,\bu}$, we obtain the desired
complex of left $A$\+modules $L^\bu\ot_B^\boL\nobreak F^\bu$.

 Arguing more formally, the theorem is a straightforward application of
the results of the appendix.
 In the context of the latter, we set
\begin{align*}
 \sA=A\modl\,&\supset\,\sE\,\supset\,\sJ=A\modl_\inj \\
 \sB=B\modl\,&\supset\,\sF\,\supset\,\sP=B\modl_\proj.
\end{align*}
 Consider the adjoint pair of DG\+functors
\begin{alignat*}{2}
 \Psi=\Hom_A(L^\bu,{-})\:&\sC^+(\sJ)&&\lrarrow\sC^+(\sB) \\
 \Phi=L^\bu\ot_B{-}\:&\sC^-(\sP)&&\lrarrow\sC^-(\sA).
\end{alignat*}
 Then the conditions of Sections~\ref{posing-problem-appx}
and~\ref{dual-setting-appx} are satisfied, so the constructions of
Sections~\ref{derived-functor-construction-appx}\+-%
\ref{dual-setting-appx} provide the derived functors $\boR\Psi$ and
$\boL\Phi$.
 The arguments in Section~\ref{deriving-adjoints-appx} show
that these two derived functors are naturally adjoint to each other,
and the first assertion of
Theorem~\ref{appx-triangulated-equivalence-thm} explains how to
deduce the claim that they are mutually inverse triangulated
equivalences from the triangulated
equivalence~\eqref{two-rings-bounded-derived-abstract-equiv}.

 Alternatively, applying the second assertion of
Theorem~\ref{appx-triangulated-equivalence-thm} together with
Lemma~\ref{two-rings-contains-injectives-flats} allows to
reprove the triangulated
equivalence~\eqref{two-rings-bounded-derived-abstract-equiv}
rather than use it, thus obtaining another and more ``conceptual''
proof of Lemma~\ref{two-rings-adjunction-isomorphisms-follow-lemma}.
\end{proof}

 Now suppose that we have two pairs of full subcategories
$\sE_\prime\subset\sE'\subset A\modl$ and $\sF_{\prime\prime}\subset
\sF''\subset B\modl$ such that both the pairs
$(\sE_\prime,\sF_{\prime\prime})$ and $(\sE',\sF'')$ satisfy
the conditions~(I\+-IV), and both the full subcategories
$\sE_\prime$ and $\sE'$ are closed under infinite direct sums in
$A\modl$, while both the full subcategories $\sF_{\prime\prime}$
and $\sF''$ are closed under infinite products in $B\modl$.
 Then, in view of the discussion in Section~\ref{pseudo-derived-introd}
and according to Theorem~\ref{two-rings-generalized-main-thm} (for
$\star=\varnothing$), we have a diagram of triangulated functors
\begin{equation} \label{two-associative-rings-abstract-diagram}
\begin{diagram}
\node{\sD^\co(A\modl)}\arrow{s,A}
\node[2]{\sD^\ctr(B\modl)}\arrow{s,A} \\
\node{\sD(\sE_\prime)}\arrow{s,A}\arrow[2]{e,=}
\node[2]{\sD(\sF_{\prime\prime})}\arrow{s,A} \\
\node{\sD(\sE')}\arrow{s,A}\arrow[2]{e,=}
\node[2]{\sD(\sF'')}\arrow{s,A} \\
\node{\sD(A\modl)}\node[2]{\sD(B\modl)}
\end{diagram}
\end{equation}

 The vertical arrows are Verdier quotient functors, so both
the triangulated categories $\sD(\sE_\prime)$ and $\sD(\sE')$ are
pseudo-coderived categories of left $A$\+modules,
and both the triangulated categories $\sD(\sF_{\prime\prime})$ and
$\sD(\sF'')$ are pseudo-contraderived categories of left $B$\+modules.
 The horizontal double lines are triangulated equivalences.
 The inner square in
the diagram~\eqref{two-associative-rings-abstract-diagram} is
commutative, as one can see from the construction of the derived
functors in Theorem~\ref{two-rings-generalized-main-thm}.

 More generally, for any symbol $\star=\b$, $+$, $-$, $\varnothing$,
$\abs+$, $\abs-$, or~$\abs$, there is a commutative diagram
of triangulated functors and triangulated equivalences
\begin{equation} \label{abstract-star-triangulated-diagram}
\begin{diagram}
\node{\sD^\star(\sE_\prime)}\arrow{s}\arrow[2]{e,=}
\node[2]{\sD^\star(\sF_{\prime\prime})}\arrow{s} \\
\node{\sD^\star(\sE')}\arrow[2]{e,=}\node[2]{\sD^\star(\sF'')}
\end{diagram}
\end{equation}

 When all the four full subcategories $\sE_\prime$, $\sE'\subset
A\modl$ and $\sF_{\prime\prime}$, $\sF''\subset B\modl$ are closed
under infinite direct sums (respectively, infinite products),
there is also a commutative diagram of triangulated functors and
triangulated equivalences~\eqref{abstract-star-triangulated-diagram}
with $\star=\co$ (resp., $\star=\ctr$).

\Section{Minimal Corresponding Classes} \label{minimal-classes-secn}

 Let $A$ and $B$ be associative rings, and $L^\bu$ be a pseudo-dualizing
complex of $A$\+$B$\+bimodules.

\begin{prop} \label{two-rings-minimal-corresponding-classes-prop}
 Fix $l_1=d_1$ and $l_2\ge d_2$.
 Then there exists a unique minimal pair of full subcategories\/
$\sE^{l_2}=\sE^{l_2}(L^\bu)\subset A\modl$ and\/ $\sF^{l_2}=\sF^{l_2}(L^\bu)
\subset B\modl$ satisfying the conditions~(I\+-IV) together with
the additional requirements that\/ $\sE^{l_2}$ is closed under infinite
direct sums in $A\modl$ and\/ $\sF^{l_2}$ is closed under infinite
products in $B\modl$.
 For any pair of full subcategories\/ $\sE\subset A\modl$ and\/
$\sF\subset B\modl$ satisfying the conditions~(I\+-IV) such that\/
$\sE$ is closed under infinite direct sums in $A\modl$ and\/ $\sF$
is closed under infinite products in $B\modl$ one has\/
$\sE^{l_2}\subset\sE$ and\/ $\sF^{l_2}\subset\sF$.
\end{prop}

\begin{proof}
 The full subcategories $\sE^{l_2}\subset A\modl$ and $\sF^{l_2}\subset
B\modl$ are constructed simultaneously by a kind of generation process.
 By construction, for any full subcategories $\sE\subset A\modl$ and
$\sF\subset B\modl$ as above we will have $\sE^{l_2}\subset\sE$ and
$\sF^{l_2}\subset\sF$.
 In particular, the pair of full subcategories $\sE=\sE_{d_1}$ and
$\sF=\sF_{d_1}$ satisfies all the mentioned conditions, so we will have
$\sE^{l_2}\subset\sE_{d_1}$ and $\sF^{l_2}\subset\sF_{d_1}$.

 Firstly, all the injective left $A$\+modules belong to $\sE^{l_2}$ and
all the projective left $B$\+modules belong to $\sF^{l_2}$, as dictated
by the conditions~(I\+-II).
 Secondly, let $E$ be an $A$\+module belonging to~$\sE^{l_2}$.
 Then $E\in\sE_{d_1}$, so the derived category object
$\boR\Hom_A(L^\bu,E)\in\sD^\b(B\modl)$ has cohomology modules
concentrated in the degrees $-d_2\le m\le d_1$.
 Pick a complex of left $B$\+modules $F^\bu$ representing
$\boR\Hom_A(L^\bu,E)$ such that $F^\bu$ is concentrated in
the degrees $-l_2\le m\le d_1$ and the $B$\+modules $F^m$ are
projective for all $-l_2+1\le m\le d_1$.
 According to~\cite[Corollary~A.5.2]{Pcosh}, we have $F^{-l_2}\in\sF$.
 So we say that the $B$\+module $F^{-l_2}$ belongs to~$\sF^{l_2}$.
 Similarly, let $F$ be a $B$\+module belonging to~$\sF^{l_2}$.
 Then $F\in\sF_{d_1}$, so the derived category object
$L^\bu\ot_B^\boL F$ has cohomology modules concentrated in
the degrees $-d_1\le m\le d_2$.
 Pick a complex of left $A$\+modules $E^\bu$ representing
$L^\bu\ot_B^\boL F$ such that $E^\bu$ is concentrated in
the degrees $-d_1\le m\le l_2$ and the $A$\+modules $E^m$ are
injective for all $-d_1\le m\le l_2-1$.
 According to the dual version of~\cite[Corollary~A.5.2]{Pcosh},
we have $E^{l_2}\in\sE$.
 So we say that the $A$\+module $E^{l_2}$ belongs to~$\sE^{l_2}$.

 Thirdly and finally, we add to $\sE^{l_2}$ all the extensions,
cokernels of injective morphisms, and infinite direct sums of its
objects, and similarly add to $\sF^{l_2}$ all the extensions, kernels
of surjective morphisms, and infinite products of its objects.
 Then the second and third steps are repeated in transfinite
iterations, as it may be necessary, until all the modules that can be
obtained in this way have been added and the full subcategories
of all such modules $\sE^{l_2}\subset A\modl$ and $\sF^{l_2}\subset
B\modl$ have been formed.
\end{proof}

\begin{rem}
 It is clear from the construction in the proof of
Proposition~\ref{two-rings-minimal-corresponding-classes-prop}
that for any two values of the parameters $l_1\ge d_1$ and $l_2\ge d_2$,
and any two full subcategories $\sE\subset A\modl$ and
$\sF\subset B\modl$ satisfying the conditions~(I\+-IV) with
the parameters~$l_1$ and~$l_2$ such that $\sE$ is closed under infinite
direct sums in $A\modl$ and $\sF$ is closed under infinite products in
$B\modl$, one has $\sE^{l_2}\subset\sE$ and $\sF^{l_2}\subset\sF$.
\end{rem}

 Notice that the conditions~(III\+-IV) become weaker as
the parameter~$l_2$ increases.
 It follows that one has $\sE^{l_2}\supset\sE^{l_2+1}$ and
$\sF^{l_2}\supset\sF^{l_2+1}$ for all $l_2\ge d_2$.
 So the inclusion relations between our classes of modules
have the form
\begin{align*}
\dotsb\subset\sE^{d_2+2}\subset\sE^{d_2+1}\subset\sE^{d_2}&\subset
\sE_{d_1}\subset\sE_{d_1+1}\subset\sE_{d_1+2}\subset\dotsb\subset
A\modl, \\
\dotsb\subset\sF^{d_2+2}\subset\sF^{d_2+1}\subset\sF^{d_2}&\subset
\sF_{d_1}\subset\sF_{d_1+1}\subset\sF_{d_1+2}\subset\dotsb\subset
B\modl.
\end{align*}

\begin{lem} \label{two-rings-abstract-classes-co-resolution-dimension}
 Let $n\ge0$ and $l_1\ge d_1$, \,$l_2\ge d_2+n$ be some integers.
 Let\/ $\sE\subset A\modl$ and\/ $\sF\subset B\modl$ be a pair of
full subcategories satisfying the conditions (I\+-IV) with
the parameters~$l_1$ and~$l_2$.
 Denote by\/ $\sE(n)\subset A\modl$ the full subcategory of all
left $A$\+modules of\/ $\sE$\+coresolution dimension not exceeding~$n$
and by\/ $\sF(n)\subset B\modl$ the full subcategory of all left
$B$\+modules of\/ $\sF$\+resolution dimension not exceeding~$n$.
 Then the two classes of modules\/ $\sE(n)$ and\/ $\sF(n)$ satisfy
the conditions (I\+-IV) with the parameters~$l_1+\nobreak n$
and~$l_2-\nobreak n$.
\end{lem}

\begin{proof}
 According to~\cite[Proposition~2.3(2)]{St}
or~\cite[Lemma~A.5.4(a\+b)]{Pcosh} (and the assertions dual to these),
the full subcategories $\sE(n)\subset A\modl$ and $\sF(n)\subset
B\modl$ satisfy the conditions~(I\+-II). 
 Using~\cite[Corollary~A.5.5(b)]{Pcosh}, one shows that for any
$A$\+module $E\in\sE(n)$ the derived category object
$\boR\Hom_A(L^\bu,E)\in\sD^\b(B\modl)$ can be represented by
a complex concentrated in the cohomological degrees
$-l_2\le m\le l_1+n$ with the terms belonging to~$\sF$.
 Moreover, one has $\Ext^m_A(L^\bu,E)=0$ for all $m<-d_2$.
 It follows that $\boR\Hom_A(L^\bu,E)$ can be also represented by
a complex concentrated in the cohomological degrees $-l_2+n\le m\le
l_1+n$ with the terms belonging to~$\sF(n)$.
 Similarly one can show that for any $B$\+module $F\in\sF(n)$
the derived category object $L^\bu\ot_B^\boL F\in\sD^b(A\modl)$ can be
represented by a complex concentrated in the cohomological degrees
$-l_1-n\le m\le l_2$ with the terms belonging to~$\sE$.
 Moreover, one has $\Tor_{-m}^B(L^\bu,F)=0$ for all $m>d_2$.
 It follows that $L^\bu\ot_B^\boL F$ can be also represented by
a complex concentrated in the cohomological degrees $-l_1-n\le m
\le l_2-n$ with the terms belonging to~$\sE(n)$.
 This proves the conditions~(III\+-IV).
\end{proof}

\begin{prop}
 For any $l_2''\ge l_2'\ge d_2$ and any conventional or exotic
derived category symbol\/ $\star=\b$, $+$, $-$, $\varnothing$,
$\abs+$, $\abs-$, or\/~$\abs$, the exact embedding functors\/
$\sE^{l_2''}\rarrow\sE^{l_2'}$ and\/ $\sF^{l_2''}\rarrow\sF^{l_2'}$ induce
triangulated equivalences
$$
 \sD^\star(\sE^{l_2''})\simeq\sD^\star(\sE^{l_2'})\quad\text{and}\quad
 \sD^\star(\sF^{l_2''})\simeq\sD^\star(\sF^{l_2'}).
$$
 The same exact embeddings also induce triangulated equivalences\/
$$
 \sD^\co(\sE^{l_2''})\simeq\sD^\co(\sE^{l_2'})\quad\text{and}\quad
 \sD^\ctr(\sF^{l_2''})\simeq\sD^\ctr(\sF^{l_2'}).
$$
\end{prop}

\begin{proof}
 As in
Proposition~\ref{two-rings-maximal-classes-essentially-the-same},
we check that the $\sE^{l_2''}$\+coresolution dimension of any object
of $\sE^{l_2'}$ does not exceed $l_2''-l_2'$ and
the $\sF^{l_2''}$\+resolution dimension of any object of $\sF^{l_2'}$
does not exceed $l_2''-l_2'$.
 Indeed, according to
Lemma~\ref{two-rings-abstract-classes-co-resolution-dimension},
the pair of full subcategories $\sE^{l_2''}(l_2''-l_2')\subset A\modl$
and $\sF^{l_2''}(l_2''-l_2')\subset B\modl$ satisfies
the conditions~(I\+-IV) with the parameters $l_1=d_1+l_2''-l_2'$
and $l_2=l_2'$.
 Furthermore, since infinite direct sums are exact and the full
subcategory $\sE^{l_2''}$ is closed under infinite direct sums in
$A\modl$,  so is the full subcategory
$\sE^{l_2''}(l_2''-l_2')$.
 Since infinite products are exact and the full subcategory
$\sF^{l_2''}$ is closed under infinite products in $B\modl$, so is
the full subcategory $\sF^{l_2''}(l_2''-l_2')$.
 It follows that $\sE^{l_2'}\subset\sE^{l_2''}(l_2''-l_2')$ and
$\sF^{l_2'}\subset\sF^{l_2''}(l_2''-l_2')$.
\end{proof}

 In particular, the unbounded derived category $\sD(\sE^{l_2})$ is
the same for all $l_2\ge d_2$ and the unbounded derived category
$\sD(\sF^{l_2})$ is the same for all $l_2\ge d_2$.

 As it was explained in Section~\ref{pseudo-derived-introd}, it follows
from the condition~(I) together with the condition that $\sE^{l_2}$ is
closed under infinite direct sums in $A\modl$ that the natural Verdier
quotient functor $\sD^\co(A\modl)\rarrow\sD(A\modl)$ factorizes into two
Verdier quotient functors
$$
 \sD^\co(A\modl)\lrarrow\sD(\sE^{l_2})\lrarrow\sD(A\modl).
$$
 Similarly, it follows from the condition~(II) together with
the condition that $\sF^{l_2}$ is closed under infinite products in
$B\modl$ that the natural Verdier quotient functor
$\sD^\ctr(B\modl)\rarrow\sD(B\modl)$ factorizes into two Verdier
quotient functors
$$
 \sD^\ctr(B\modl)\lrarrow\sD(\sF^{l_2})\lrarrow\sD(B\modl).
$$
 In other words, the triangulated category $\sD(\sE^{l_2})$ is
a pseudo-coderived category of left $A$\+modules and
the triangulated category $\sD(\sF^{l_2})$ is a pseudo-contraderived
category of left $B$\+modules.

 These are called the \emph{upper pseudo-coderived category} of
left $A$\+modules and the \emph{upper pseudo-contraderived category} of
left $B$\+modules corresponding to the pseudo-dualizing complex~$L^\bu$. 
 The notation is
$$
 \sD_\prime^{L^\bu}(A\modl)=\sD(\sE^{l_2}) \quad\text{and}\quad
 \sD_{\prime\prime}^{L^\bu}(B\modl)=\sD(\sF^{l_2}).
$$
 The next theorem provides, in particular, a triangulated
equivalence between the upper pseudo-coderived and the upper
pseudo-contraderived category,
$$
 \sD_\prime^{L^\bu}(A\modl)=\sD(\sE^{l_2})\.\simeq\.
 \sD(\sF^{l_2})=\sD_{\prime\prime}^{L^\bu}(B\modl). 
$$

\begin{thm} \label{two-rings-upper-main-thm}
 For any symbol\/ $\star=\b$, $+$, $-$, $\varnothing$, $\abs+$,
$\abs-$, or\/~$\abs$, there is a triangulated equivalence\/
$\sD^\star(\sE^{l_2})\simeq\sD^\star(\sF^{l_2})$ provided by
(appropriately defined) mutually inverse functors\/
$\boR\Hom_A(L^\bu,{-})$ and $L^\bu\ot_B^\boL{-}$.
\end{thm}

\begin{proof}
 This is another particular case of
Theorem~\ref{two-rings-generalized-main-thm}.
 It is explained in the proof of that theorem what ``appropriately defined''
means.
\end{proof}

 Substituting $\sE'=\sE_{l_1}$, \ $\sE_\prime=\sE^{l_2}$, \
$\sF''=\sF_{l_1}$, and $\sF_{\prime\prime}=\sF^{l_2}$ (for some
$l_1\ge\nobreak d_1$ and $l_2\ge\nobreak d_2$) into the commutative
diagram of triangulated
functors~\eqref{two-associative-rings-abstract-diagram} 
from Section~\ref{two-rings-abstract-classes}, one obtains
the commutative diagram of triangulated
functors~\eqref{two-associative-rings-triangulated-diagram} 
promised in Section~\ref{two-associative-rings-introd}
of the Introduction.

\Section{Dedualizing Complexes}

 Let $A$ and $B$ be associative rings.
 A \emph{dedualizing complex} of $A$\+$B$\+bimodules $L^\bu=T^\bu$
is a pseudo-dualizing complex (according to the definition
in Section~\ref{two-rings-auslander-bass-subsecn}) satisfying
the following additional condition:
\begin{enumerate}
\renewcommand{\theenumi}{\roman{enumi}}
\item As a complex of left $A$\+modules, $T^\bu$ is quasi-isomorphic
to a finite complex of projective $A$\+modules, and as a complex of
right $B$\+modules, $T^\bu$ is quasi-isomorphic to a finite complex
of projective $B$\+modules.
\end{enumerate}

 Taken together, the conditions~(i) and~(ii) mean that, as a complex
of left $A$\+modules, $T^\bu$ is quasi-isomorphic to a finite complex
of finitely generated projective $A$\+modules, and as a complex of
right $B$\+modules, $T^\bu$ is quasi-isomorphic to a finite complex
of finitely generated projective $B$\+modules.
 In other words, $T^\bu$ is a perfect complex of left $A$\+modules and
a perfect complex of right $B$\+modules.

 This definition of a dedualizing complex is slightly less general
than that of a \emph{tilting complex} in the sense
of~\cite[Theorem~1.1]{Ric2} and slightly more general than that of
a \emph{two-sided tilting complex} in the sense
of~\cite[Definition~3.4]{Ric2}.

 Let $L^\bu=T^\bu$ be a dedualizing complex of $A$\+$B$\+bimodules.
 We refer to the beginning of
Section~\ref{two-rings-auslander-bass-subsecn}
for the discussion of the pair of adjoint derived functors
$\boR\Hom_A(T^\bu,{-})\:\sD(A\modl)\rarrow\sD(B\modl)$ and
$T^\bu\ot_B^\boL{-}\:\sD(B\modl)\rarrow\sD(A\modl)$.
 The following assertion is a version of~\cite[Proposition~5.1]{Ric2}
and~\cite[Theorem~4.2\,(2)$\Leftrightarrow$(4)]{DS}.

\begin{prop}
 The derived functors\/ $\boR\Hom_A(T^\bu,{-})$ and $T^\bu\ot_B^\boL{-}$
are mutually inverse triangulated equivalences between the conventional
unbounded derived categories\/ $\sD(A\modl)$ and\/ $\sD(B\modl)$.
\end{prop}

\begin{proof}
 We have to show that the adjunction morphisms are isomorphisms.
 Let $J^\bu$ be a homotopy injective complex of left $A$\+modules.
 Then the complex of left $B$\+modules $\Hom_A(T^\bu,J^\bu)$
represents the derived category object $\boR\Hom_A(T^\bu,J^\bu)\in
\sD(B\modl)$.
 Let ${}'\.T^\bu$ be a finite complex of finitely generated projective
right $B$\+modules endowed with a quasi-isomorphism of complexes of
right $B$\+modules ${}'\.T^\bu\rarrow T^\bu$.
 Then the adjunction morphism $T^\bu\ot_B^\boL\boR\Hom_A(T^\bu,J^\bu)
\rarrow J^\bu$ is represented, as a morphism in the derived category of
abelian groups, by the morphism of complexes $'\.{}T^\bu\ot_B
\Hom_A(T^\bu,J^\bu)\rarrow J^\bu$.
 Now the complex of abelian groups ${}'\.T^\bu\ot_B\Hom_A(T^\bu,J^\bu)$
is naturally isomorphic to $\Hom_A(\Hom_{B^\rop}({}'\.T^\bu,T^\bu),
J^\bu)$, and the morphism of complexes of left $A$\+modules $A\rarrow
\Hom_{B^\rop}({}'\.T^\bu,T^\bu)$ is a quasi-isomorphism by
the condition~(iii).

 Similarly, let $P^\bu$ be a homotopy flat complex of left $B$\+modules.
 Then the complex of left $A$\+modules $T^\bu\ot_BP^\bu$ represents
the derived category object $T^\bu\ot_B^\boL P^\bu\in\sD(A\modl)$.
 Let ${}''\.T^\bu$ be a finite complex of finitely generated projective
left $A$\+modules endowed with a quasi-isomorphism of complexes of
left $A$\+modules ${}''T^\bu\rarrow T^\bu$.
 Then the adjunction morphism $P^\bu\rarrow\boR\Hom_A(T^\bu,\>
T^\bu\ot_B^\boL P^\bu)$ is represented, as a morphism in the derived
category of abelian groups, by the morphism of complexes $P^\bu\rarrow
\Hom_A({}''\.T^\bu,\>T^\bu\ot_BP^\bu)$.
 Now the complex of abelian groups $\Hom_A({}''\.T^\bu,\>
T^\bu\ot_BP^\bu)$ is naturally isomorphic to $\Hom_A({}''\.T^\bu,T^\bu)
\ot_BP^\bu$, and the morphism of complexes of right $B$\+modules
$B\rarrow\Hom_A({}''\.T^\bu,T^\bu)$ is a quasi-isomorphism by
the condition~(iii).
\end{proof}

 In particular, it follows that the derived Bass and Auslander classes
associated with a dedualizing complex $L^\bu=T^\bu$ (as discussed in
Section~\ref{two-rings-auslander-bass-subsecn}) coincide with the whole
bounded derived categories $\sD^\b(A\modl)$ and $\sD^\b(B\modl)$, and
the triangulated
equivalence~\eqref{bounded-derived-bass-auslander-equiv} takes the form
$\sD^\b(A\modl)\simeq\sD^\b(B\modl)$.

 Now let us choose the parameter~$l_1$ in such a way that $T^\bu$ is
quasi-isomorphic to a complex of (finitely generated) projective left
$A$\+modules concentrated in the cohomological degrees
$-l_1\le m\le d_2$ and to a complex of (finitely generated) projective
right $B$\+modules concentrated in the cohomological degrees
$-l_1\le m\le d_2$.
 Then we have $\sE_{l_1}(T^\bu)=A\modl$ and $\sF_{l_1}(T^\bu)=B\modl$.

 The next corollary is a (partial) extension
of~\cite[Theorem~6.4]{Ric}, \cite[Theorem~3.3 and Proposition~5.1]{Ric2},
and~\cite[Proposition~5.1]{DS}.

\begin{cor} \label{two-rings-dedualizing-cor}
 For any symbol\/ $\star=\b$, $+$, $-$, $\varnothing$, $\abs+$,
$\abs-$, $\co$, $\ctr$, or\/~$\abs$, there is a triangulated
equivalence\/ $\sD^\star(A\modl)\simeq\sD^\star(B\modl)$
provided by (appropriately defined) mutually inverse functors\/
$\boR\Hom_A(T^\bu,{-})$ and $T^\bu\ot_B^\boL{-}$.
\end{cor}

\begin{proof}
 In view of the above observations, this is a particular case of
Theorem~\ref{two-rings-lower-main-thm}.
\end{proof}

\Section{Dualizing Complexes} \label{two-rings-dualizing-subsecn}

 Let $A$ and $B$ be associative rings.
 Our aim is to work out a generalization of the results
of~\cite[Theorem~4.8]{IK} and~\cite[Sections~2 and~4]{Pfp} falling
in line with the exposition in the present paper
(with the Noetherianness/coherence assumptions removed).

 Firstly we return to the discussion of sfp\+injective and
sfp\+flat modules started in Section~\ref{sfp-modules-subsecn}.
 Denote the full subcategory of sfp\+injective left $A$\+modules
by $A\modl_\sfpin\subset A\modl$ and the full subcategory of
sfp\+flat left $B$\+modules by $B\modl_\sfpfl\subset B\modl$.
 It is clear from
Lemma~\ref{sfp-injective-flat-modules-are-closed-under} that
the categories $A\modl_\sfpin$ and $B\modl_\sfpfl$ have exact category
structures inherited from the abelian categories $A\modl$ and $B\modl$.

\begin{prop} \label{sfp-inj-homotopy-coderived-comparison}
\textup{(a)} The triangulated functor\/ $\sD^\co(A\modl_\sfpin)\rarrow
\sD^\co(A\modl)$ induced by the embedding of exact categories
$A\modl_\sfpin\rarrow A\modl$ is an equivalence of triangulated
categories. \par
\textup{(b)} If all sfp\+injective left $A$\+modules have finite
injective dimensions, then the triangulated functor\/
$\Hot(A\modl_\inj)\rarrow\sD^\co(A\modl)$ induced by the embedding
of additive/exact categories $A\modl_\inj\rarrow A\modl$ is
an equivalence of triangulated categories.
\end{prop}

\begin{proof}
 Part~(a) is but an application of the assertion dual
to~\cite[Proposition~A.3.1(b)]{Pcosh} (cf.~\cite[Theorem~2.2]{Pfp}).
 Part~(b) was proved in~\cite[Section~3.7]{Pkoszul} (for a more
general argument, one can use the assertion dual
to~\cite[Corollary~A.6.2]{Pcosh}).
 In fact, the assumption in part~(b) can be weakened by requiring only
that fp\+injective left $A$\+modules have finite injective dimensions,
as infinite direct sums of fp\+injective left $A$\+modules are
fp\+injective over an arbitrary ring (cf.~\cite[Theorem~2.4]{Pfp}).
\end{proof}

\begin{prop} \label{sfp-flat-homotopy-contraderived-comparison}
\textup{(a)} The triangulated functor\/ $\sD^\ctr(B\modl_\sfpfl)\rarrow
\sD^\ctr(B\modl)$ induced by the embedding of exact categories
$B\modl_\sfpfl\rarrow B\modl$ is an equivalence of triangulated
categories. \par
\textup{(b)} If all sfp\+flat left $B$\+modules have finite projective
dimensions, then the triangulated functor\/
$\Hot(B\modl_\proj)\rarrow\sD^\ctr(B\modl)$ induced by the embedding
of additive/exact categories $B\modl_\proj\rarrow B\modl$ is
an equivalence of triangulated categories.
\end{prop}

\begin{proof}
 Part~(a) is but an application of~\cite[Proposition~A.3.1(b)]{Pcosh}
(cf.~\cite[Theorem~4.4]{Pfp}).
 Part~(b) was proved in~\cite[Section~3.8]{Pkoszul} (for a more
general argument, see~\cite[Corollary~A.6.2]{Pcosh}).
\end{proof}

 The following lemma is a version of~\cite[Lemma~4.1]{Pfp} applicable
to arbitrary rings.

\begin{lem}
\textup{(a)} Let $P$ be a flat left $B$\+module and $K$ be
an $A$\+sfp\+injective $A$\+$B$\+bimodule.
 Then the tensor product $K\ot_B P$ is an sfp\+injective left
$A$\+module. \par
\textup{(b)} Let $J$ be an injective left $A$\+module and $K$
be a $B$\+sfp\+injective $A$\+$B$\+bimodule.
 Then the left $B$\+module\/ $\Hom_A(K,J)$ is sfp\+flat.
\end{lem}

\begin{proof}
 This is a particular case of the next
Lemma~\ref{sfp-flat-inj-tensor-hom-homological-dimension}.
\end{proof}

\begin{lem}  \label{sfp-flat-inj-tensor-hom-homological-dimension}
\textup{(a)} Let $P^\bu$ be a complex of flat left $B$\+modules
concentrated in the cohomological degrees $-n\le m\le 0$ and
$K^\bu$ be a complex of $A$\+$B$\+bimodules which, as a complex
of left $A$\+modules, is quasi-isomorphic to a complex of
sfp\+injective $A$\+modules concentrated in the cohomological
degrees $-d\le m\le l$.
 Then the tensor product $K^\bu\ot_B P^\bu$ is a complex of left
$A$\+modules quasi-isomorphic to a complex of sfp\+injective
left $A$\+modules concentrated in the cohomological degrees
$-n-d\le m\le l$. \par
\textup{(b)} Let $J^\bu$ be a complex of injective left $A$\+modules
concentrated in the cohomological degrees $0\le m\le n$ and
$K^\bu$ be a complex of $A$\+$B$\+bimodules which, as a complex
of right $B$\+modules, is quasi-isomorphic to a complex of
sfp\+injective right $B$\+modules concentrated in the cohomological
degrees $-d\le m\le l$.
 Then the complex of left $B$\+modules $\Hom_A(K^\bu,J^\bu)$ is
quasi-isomorphic to a complex of sfp\+flat $B$\+modules concentrated
in the cohomological degrees $-l\le m\le n+d$.
\end{lem}

\begin{proof}
 Part~(a): clearly, the tensor product $K^\bu\ot_BP^\bu$ is
quasi-isomorphic to a complex of left $A$\+modules concentrated
in the cohomological degrees $-n-d\le m\le l$; the nontrivial aspect
is to show that there is such a complex with sfp\+injective terms.
 Equivalently, this means that $\Ext_A^i(M,\>K^\bu\ot_BP^\bu)=0$ for
all strongly finitely presented left $A$\+modules $M$ and all $i>l$.
 Indeed, let $R^\bu$ be a resolution of $M$ by finitely generated
projective left $A$\+modules.
 Without loss of generality, we can assume that $K^\bu$ is a finite
complex of $A$\+$B$\+bimodules.
 Then the complex $\Hom_A(R^\bu,\>K^\bu\ot_BP^\bu)$ is isomorphic to
$\Hom_A(R^\bu,K^\bu)\ot_BP^\bu$ and the cohomology modules of
the complex $\Hom_A(R^\bu,K^\bu)$ are concentrated in the degrees
$-d\le m\le l$.

 Part~(b): clearly, the complex $\Hom_A(K^\bu,J^\bu)$ is
quasi-isomorphic to a complex of left $B$\+modules concentrated in
the cohomological degrees $-l\le m\le n+d$; we have to show that
there is such a complex with sfp\+flat terms.
 Equivalently, this means that $\Tor^B_i(N,\Hom_A(K^\bu,J^\bu))=0$ for
all strongly finitely presented right $B$\+modules $N$ and all $i>l$.
 Indeed, let $Q^\bu$ be a resolution of $N$ by finitely generated
projective right $B$\+modules.
 Without loss of generality, we can assume that $K^\bu$ is a finite
complex of $A$\+$B$\+bimodules.
 Then the complex $Q^\bu\ot_B\Hom_A(K^\bu,J^\bu)$ is isomorphic to
$\Hom_A(\Hom_{B^\rop}(Q^\bu,K^\bu),J^\bu)$ and the cohomology modules of
the complex $\Hom_{B^\rop}(Q^\bu,K^\bu)$ are concentrated in the degrees
$-d\le m\le l$.
\end{proof}

 A \emph{dualizing complex} of $A$\+$B$\+bimodules $L^\bu=D^\bu$ is
a pseudo-dualizing complex (according to the definition in
Section~\ref{two-rings-auslander-bass-subsecn}) satisfying
the following additional condition:
\begin{enumerate}
\renewcommand{\theenumi}{\roman{enumi}}
\item As a complex of left $A$\+modules, $D^\bu$ is quasi-isomorphic
to a finite complex of sfp\+injective $A$\+modules, and as a complex
of right $B$\+modules, $D^\bu$ is quasi-isomorphic to a finite complex
of sfp\+injective $B$\+modules.
\end{enumerate}

 This definition of a dualizing complex is a version of the definition
of a ``cotilting bimodule complex'' in~\cite[Section~2]{Miy},
reproduced as the definition of a ``weak dualizing complex''
in~\cite[Section~3]{Pfp} (cf.\  the definition of a ``dualizing
complex'' in~\cite[Section~4]{Pfp}), extended from the case of coherent
rings to arbitrary rings $A$ and~$B$ in the spirit
of~\cite[Definition~2.1]{HW} and~\cite[Section~B.4]{Pcosh}.
 (Other versions of the definition of a dualizing complex of bimodules
known in the literature can be found in~\cite[Definition~1.1]{YZ}
and~\cite[Definition~1.1]{CFH}.)
 In order to prove the results below, we will have to impose some
homological dimension conditions on the rings $A$ and $B$, bringing
our definition of a dualizing complex even closer to the definition
in~\cite{Miy} and the definition of a weak dualizing complex
in~\cite{Pfp}.

 Specifically, we will have to assume that all sfp\+injective left
$A$\+modules have finite injective dimensions.
 This assumption always holds when the ring $A$ is left coherent
and there exists an integer $n\ge0$ such that every left ideal in $A$
is generated by at most~$\aleph_n$ elements~\cite[Proposition~2.3]{Pfp}.

 We will also have to assume that all sfp\+flat left $B$\+modules have
finite projective dimensions.
 For a right coherent ring $B$, this would simply mean that all flat
left $B$\+modules have finite projective dimensions.
 The class of rings satisfying the latter condition was discussed,
under the name of ``left $n$\+perfect rings'', in
the paper~\cite{EJLR}.
 We refer to~\cite[Proposition~4.3]{Pfp}, the discussions
in~\cite[Section~3]{IK} and~\cite[Section~3.8]{Pkoszul}, and
the references therein, for further sufficient conditions.

 Let us choose the parameter~$l_2$ in such a way that $D^\bu$ is
quasi-isomorphic to a complex of sfp\+injective left $A$\+modules
concentrated in the cohomological degrees $-d_1\le m\le l_2$ and
to a complex of sfp\+injective right $B$\+modules concentrated
in the cohomological degrees $-d_1\le m\le l_2$.

\begin{prop}
 Let $A$ and $B$ be associative rings such that all sfp\+injective
left $A$\+modules have finite injective dimensions and all
sfp\+flat left $B$\+modules have finite projective dimensions.
 Let $L^\bu=D^\bu$ be a dualizing complex of $A$\+$B$\+bimodules,
and let the parameter~$l_2$ be chosen as stated above.
 Then the related minimal corresponding classes\/
$\sE^{l_2}=\sE^{l_2}(D^\bu)$ and\/ $\sF^{l_2}=\sF^{l_2}(D^\bu)$ are
contained in the classes of sfp\+injective $A$\+modules and
spf\+flat $B$\+modules, $\sE^{l_2}\subset A\modl_\sfpin$ and\/
$\sF^{l_2}\subset B\modl_{\sfpfl}$.

 Moreover, let $n\ge0$~be an integer such that the injective dimensions
of sfp\+injective left $A$\+modules do not exceed~$n$ and the projective
dimensions of sfp\+flat left $B$\+modules do not exceed~$n$.
 Then the classes of modules\/ $\sE=A\modl_\sfpin$ and\/
$\sF=B\modl_\sfpfl$ satisfy the conditions~(I\+-IV) with the parameters
$l_1=n+d_1$ and~$l_2$.
\end{prop}

\begin{proof}
 The second assertion is true, as the conditions~(I\+-II) are satisfied
by Lemma~\ref{sfp-injective-flat-modules-are-closed-under}
and the conditions~(III\+-IV) hold by
Lemma~\ref{sfp-flat-inj-tensor-hom-homological-dimension}.
 The first assertion follows from the second one together with
Lemma~\ref{sfp-injective-flat-modules-are-closed-under}.
\end{proof}

 Let $B\modl_\flat\subset B\modl$ denote the full subcategory of
flat left $B$\+modules.
 It inherits the exact category structure of the abelian category
$B\modl$.

\begin{cor} \label{two-rings-dualizing-main-cor}
 Let $A$ and $B$ be associative rings such that all sfp\+injective
left $A$\+modules have finite injective dimensions and all
sfp\+flat left $B$\+modules have finite projective dimensions.
 Let $L^\bu=D^\bu$ be a dualizing complex of $A$\+$B$\+bimodules,
and let the parameter~$l_2$ be chosen as above.
 Then there is a triangulated equivalence\/ $\sD^\co(A\modl)\simeq
\sD^\ctr(B\modl)$ provided by (appropriately defined) mutually
inverse functors\/ $\boR\Hom_A(D^\bu,{-})$ and $D^\bu\ot_B^\boL{-}$.

 Furthermore, there is a chain of triangulated equivalences
\begin{multline*}
 \sD^\co(A\modl)\simeq\sD^{\abs=\varnothing}(A\modl_\sfpin)
 \simeq\sD^{\abs=\varnothing}(\sE^{l_2})\simeq \\ \Hot(A\modl_\inj)
 \simeq\Hot(B\modl_\proj) \\ \simeq\sD^{\abs=\varnothing}(\sF^{l_2})
 \simeq\sD^{\abs=\varnothing}(B\modl_\flat)\simeq
 \sD^{\abs=\varnothing}(B\modl_\sfpfl) \simeq\sD^\ctr(B\modl),
\end{multline*}
where the notation\/ $\sD^{\abs=\varnothing}(\sC)$ is a shorthand for
an identity isomorphism of triangulated categories\/
$\sD^\abs(\sC)=\sD(\sC)$ between the absolute derived category and
the conventional derived category of an exact category\/~$\sC$.
 Moreover, for any symbol\/ $\star=\b$, $+$, $-$, or\/~$\varnothing$,
there are triangulated equivalences
\begin{multline*}
 \sD^\star(A\modl_\sfpin)\simeq\sD^\star(\sE^{l_2}) \\ \simeq
 \Hot^\star(A\modl_\inj) \simeq\Hot^\star(B\modl_\proj) \\ \simeq
 \sD^\star(\sF^{l_2})\simeq\sD^\star(B\modl_\flat)\simeq
 \sD^\star(B\modl_\sfpfl).
\end{multline*}
\end{cor}

\begin{proof}
 The exact categories $A\modl_\sfpin$ and $B\modl_\sfpfl$ have finite
homological dimensions by assumption.
 Hence so do their full subcategories $\sE^{l_2}$, $\sF^{l_2}$, and
$B\modl_\flat$ satisfying the condition~(I) or~(II).
 It follows easily (see, e.~g., \cite[Remark~2.1]{Psemi}
and~\cite[Proposition~A.5.6]{Pcosh}) that a complex in any one of these
exact categories is acyclic if and only if it is absolutely acyclic,
and that their (conventional or absolute) derived categories are
equivalent to the homotopy categories of complexes of injective or
projective objects.
 The same, of course, applies to the coderived and/or contraderived
categories of those of these exact categories that happen to be
closed under the infinite direct sums or infinite products in their
respective abelian module categories.
 The same also applies to the bounded versions of the conventional
or absolute derived categories and bounded versions of the homotopy
categories.

 Propositions~\ref{sfp-inj-homotopy-coderived-comparison}
and~\ref{sfp-flat-homotopy-contraderived-comparison} provide
the equivalences with the coderived category $\sD^\co(A\modl)$ or
the contraderived category $\sD^\ctr(B\modl)$.
 Thus we have shown in all the cases that the mentioned triangulated
categories of complexes of $A$\+modules are equivalent to each other
and the mentioned triangulated categories of complexes of $B$\+modules
are equivalent to each other.
 It remains to construct the equivalences connecting complexes of
$A$\+modules with complexes of $B$\+modules.

 Specifically, the equivalence $\sD^\co(A\modl)\simeq\sD^\ctr(B\modl)$
can be obtained in the same way as in~\cite[Theorem~4.5]{Pfp}, using
the equivalence $\sD^\co(A\modl)\simeq\Hot(A\modl_\inj)$ in order to
construct the derived functor $\boR\Hom_A(D^\bu,{-})$ and
the equivalence $\sD^\ctr(B\modl)\simeq\sD^\abs(B\modl_\flat)$ or
$\sD^\ctr(B\modl)\simeq\Hot(B\modl_\proj)$ in order to construct
the derived functor $D^\bu\ot_B^\boL{-}$.
 More generally, the equivalence $\sD^\star(\sE^{l_2})\simeq
\sD^\star(\sF^{l_2})$ can be produced as a particular case of
Theorem~\ref{two-rings-upper-main-thm} above.
\end{proof}

\Section{Base Change} \label{two-rings-base-change-subsecn}
\label{two-rings-relative-dualizing-subsecn}

 The aim of this section and the next one is to formulate
a generalization of the definitions and results of~\cite[Section~5]{Pfp}
that would fit naturally in our present context.
 Our exposition is informed by that in~\cite[Section~5]{Chr}.

 Let $A\rarrow R$ and $B\rarrow S$ be two homomorphisms of associative
rings.
 Let $\sE\subset A\modl$ be a full subcategory satisfying
the condition~(I), and let $\sF\subset B\modl$ be a full subcategory
satisfying the condition~(II).
 We denote by $\sG=\sG_\sE\subset R\modl$ the full subcategory formed
by all the left $R$\+modules whose underlying left $A$\+modules belong
to~$\sE$, and by $\sH=\sH_\sF\subset S\modl$ the full subcategory formed
by all the left $S$\+modules whose underlying left $B$\+modules belong
to~$\sF$.

\begin{lem} \label{underlying-modules-of-injectives-projectives}
\textup{(a)} The full subcategory\/ $\sG_\sE\subset R\modl$ satisfies
the condition~(I) if and only if the underlying $A$\+modules of all
the injective left $R$\+modules belong to\/~$\sE$. \par
\textup{(b)} The full subcategory\/ $\sH_\sF\subset S\modl$ satisfies
the condition~(II) if and only if the underlying $B$\+modules of all
the projective left $S$\+modules belong to\/~$\sF$.  \qed
\end{lem}

 Assume further that the equivalent conditions of
Lemma~\ref{underlying-modules-of-injectives-projectives}(a)
and~(b) hold, and additionally that the full subcategory
$\sE\subset A\modl$ is closed under infinite direct sums and the full
subcategory $\sF\subset B\modl$ is closed under infinite products.
 Then we get two commutative diagrams of triangulated functors,
where the vertical arrows are Verdier quotient functors described
in Section~\ref{pseudo-derived-introd}, and the horizontal arrows
are the forgetful functors:
$$
\begin{diagram}
\node{\sD^\co(R\modl)}\arrow{s,A}\arrow[2]{e}
\node[2]{\sD^\co(A\modl)}\arrow{s,A} \\
\node{\sD(\sG_\sE)}\arrow{s,A}\arrow[2]{e,}
\node[2]{\sD(\sE)}\arrow{s,A} \\
\node{\sD(R\modl)}\arrow[2]{e}
\node[2]{\sD(A\modl)}
\end{diagram} \qquad
\begin{diagram}
\node{\sD^\ctr(S\modl)}\arrow{s,A}\arrow[2]{e}
\node[2]{\sD^\ctr(B\modl)}\arrow{s,A} \\
\node{\sD(\sH_\sF)}\arrow{s,A}\arrow[2]{e,}
\node[2]{\sD(\sF)}\arrow{s,A} \\
\node{\sD(S\modl)}\arrow[2]{e}
\node[2]{\sD(B\modl)}
\end{diagram}
$$

 We recall that a triangulated functor is called \emph{conservative}
if it reflects isomorphisms, or equivalently, takes nonzero objects
to nonzero objects.
 For example, the forgetful functors $\sD(R\modl)\rarrow\sD(A\modl)$
and $\sD(S\modl)\rarrow\sD(B\modl)$ are conservative, while
the forgetful functors $\sD^\co(R\modl)\rarrow\sD^\co(A\modl)$ and
$\sD^\ctr(S\modl)\rarrow\sD^\ctr(B\modl)$ are \emph{not}, in general.

\begin{lem} \label{semi-pseudo-triangulated-conservative}
 The forgetful functors\/ $\sD(\sG_\sE)\rarrow\sD(\sE)$ and\/
$\sD(\sH_\sF)\rarrow\sD(\sF)$ are conservative.
\end{lem}

\begin{proof}
 Follows from the definition of the derived category of
an exact category.
\end{proof}

 One can say that a complex of left $A$\+modules is
\emph{$\sE$\+pseudo-coacyclic} if its image under the Verdier
quotient functor $\sD^\co(A\modl)\rarrow\sD(\sE)$ vanishes.
 All coacyclic complexes are pseudo-coacyclic, and all pseudo-coacyclic
complexes are acyclic.

 Similarly, one can say that a complex of left $B$\+modules is
\emph{$\sF$\+pseudo-contraacyclic} if its image under the Verdier
quotient functor $\sD^\ctr(B\modl)\rarrow\sD(\sF)$ vanishes.
 All contraacyclic complexes are pseudo-contraacyclic, and all
pseudo-contraacyclic complexes are acyclic.

\begin{lem} \label{semi-pseudo-co-contra-acyclic}
\textup{(a)} Let\/ $\sE\subset A\modl$ be a full subcategory
satisfying the condition~(I), closed under infinite direct sums, and
containing the underlying $A$\+modules of injective left $R$\+modules.
 Then a complex of left $R$\+modules is\/
$\sG_\sE$\+pseudo-coacyclic if and only if it is\/
$\sE$\+pseudo-coacyclic as a complex of left $A$\+modules. \par
\textup{(b)} Let\/ $\sF\subset B\modl$ be a full subcategory
satisfying the condition~(II), closed under infinite products, and
containing the underlying $B$\+modules of projective left $S$\+modules.
 Then a complex of left $S$\+modules is\/
$\sH_\sF$\+pseudo-contraacyclic if and only if it is\/
$\sF$\+pseudo-contraacyclic as a complex of left $B$\+modules.
\end{lem}

\begin{proof}
 This is a restatement of
Lemma~\ref{semi-pseudo-triangulated-conservative}.
\end{proof}

 The terminology in the following definition follows that
in~\cite[Section~5]{Pfp}, where ``relative dualizing complexes''
are discussed.
 In~\cite[Section~5]{Chr}, a related phenomenon is called ``base
change''.

 A \emph{relative pseudo-dualizing complex} for a pair of associative
ring homomorphisms $A\rarrow R$ and $B\rarrow S$ is a set of data
consisting of a pseudo-dualizing complex of $A$\+$B$\+bimodules
$L^\bu$, a pseudo-dualizing complex of $R$\+$S$\+bimodules $U^\bu$,
and a morphism of complexes of $A$\+$B$\+bimodules $L^\bu\rarrow U^\bu$
satisfying the following condition:
\begin{enumerate}
\renewcommand{\theenumi}{\roman{enumi}}
\setcounter{enumi}{3}
\item the induced morphism $R\ot_A^\boL L^\bu\rarrow U^\bu$ is
an isomorphism in the derived category of left $R$\+modules
$\sD^-(R\modl)$, and the induced morphism $L^\bu\ot_B^\boL S\rarrow
U^\bu$ is an isomorphism in the derived category of right
$S$\+modules $\sD^-(\modr S)$.
\end{enumerate}

 Notice that the condition~(ii) in the definition of
a pseudo-dualizing complex in
Section~\ref{two-rings-auslander-bass-subsecn} holds for
the complex $U^\bu$ whenever it holds for the complex $L^\bu$
and the above condition~(iv) is satisfied.
 The following result, which is our version of~\cite[Theorem~5.1]{Chr},
explains what happens with the condition~(iii).
 We will assume that the complex $L^\bu$ is concentrated in
the cohomological degrees $-d_1\le m\le d_2$ and the complex $U^\bu$
is concentrated in the cohomological degrees $-t_1\le m\le t_2$.
 Let $L^\bu\.{}^\rop$ denote the complex $L^\bu$ viewed as a complex
of $B^\rop$\+$A^\rop$\+bimodules.

\begin{prop} \label{relative-pseudo-dualizing-auslander-equivalence}
 Let $L^\bu$ be a pseudo-dualizing complex of $A$\+$B$\+bimodules,
$U^\bu$ be a finite complex of $R$\+$S$\+bimodules, and $L^\bu\rarrow
U^\bu$ be a morphism of complexes of $A$\+$B$\+bimodules satisfying
the condition~(iv).
 Then $U^\bu$ is a pseudo-dualizing complex of $R$\+$S$\+bimodules
if and only if there exists an integer $l_1\ge d_1$ such that the right
$A$\+module $R$ belongs to the class\/ $\sF_{l_1}(L^\bu\.{}^\rop)\subset
A^\rop\modl$ and the left $B$\+module $S$ belongs to the class\/
$\sF_{l_1}(L^\bu)\subset B\modl$.
\end{prop}

\begin{proof}
 The key observation is that the natural isomorphism
$\boR\Hom_R(U^\bu,U^\bu)\simeq\boR\Hom_R(R\ot_A^\boL L^\bu,\>U^\bu)
\simeq\boR\Hom_A(L^\bu,U^\bu)\simeq\boR\Hom_A(L^\bu,\>
L^\bu\ot_B^\boL S)$ identifies the homothety morphism
$S^\rop\rarrow\boR\Hom_R(U^\bu,U^\bu)$ with the adjunction morphism
$S\rarrow\boR\Hom_A(L^\bu,\>L^\bu\ot_B^\boL S)$.
 Similarly, the natural isomorphism $\boR\Hom_{S^\rop}(U^\bu,U^\bu)\simeq
\boR\Hom_{B^\rop}(L^\bu,\>R\ot_A^\boL L^\bu)$ identifies the homothety
morphism $R\rarrow\boR\Hom_{S^\rop}(U^\bu,U^\bu)$ with the adjunction
morphism $R\rarrow\boR\Hom_{B^\rop}(L^\bu,\>R\ot_A^\boL L^\bu)$.
 It remains to say that one can take any integer $l_1$ such that
$l_1\ge d_1$ and $l_1\ge t_1$.
\end{proof}

 The next proposition is our version of~\cite[Proposition~5.3]{Chr}.

\begin{prop} \label{relative-maximal-corresponding-classes}
 Let $L^\bu\rarrow U^\bu$ be a relative pseudo-dualizing complex for
a pair of ring homomorphisms $A\rarrow R$ and $B\rarrow S$.
 Let $l_1$ be an integer such that $l_1\ge d_1$ and $l_1\ge t_1$.
 Then \par
\textup{(a)} a left $R$\+module belongs to the full subcategory\/
$\sE_{l_1}(U^\bu)\subset R\modl$ if and only if its underlying
$A$\+module belongs to the full subcategory\/
$\sE_{l_1}(L^\bu)\subset A\modl$; \par
\textup{(b)} a left $S$\+module belongs to the full subcategory\/
$\sF_{l_1}(U^\bu)\subset S\modl$ if and only if its underlying
$B$\+module belongs to the full subcategory\/
$\sF_{l_1}(L^\bu)\subset B\modl$.
\end{prop}

\begin{proof}
 The assertions follow from the commutative diagrams of the pairs
of adjoint functors and the forgetful functors
$$\dgARROWLENGTH=4em
\begin{diagram} 
\node{\sD(R\modl)}\arrow[2]{e,t}{\boR\Hom_R(U^\bu,{-})}\arrow{s}
\node[2]{\sD(S\modl)}\arrow{s} \\
\node{\sD(A\modl)}\arrow[2]{e,t}{\boR\Hom_A(L^\bu,{-})}
\node[2]{\sD(B\modl)}
\end{diagram}
\,\,\,\,
\begin{diagram} 
\node{\sD(R\modl)}\arrow{s}
\node[2]{\sD(S\modl)}\arrow{s}\arrow[2]{w,t}{U^\bu\ot^\boL_S{-}} \\
\node{\sD(A\modl)}
\node[2]{\sD(B\modl)}\arrow[2]{w,t}{L^\bu\ot^\boL_B{-}}
\end{diagram}
$$
together with the compatibility of the adjunctions with the forgetful
functors and conservativity of the forgetful functors.
\end{proof}

\begin{prop} \label{relative-abstract-corresponding-classes}
 Let $L^\bu\rarrow U^\bu$ be a relative pseudo-dualizing complex for
a pair of ring homomorphisms $A\rarrow R$ and $B\rarrow S$, and
let\/ $\sE\subset A\modl$ and\/ $\sF\subset B\modl$ be a pair of full
subcategories satisfying the conditions~(I\+-IV) with respect to
the pseudo-dualizing complex $L^\bu$ with some parameters $l_1$
and~$l_2$ such that $l_1\ge\nobreak d_1$, $l_1\ge\nobreak t_1$,
$l_2\ge\nobreak d_2$, and $l_2\ge\nobreak t_2$.
 Suppose that the underlying $A$\+modules of all the injective
left $R$\+modules belong to\/ $\sE$ and the underlying $B$\+modules
of all the projective left $S$\+modules belong to\/~$\sF$.
 Then the pair of full subcategories\/ $\sG_\sE\subset R\modl$ and\/
$\sH_\sF\subset S\modl$ satisfies the conditions~(I\+-IV) with
respect to the pseudo-dualizing complex $U^\bu$ with the same
parameters $l_1$ and~$l_2$.
\end{prop}

\begin{proof}
 The conditions~(I\+-II) hold by
Lemma~\ref{underlying-modules-of-injectives-projectives},
and the conditions~(III\+-IV) are easy to check using
the standard properties of the (co)resolution
dimensions~\cite[Corollary~A.5.2]{Pcosh}.
\end{proof}

\begin{cor} \label{relative-abstract-main-cor}
 In the context and assumptions of
Proposition~\textup{\ref{relative-abstract-corresponding-classes}},
for any symbol\/ $\star=\b$, $+$, $-$, $\varnothing$, $\abs+$,
$\abs-$, $\co$, $\ctr$, or\/~$\abs$, there is a triangulated
equivalence\/ $\sD^\star(\sG_\sE)\simeq\sD^\star(\sH_\sF)$ provided
by (appropriately defined) mutually inverse functors\/
$\boR\Hom_R(U^\bu,{-})$ and $U^\bu\ot_S^\boL{-}$.

 Here, in the case\/ $\star=\co$ it is assumed that the full
subcategories\/ $\sE\subset A\modl$ and\/ $\sF\subset B\modl$ are
closed under infinite direct sums, while in the case\/ $\star=\ctr$
it is assumed that these two full subcategories are closed under
infinite products.
\end{cor}

\begin{proof}
 This is a particular case of
Theorem~\ref{two-rings-generalized-main-thm}.
\end{proof}

 In the situation of Corollary~\ref{relative-abstract-main-cor}
the triangulated equivalences $\sD^\star(\sG_\sE)\simeq
\sD^\star(\sH_\sF)$ and $\sD^\star(\sE)\simeq\sD^\star(\sF)$ form
a commutative diagram with the triangulated forgetful functors
\begin{equation} \label{relative-pseudo-dualizing-triangulated-diagram}
\begin{diagram} 
\node{\sD^\star(\sG_\sE)}\arrow[2]{e,=}\arrow{s}
\node[2]{\sD^\star(\sH_\sF)}\arrow{s} \\
\node{\sD^\star(\sE)}\arrow[2]{e,=}
\node[2]{\sD^\star(\sF)}
\end{diagram}
\end{equation}

\Section{Relative Dualizing Complexes}

 Let $A$ be an associative ring.
 The \emph{sfp\+injective dimension} of an $A$\+module is the minimal
length of its coresolution by sfp\+injective $A$\+modules.
 The sfp\+injective dimension of a left $A$\+module $E$ is equal to
the supremum of all the integers $n\ge\nobreak0$ for which there exists
a strongly finitely presented left $A$\+module $M$ such that
$\Ext^n_A(M,E)\ne0$.
 The \emph{sfp\+flat dimension} of an $A$\+module is the minimal
length of its resolution by sfp\+flat $A$\+modules.
 The sfp\+flat dimension of a left $A$\+module $F$ is equal to
the supremum of all the integers $n\ge\nobreak0$ for which there exists
a strongly finitely presented right $A$\+module $N$ such that
$\Tor_n^A(N,F)\ne0$.

\begin{lem} \label{sfp-flat-injective-dimension-dualization}
 The sfp\+flat dimension of a right $A$\+module $G$ is equal to
the sfp\+injective dimension of the left $A$\+module\/
$\Hom_\boZ(G,\boQ/\boZ)$.  \qed
\end{lem}

 Let $A\rarrow R$ and $B\rarrow S$ be homomorphisms of associative
rings.

\begin{lem} \label{sfp-injective-flat-dimensions-of-underlying-modules}
\textup{(a)} The supremum of sfp\+injective dimensions of the underlying
left $A$\+modules of injective left $R$\+modules is equal to
the sfp\+flat dimension of the right $A$\+module~$R$. \par
\textup{(b)} The supremum of sfp\+flat dimensions of the underlying
left $B$\+modules of projective left $S$\+modules is equal to
the sfp\+flat dimension of the left $B$\+module~$S$.
\end{lem}

\begin{proof}
 In part~(a), one notices that the injective left $R$\+modules are
precisely the direct summands of infinite products of copies of
the $R$\+module $\Hom_\boZ(R,\boQ/\boZ)$, and takes into account
Lemma~\ref{sfp-flat-injective-dimension-dualization}.
 Part~(b) is easy
(cf.\ Lemma~\ref{sfp-injective-flat-modules-are-closed-under}).
\end{proof}

  Assume that all sfp\+injective left $A$\+modules have finite
injective dimensions and all sfp\+flat left $B$\+modules have finite
projective dimensions, as in Section~\ref{two-rings-dualizing-subsecn}.
 Fix an integer $n\ge0$, and set $\sE=A\modl_\sfpin(n)\subset A\modl$
to be the full subcategory of all left $A$\+modules whose sfp\+injective
dimension does not exceed~$n$.
 Similarly, set $\sF=B\modl_\sfpfl(n)\subset B\modl$ to be the full
subcategory of all left $B$\+modules whose sfp\+flat dimension
does not exceed~$n$.

\begin{prop} \label{derived-finite-sfp-dim-co-contra-derived}
\textup{(a)} The embedding of exact/abelian categories\/
$\sE\rarrow A\modl$ induces an equivalence of triangulated categories\/
$\sD^{\abs=\varnothing}(\sE)\simeq\sD^\co(A\modl)$. \par
\textup{(b)} The embedding of exact/abelian categories\/
$\sF\rarrow B\modl$ induces an equivalence of triangulated categories\/
$\sD^{\abs=\varnothing}(\sF)\simeq\sD^\ctr(B\modl)$.
\end{prop}

\begin{proof}
 Follows from~\cite[Remark~2.1]{Psemi}, 
Propositions~\ref{sfp-inj-homotopy-coderived-comparison}\+-%
\ref{sfp-flat-homotopy-contraderived-comparison},
and~\cite[Proposition~A.5.6]{Pcosh}
(cf.\ the proof of Corollary~\ref{two-rings-dualizing-main-cor}).
\end{proof}

 In other words, in the terminology of
Section~\ref{two-rings-base-change-subsecn}, one can say that
the class of $\sE$\+pseudo-coacyclic complexes coincides with that
of coacyclic complexes of left $A$\+modules, while the class of
$\sF$\+pseudo-contraacyclic complexes coincides with that of
contraacyclic complex of left $B$\+modules.

\medskip

 The following definitions were given in the beginning
of~\cite[Section~5]{Pfp}.
 The \emph{$R/A$\+semicoderived category} of left $R$\+modules
$\sD^\sico_A(R\modl)$ is defined as the quotient category of
the homotopy category of complexes of left $R$\+modules $\Hot(R\modl)$
by its thick subcategory of complexes of $R$\+modules \emph{that are
coacyclic as complexes of $A$\+modules}.
 Similarly, the \emph{$S/B$\+semicontraderived category} of left
$S$\+modules $\sD^\sictr_B(S\modl)$ is defined as the quotient category
of the homotopy category of complexes of left $S$\+modules
$\Hot(S\modl)$ by its thick subcategory of complexes of $S$\+modules
\emph{that are contraacyclic as complexes of $B$\+modules}.

 As in Section~\ref{two-rings-base-change-subsecn}, we denote by
$\sG_\sE\subset R\modl$ the full subcategory of all left $R$\+modules
whose underlying $A$\+modules belong to~$\sE$, and by $\sH_\sF
\subset S\modl$ the full subcategory of all left $S$\+modules whose
underlying $B$\+modules belong to~$\sF$.
 The next proposition is our version of~\cite[Theorems~5.1
and~5.2]{Pfp}.

\begin{prop} \label{semi-co-contra-derived-comparison}
\textup{(a)} Assume that all sfp\+injective left $A$\+modules have
finite injective dimensions and the sfp\+flat dimension of the right
$A$\+module $R$ does not exceed~$n$.
 Then the embedding of exact/abelian categories\/
$\sG_\sE\rarrow R\modl$ induces an equivalence of triangulated
categories\/ $\sD(\sG_\sE)\simeq\sD^\sico_A(R\modl)$. \par
\textup{(b)} Assume that all sfp\+flat left $B$\+modules have
finite projective dimensions and the sfp\+flat dimension of the left
$B$\+module $S$ does not exceed~$n$.
 Then the embedding of exact/abelian categories\/
$\sH_\sF\rarrow S\modl$ induces an equivalence of triangulated
categories\/ $\sD(\sH_\sF)\simeq\sD^\sictr_B(S\modl)$.
\end{prop}

\begin{proof}
 The assumptions of Lemma~\ref{semi-pseudo-co-contra-acyclic}(a)
or~(b) hold by
Lemma~\ref{sfp-injective-flat-dimensions-of-underlying-modules},
so its conclusion is applicable, and it remains to recall
Proposition~\ref{derived-finite-sfp-dim-co-contra-derived}.
\end{proof}

 So, in the assumptions of
Proposition~\ref{semi-co-contra-derived-comparison},
the $R/A$\+semicoderived category of left $R$\+modules is
a pseudo-coderived category of left $R$\+modules and
the $S/B$\+semicontraderived category of left $S$\+modules
is a pseudo-contraderived category of left $S$\+modules, in
the sense of Section~\ref{pseudo-derived-introd}.

\medskip

 A \emph{relative dualizing complex} for a pair of associative ring
homomorphisms $A\rarrow R$ and $B\rarrow S$ is a relative
pseudo-dualizing complex $L^\bu\rarrow U^\bu$ in the sense of
the definition in Section~\ref{two-rings-relative-dualizing-subsecn}
such that $L^\bu=D^\bu$ is a dualizing complex of $A$\+$B$\+bimodules
in the sense of the definition in
Section~\ref{two-rings-dualizing-subsecn}.
 In other words, the condition~(i) of
Section~\ref{two-rings-dualizing-subsecn} and the conditions~(ii\+iii)
of Section~\ref{two-rings-auslander-bass-subsecn} have to be
satisfied for $D^\bu$, the condition~(iii)
of Section~\ref{two-rings-auslander-bass-subsecn} has to be satisfied
for $U^\bu$, and the condition~(iv) of
Section~\ref{two-rings-relative-dualizing-subsecn} has to be satisfied
for the morphism $D^\bu\rarrow U^\bu$.

 Notice that, in the assumption of finiteness of flat dimensions of
the right $A$\+module $R$ and the left $B$\+module~$S$,
the condition~(iii) for the complex $U^\bu$ follows from the similar
condition for the complex $L^\bu$ together with the condition~(iv),
by Proposition~\ref{relative-pseudo-dualizing-auslander-equivalence}
and Remark~\ref{finite-injective-flat-dim-bass-auslander-remark}.

 The following corollary is our generalization
of~\cite[Theorem~5.6]{Pfp}.

\begin{cor} \label{relative-dualizing-cor}
 Let $A$ and $B$ be associative rings such that all sfp\+injective
left $A$\+modules have finite injective dimensions and all sfp\+flat
left $B$\+modules have finite projective dimensions.
 Let $A\rarrow R$ and $B\rarrow S$ be associative ring homomorphisms
such that  the ring $R$ is a right $A$\+module of finite flat dimension
and the ring $S$ is a left $B$\+module of finite flat dimension.
 Let $D^\bu\rarrow U^\bu$ be a relative dualizing complex for
$A\rarrow R$ and $B\rarrow S$.
 Then there is a triangulated equivalence\/ $\sD^\sico_A(R\modl)\simeq
\sD^\sictr_B(S\modl)$ provided by mutually inverse functors\/
$\boR\Hom_R(U^\bu,{-})$ and $U^\bu\ot_S^\boL{-}$.
\end{cor}

\begin{proof}
 Combine Corollary~\ref{relative-abstract-main-cor}
(for $\star=\varnothing$) with
Proposition~\ref{semi-co-contra-derived-comparison}.
\end{proof}

 The triangulated equivalences provided by
Corollaries~\ref{two-rings-dualizing-main-cor}
and~\ref{relative-dualizing-cor} form a commutative diagram with
the (conservative) triangulated forgetful functors
\begin{equation}
\dgARROWLENGTH=4em
\begin{diagram} 
\node{\sD^\sico_A(R\modl)}\arrow[2]{e,=}\arrow{s}
\node[2]{\sD^\sictr_B(S\modl)}\arrow{s} \\
\node{\sD^\co(A\modl)}\arrow[2]{e,=}
\node[2]{\sD^\ctr(B\modl)}
\end{diagram}
\end{equation}
 This is the particular case of the commutative
diagram~\eqref{relative-pseudo-dualizing-triangulated-diagram}
that occurs in the situation of
Corollary~\ref{relative-dualizing-cor}.

\Section{Deconstructibility of the Auslander and Bass Classes}
\label{deconstructibility-secn}

 Let $L^\bu$ be a pseudo-dualizing complex for associative rings $A$
and~$B$.
 Assume that the finite complex $L^\bu$ is situated in the cohomological
degrees $-d_1\le m\le d_2$, and choose an integer $l_1\ge d_1$.
 As in Section~\ref{two-rings-auslander-bass-subsecn}, we consider
the Auslander class of left $B$\+modules $\sF_{l_1}=\sF_{l_1}(L^\bu)
\subset B\modl$ and the Bass class of left $A$\+modules
$\sE_{l_1}=\sE_{l_1}(L^\bu)\subset A\modl$.
 In this section we will show, generalizing the result
of~\cite[Proposition~3.10]{EH}, that $\sF_{l_1}$ and $\sE_{l_1}$ are
\emph{deconstructible classes} of modules.

 We recall the notation $\tau_{\le n}$ and $\tau_{\ge n}$ for
the functors of canonical truncation of complexes of modules.

\begin{lem} \label{ext-tor-filtered-colimit-lemma}
\textup{(a)} For every $n\in\boZ$, the functor assigning to a left
$B$\+module $N$ the left $A$\+module\/ $\Tor_n^B(L^\bu,N)$
preserves filtered inductive limits. \par
\textup{(b)} For every $n\in\boZ$ and $i\in\boZ$, the functor
assigning to a left $B$\+module $N$ the left $B$\+module\/
$\Ext^i_A(L^\bu,\tau_{\ge-n}(L^\bu\ot_B^\boL N))$ preserves
filtered inductive limits. \par
\textup{(c)} For every $n\in\boZ$, the functor assigning to a left
$A$\+module $M$ the left $B$\+module\/ $\Ext^n_A(L^\bu,M)$
preserves filtered inductive limits. \par
\textup{(d)} For every $i\in\boZ$, the functor assigning to a left
$A$\+module $M$ the left $A$\+module\/
$\Tor^B_i(L^\bu,\boR\Hom_A(L^\bu,M))$ preserves filtered
inductive limits.
\end{lem}

\begin{proof}
 First of all we notice that in all the four cases it suffices to check
that the functor in question preserves filtered inductive limits when
viewed as a functor taking values in the category of abelian groups.
 Part~(a) holds for any complex of $A$\+$B$\+bimodules $L^\bu$
(it suffices to replace $L^\bu$, viewed as a complex of right
$B$\+modules, by its homotopy flat resolution).
 Part~(c) holds because the complex $L^\bu$, viewed as a complex
of left $A$\+modules, is quasi-isomorphic to a bounded above
complex of finitely generated projective left $A$\+modules.

 To prove part~(b), it is convenient to use the existence of
a (nonadditive) functor assigning to a module its free resolution.
 Given a filtered diagram of left $B$\+modules $(N_\alpha)$,
this allows to construct a filtered diagram of nonpositively
cohomologically graded complexes of free left $B$\+modules
$(Q^\bu_\alpha)$ endowed with a quasi-isomorphism of diagrams
of complexes of left $B$\+modules $Q^\bu_\alpha\rarrow N_\alpha$.
 Then the inductive limit of such free resolutions
$\varinjlim_\alpha Q^\bu_\alpha$ is a flat resolution of
the left $B$\+module $\varinjlim_\alpha N_\alpha$.

 Now we have a filtered diagram of complexes of left $A$\+modules
$\tau_{\ge -n}(L^\bu\ot_B Q^\bu_\alpha)$ representing the derived
category objects $\tau_{\ge -n}(L^\bu\ot_B^\boL N_\alpha)$,
and the complex of left $A$\+modules $\tau_{\ge -n}(L^\bu\ot_B
\varinjlim_\alpha Q^\bu_\alpha)$ representing the derived category
object $\tau_{\ge -n}(L^\bu\ot_B^\boL \varinjlim_\alpha N_\alpha)$
is the inductive limit of this diagram.
 Denote by ${}'\!L^\bu$ a bounded above complex of finitely generated
projective left $A$\+modules endowed with a quasi-isomorphism of
complexes of left $A$\+modules ${}'\!L^\bu\rarrow L^\bu$.
 It remains to observe the functor $\Hom_A({}'\!L^\bu,{-})$ preserves
inductive limits of diagrams of uniformly bounded below complexes of
left $A$\+modules.

 To prove part~(d), we use the existence of an (also nonadditive)
functor assigning to a module its injective resolution.
 Given a filtered diagram of left $A$\+modules $(M_\alpha)$, this allows
to produce a filtered diagram of nonnegatively cohomologically graded
complexes of injective left $A$\+modules $(J^\bu_\alpha)$ endowed with
a quasi-isomorphism of diagrams of complexes of left $A$\+modules
$M_\alpha\rarrow J^\bu_\alpha$.
 By Lemma~\ref{sfp-injective-flat-modules-are-closed-under}(a),
the inductive limit $\varinjlim_\alpha J^\bu_\alpha$ is a complex of
sfp\+injective left $A$\+modules (but we will not need to use this fact).
 Let $H^\bu$ be a nonnegatively cohomologically graded complex of
injective left $A$\+modules endowed with a quasi-isomorphism of
complexes of left $A$\+modules $\varinjlim_\alpha J^\bu_\alpha
\rarrow H^\bu$.
 Then $H^\bu$ is an injective resolution of the left $A$\+module
$\varinjlim_\alpha M_\alpha$.

 Let us show that the induced morphism of complexes of left $B$\+modules
$\varinjlim_\alpha\Hom_A(L^\bu,J^\bu_\alpha)\rarrow
\Hom_A(L^\bu,H^\bu)$ is a quasi-isomorphism.
 It is enough to check that this is a quasi-isomorphism of complexes of
abelian groups.
 As above, let ${}'\!L^\bu$ be a bounded above complex of finitely
generated projective left $A$\+modules endowed with a quasi-isomorphism
of complexes of left $A$\+modules ${}'\!L^\bu\rarrow L^\bu$.
 Then the morphisms of complexes of abelian groups
$\Hom_A(L^\bu,J^\bu_\alpha)\rarrow\Hom_A({}'\!L^\bu,J^\bu_\alpha)$
and $\Hom_A(L^\bu,H^\bu)\rarrow\Hom_A({}'\!L^\bu,H^\bu)$ are
quasi-isomorphisms, since $J^\bu_\alpha$ and $H^\bu$ are bounded
below complexes of injective left $A$\+modules.
 Hence it suffices to check that the morphism of complexes of abelian
groups $\varinjlim_\alpha\Hom_A({}'\!L^\bu,J^\bu_\alpha)\rarrow
\Hom_A({}'\!L^\bu,H^\bu)$ is a quasi-isomorphism.
{\hbadness=1400\par}

 The natural map $\varinjlim_\alpha\Hom_A({}'\!L^\bu,J^\bu_\alpha)
\rarrow\Hom_A({}'\!L^\bu,\varinjlim_\alpha J^\bu_\alpha)$ is
an isomorphism of complexes of abelian groups, since ${}'\!L^\bu$ is
a bounded above complex of finitely presented left $A$\+modules and
$J^\bu_\alpha$ are uniformly bounded below complexes of
left $A$\+modules.
 So it remains to check that applying the functor
$\Hom_A({}'\!L^\bu,{-})$ to the morphism of complexes of left
$A$\+modules $\varinjlim_\alpha J^\bu_\alpha\rarrow H^\bu$
produces a quasi-isomorphism of complexes of abelian groups.
 Let $K^\bu$ denote the cone of the latter morphism of complexes of
left $A$\+modules.
 Then $K^\bu$ is a bounded below acyclic complex of left $A$\+modules.
 Since ${}'\!L^\bu$ is a bounded above complex of projective left
$A$\+modules, the complex of abelian groups
$\Hom_A({}'\!L^\bu,K^\bu)$ is acyclic.

 This proves that the morphism of complexes of left $B$\+modules
$\varinjlim_\alpha\Hom_A(L^\bu,J^\bu_\alpha)\allowbreak\rarrow
\Hom_A(L^\bu,H^\bu)$ is a quasi-isomorphism.
 Now the complex $\Hom_A(L^\bu,J_\alpha^\bu)$ represents the derived
category object $\boR\Hom_A(L^\bu,M_\alpha)$, while the complex
$\Hom_A(L^\bu,H^\bu)$ represents the derived category object
$\boR\Hom_A(L^\bu,\varinjlim_\alpha M_\alpha)\in\sD^+(B\modl)$.
 Denote by ${}''\!L^\bu$ a bounded above compelex of flat right
$B$\+modules endowed with a quasi-isomorphism of complexes of right
$B$\+modules ${}''\!L^\bu\rarrow L^\bu$.
 It remains to observe that the functor ${}''\!L^\bu\ot_B{-}$ takes
inductive limits of complexes of left $B$\+modules to inductive limits of
complexes of abelian groups, and quasi-isomorphisms of complexes of
left $B$\+modules to quasi-isomorphisms of complexes of abelian groups.
\end{proof}

\begin{cor} \label{two-rings-closed-filtered-colimits}
\textup{(a)} The full subcategory\/ $\sF_{l_1}\subset B\modl$ is closed
under filtered inductive limits. \par
\textup{(b)} The full subcategory\/ $\sE_{l_1}\subset A\modl$ is closed
under filtered inductive limits.
\end{cor}

\begin{proof}
 Part~(a): let $(F_\alpha)$ be a filtered diagram of left $B$\+modules
$F_\alpha\in\sF_{l_1}$.
 Set $F=\varinjlim_\alpha F_\alpha$.
 Then, by the definition, $\Tor_n^B(L^\bu,F_\alpha)=0$ for $n>l_1$,
and by Lemma~\ref{ext-tor-filtered-colimit-lemma}(a) it follows that
$\Tor_n^B(L^\bu,F)=0$ for $n>l_1$.
 Hence $L^\bu\ot_B^\boL F_\alpha=
\tau_{\ge-l_1}(L^\bu\ot_B^\boL F_\alpha)$ and
$L^\bu\ot_B^\boL F=\tau_{\ge-l_1}(L^\bu\ot_B^\boL F)$.
 By the definition, the adjunction morphisms $F_\alpha\rarrow
\boR\Hom_A(L^\bu,\>L^\bu\ot_B^\boL F_\alpha)$ are isomorphisms in
$\sD(B\modl)$, and by Lemma~\ref{ext-tor-filtered-colimit-lemma}(b)
we can conclude that the adjunction morphism $F\rarrow
\boR\Hom_A(L^\bu,\>L^\bu\ot_B^\boL F)$ is an isomorphism in
$\sD(B\modl)$, too.
 The proof of part~(b) is similar, based on
Lemma~\ref{ext-tor-filtered-colimit-lemma}(c\+d).
\end{proof}

 Let $R$ be an associative ring and $\sG\subset R\modl$ be a class of
left $R$\+modules.
 Following the papers~\cite{EH,ST}, we say that $\sG$ is
a \emph{Kaplansky class} if there exists a cardinal~$\kappa$ such that
for every left $R$\+module $G\in\sG$ and every element $x\in G$
there is an $R$\+submodule $F\subset G$ with less than~$\kappa$
generators for which $F$, $G/F\in\sG$ and $x\in F$.

\begin{lem} \label{ext-tor-cardinality}
 Let $\lambda$~be an infinite cardinal not smaller than
the cardinalities of the rings $A$ and~$B$.  Then \par
\textup{(a)} for every complex $N^\bu$ of left $B$\+modules with
the cohomology modules $H^i(N^\bu)$, \ $i\in\boZ$, of the cardinalities
not exceeding~$\lambda$, the cardinalities of the left $A$\+modules
$\Tor^B_j(L^\bu,N^\bu)$, \,$j\in\boZ$, do not exceed~$\lambda$; \par
\textup{(b)} for every \emph{bounded below} complex $M^\bu$ of left
$A$\+modules with the cohomology modules $H^i(M^\bu)$, \ $i\in\boZ$,
of the cardinalities not exceeding~$\lambda$, the cardinalities of
the left $B$\+modules $\Ext^j_A(L^\bu,N^\bu)$, \,$j\in\boZ$,
do not exceed~$\lambda$.
\end{lem}

\begin{proof}
 Part~(a): representing the complex $N^\bu$ as the inductive limit of its
subcomplexes of canonical truncation $\tau_{\le n}(N^\bu)$, one can
assume $N^\bu$ to be bounded above (since the $\Tor$ commutes with
filtered inductive limits and the class of all modules of the cardinality
not exceeding~$\lambda$ is preserved by countable inductive limits).
 Then there is a bounded above complex of left $B$\+modules ${}'\!N^\bu$
with the terms ${}'\!N^i$ of the cardinality not exceeding~$\lambda$,
endowed with a quasi-isomorphism of complexes of left $B$\+modules
${}'\!N^\bu\rarrow N^\bu$.
 One can also replace $L^\bu$ by a quasi-isomorphic bounded above
complex of finitely generated projective right $B$\+modules, and
the desired assertion follows.

 Part~(b): for every given degree~$j$, the left $B$\+module
$\Ext_A^j(L^\bu,M^\bu)$ only depends on a large enough finite
truncation $\tau_{\le n}M^\bu$ of the complex~$M^\bu$ (since
the complex $L^\bu$ is bounded above).
 So one can assume $M^\bu$ to be a finite complex; and then there is
a finite complex of left $A$\+modules ${}'\!M^\bu$ with the terms
${}'\!M^i$ of the cardinality not exceeding~$\lambda$, endowed with
a quasi-isomorphism of complexes of left $A$\+modules
${}'\!M^\bu\rarrow M^\bu$.
 It remains to replace $L^\bu$ by a quasi-isomorphic bounded above
complex of finitely generated projective left $A$\+modules.
\end{proof}

 The following lemma is our version of~\cite[Proposition~3.10]{EH}.

\begin{lem} \label{auslander-bass-kaplansky}
\textup{(a)} The class of left $B$\+modules\/ $\sF_{l_1}\subset B\modl$
is a Kaplansky class. \par
\textup{(b)} The class of left $A$\+modules\/ $\sE_{l_1}\subset A\modl$
is a Kaplansky class.
\end{lem}

\begin{proof}
 Let $\lambda$ be an infinite cardinal not smaller than
the cardinalities of the rings $A$ and~$B$.
 We will show that for every left $B$\+module $G\in\sF_{l_1}$ and any
$B$\+submodule $N\subset G$ with at most~$\lambda$ generators there
exists a $B$\+submodule $F\subset G$ with at most~$\lambda$ generators
such that $F$, $G/F\in\sF_{l_1}$ and $N\subset F\subset G$.
 The class of left $A$\+modules $\sE_{l_1}$ has the same property (with
respect to the same cardinals~$\lambda$).

 Indeed, by Lemma~\ref{ext-tor-cardinality}, for any left $B$\+module
$N$ of the cardinality not exceeding~$\lambda$ and every $i$,
$n\in\boZ$, the cardinalities of the modules $\Tor_i^B(L^\bu,N)$ do not
exceed~$\lambda$; and it follows that the cardinalities of the modules
$\Ext^i_A(L^\bu,\tau_{\ge-n}(L^\bu\ot_B^\boL N))$ do not
exceed~$\lambda$, either.
 (Similarly, for any left $A$\+module $M$ of the cardinality
not exceeding~$\lambda$, the cardinalities of the modules
$\Ext_A^i(L^\bu,M)$ do not exceed~$\lambda$, and consequently
the cardinalities of the modules $\Tor_i^B(L^\bu,\boR\Hom_A(L^\bu,M))$
do not exceed~$\lambda$ as well.)

 Furthermore,  for any left $B$\+module $G\in\sF_{l_1}$ and
a submodule $N\subset G$ of the cardinality not exceeding~$\lambda$,
from the homology long exact sequence induced by the short exact
sequence of left $B$\+modules $0\rarrow N\rarrow G\rarrow G/N\rarrow0$
one observes that the cardinalities of the modules $\Tor^i_B(L^\bu,G/N)$
do not exceed~$\lambda$ for all $i>l_1$.
 The same applies to the cardinalities of the modules $\Ext^i_A(L^\bu,
\tau_{\ge-l_1}(L^\bu\ot_B^\boL G/N))$, \ $i\in\boZ\setminus\{0\}$,
and the cardinalities of the kernel and the cokernel of the natural map
$G/N\rarrow\Ext^0_A(L^\bu,\tau_{\ge-l_1}(L^\bu\ot_B^\boL G/N))$.

 Let $N_\alpha$ denote the filtered diagram of all $B$\+submodules
$N\subset N_\alpha\subset G$ such that the left $B$\+module
$N_\alpha/N$ is finitely generated.
 Then $G=\varinjlim_\alpha N_\alpha$.
 By Lemma~\ref{ext-tor-filtered-colimit-lemma}(a), for every $n>l_1$
one has $\varinjlim_\alpha\Tor_n^B(L^\bu,N_\alpha)=
\Tor_n^B(L^\bu,G)=0$.
 Hence for every element $\xi\in\Tor_n^B(L^\bu,N)$ there
exists an index $\alpha=\alpha(\xi)$ such that the image of~$\xi$ in
$\Tor_n^B(L^\bu,N_\alpha)$ vanishes.
 Similarly, one has $\varinjlim_\alpha\Tor_n^B(L^\bu,G/N_\alpha)=
\Tor_n^B(L^\bu,0)=0$.
 Hence for every element $\xi'\in\Tor_n^B(L^\bu,G/N)$, \
$n>l_1$, there exists $\alpha=\alpha(\xi')$ such that the image
of~$\xi'$ in $\Tor_n^B(L^\bu,G/N_\alpha)$ vanishes.

 In the same fashion, using
Lemma~\ref{ext-tor-filtered-colimit-lemma}(b), one finds
an index~$\alpha(\eta)$ for every element~$\eta$ of
$\Ext^i_A(L^\bu,\tau_{\ge-l_1}(L^\bu\ot_B^\boL N))$,
\ $i\in\boZ\setminus\{0\}$, an index~$\alpha(\eta')$ for every
element~$\eta'$ of $\Ext^i_A(L^\bu,
\tau_{\ge-l_1}(L^\bu\ot_B^\boL G/N))$,
\ $i\in\boZ\setminus\{0\}$, an index~$\alpha(\zeta)$ for every
element~$\zeta$ of the kernel or cokernel of the map
$N\rarrow\Ext^0_A(L^\bu,\tau_{\ge-l_1}(L^\bu\ot_B^\boL N))$,
and an index~$\alpha(\zeta')$ for every element~$\zeta'$ of
the kernel or cokernel of the map $G/N\rarrow
\Ext^0_A(L^\bu,\tau_{\ge-l_1}(L^\bu\ot_B^\boL G/N))$.
 Set $N'\subset G$ to be the sum of all the submodules
$N_\alpha\subset G$ over the chosen indices $\alpha=\alpha(\xi)$,
$\alpha(\xi')$, $\alpha(\eta)$, etc.
 The cardinality of the set of all chosen indices does not
exceed~$\lambda$, hence the cardinality of left $B$\+module $N'$
does not exceed~$\lambda$, either.

 By construction, the map $\Tor_n^B(L^\bu,N)\rarrow\Tor_n^B(L^\bu,N')$
induced by the embedding of left $B$\+modules $N\rarrow N'$
vanishes for all $n>l_1$, and so does the map $\Tor_n^B(L^\bu,G/N)
\rarrow\Tor_n^B(L^\bu,G/N')$, the map
$\Ext^i_A(L^\bu,\tau_{\ge-l_1}(L^\bu\ot_B^\boL N))\rarrow
\Ext^i_A(L^\bu,\tau_{\ge-l_1}(L^\bu\ot_B^\boL N'))$ for
$i\in\boZ\setminus\{0\}$, etc.
 It remains to iterate our construction over the well-ordered set of
nonnegative integers, producing an increasing chain of $B$\+submodules
$N\subset N'\subset N''\subset\dotsb\subset G$, and put
$F=\bigcup_{m\ge0}N^{(m)}\subset G$.
 By Lemma~\ref{ext-tor-filtered-colimit-lemma}(a\+b), we have
$F\in\sF_{l_1}$ and $G/F\in\sF_{l_1}$.
 This finishes the proof of part~(a), and the proof of part~(b)
is similar.
\end{proof}

 Let $R$ be an associative ring and $\sS\subset R\modl$ be a class of
left $R$\+modules.
 A left $R$\+module $G$ is said to be \emph{$\sS$\+filtered} if there
exists a ordinal~$\alpha$ and an increasing chain of $R$\+submodules
$G_i\subset G$ indexed by $0\le i\le\alpha$, such that $G_0=0$, \
$G_\alpha=G$, \ $G_j=\bigcup_{i<j} G_i$ for all limit ordinals
$j\le\alpha$, and the quotient module $G_{i+1}/G_i$ is isomorphic to
a left $R$\+module from $\sS$ for every $0\le i<\alpha$.
 A class of left $R$\+modules $\sG\subset R\modl$ is said to be
\emph{deconstructible} if there is a set (rather than a proper class) of
left $R$\+modules $\sS$ such that $\sG$ consists precisely of all
the $\sS$\+filtered left $R$\+modules.

\begin{cor} \label{auslander-bass-deconstructible}
\textup{(a)} The class of left $B$\+modules\/ $\sF_{l_1}\subset B\modl$
is deconstructible. \par
\textup{(b)} The class of left $A$\+modules\/ $\sE_{l_1}\subset A\modl$
is deconstructible.
\end{cor}

\begin{proof}
 It is not difficult to see that any Kaplansky class of modules $\sG$
closed under extensions and filtered inductive limits is deconstructible
(one take $\sS$ to be a set of representatives of the isomorphism
classes of left $R$\+modules from $\sG$ with less than~$\kappa$
generators).
 Conversely, any deconstructible class is Kaplansky
(see~\cite[Lemmas~6.7 and~6.9]{HT} or~\cite[Lemma~2.5]{ST}).
 Thus the assertions of the corollary follow from
Lemmas~\ref{two-rings-E-F-closed-kernels-cokernels},
\ref{two-rings-closed-filtered-colimits},
and~\ref{auslander-bass-kaplansky}.
\end{proof}

 In view of the Eklof--Trlifaj theorem~\cite[Theorem~10]{ET}, it follows
from Lemma~\ref{two-rings-E-F-closed-kernels-cokernels}(b) and
Corollary~\ref{auslander-bass-deconstructible}(a) that the Auslander
class $\sF_{l_1}(L^\bu)\subset B\modl$ is the left-hand part of
a \emph{hereditary complete cotorsion pair} in $B\modl$
(cf.~\cite[Theorem~3.11]{EH}).
 It would be interesting to know whether the Bass class
$\sE_{l_1}(L^\bu)\subset A\modl$ is always the \emph{right-hand} part of
a hereditary complete cotorsion pair (cf.\ the recent paper~\cite{SS},
where the authors prove that Gorenstein injective modules over
an arbitrary ring form the right-hand part of a hereditary
complete cotorsion pair).

\Section{Derived Deconstructible Classes have Hom Sets}
\label{hom-sets-secn}

 The notions of a Kaplansky class and a deconstructible class of objects
are applicable not only to the categories of modules, but more generally
to Grothendieck abelian categories~\cite{St0} (and even more generally
to locally presentable abelian categories~\cite[Section~4]{PR}).
 For the purposes of this section, it is important that one can speak
about Kaplansky or deconstructible classes of \emph{complexes
of modules}.

 Specifically, let $\sG\subset R\modl$ be a deconstructible class of
modules (e.~g., a Kaplansky class of modules closed under extensions
and filtered inductive limits).
 Then the full subcategory $\sG$ inherits the exact category structure
from the abelian category $R\modl$, so one can speak about exact
sequences in~$\sG$.
 It is important for us that the class of all exact sequences in $\sG$
is a Kaplansky class in the abelian category $\sC(R\modl)$ of complexes
of left $R$\+modules.
 The following lemma provides a precise formulation suitable for our
purposes.

\begin{lem} \label{locally-small-localization-lemma}
 Let\/ $\sG\subset R\modl$ be a deconstructible class of
left $R$\+modules.
 Then there is a proper class of cardinals~$\kappa$ with the following
property.
 For every exact complex $G^\bu$ in\/ $\sG$ and any subcomplex of left
$R$\+modules $N^\bu\subset G^\bu$ whose every term $N^i$, \,$i\in\boZ$,
is an $R$\+module with less than~$\kappa$ generators, there exists
a subcomplex of left $R$\+modules $F^\bu\subset G^\bu$ for which
$N^\bu\subset F^\bu\subset G^\bu$, both $F^\bu$ and $G^\bu/F^\bu$
are exact complexes in\/ $\sG$, and every term $F^i$, \,$i\in\boZ$, of
the complex $F^\bu$ is an $R$\+module with less than~$\kappa$ generators.
\end{lem}

\begin{proof}
 This is a particular case of~\cite[Theorem~4.2(2)]{St0}.
\end{proof}

 For the terminological discussion of ``existence of Hom sets in Verdier
quotient categories'', see Section~\ref{introd-hom-sets} in
the introduction.

\begin{thm} \label{derived-deconstructible-class-has-hom-sets}
 Let $R$ be an associative ring and\/ $\sG\subset R\modl$ be
a deconstructible class of left $R$\+modules.
 Then, for any conventional derived category symbol\/ $\star=\b$, $+$,
$-$, or\/~$\varnothing$, the derived category\/ $\sD^\star(\sG)$
has Hom sets.
\end{thm}

\begin{proof}
 The derived category $\sD^\star(\sG)$ is obtained from the (similarly
bounded or unbounded) homotopy category $\Hot^\star(\sG)$ by
inverting all the morphisms whose cones are homotopy equivalent to
exact complexes in~$\sG$ \cite{Neem0,Kel,Bueh}.
 In order to show that the resulting category has Hom sets, we check
that the class of all morphisms we are inverting is locally small in
$\Hot^\star(\sG)$ in the sense of~\cite[Set-Theoretic
Considerations~10.3.6]{Wei}.
 Indeed, given any ($\star$\+bounded) complex $X^\bu$ in $\sG$,
morphisms $X_1^\bu\rarrow X^\bu$ in $\Hot^\star(\sG)$ with an exact cone
are in bijection with morphisms $X^\bu\rarrow G^\bu$ in
$\Hot^\star(\sG)$ with $G^\bu$ an exact complex.
 Given a fixed complex $X^\bu$, let $\lambda$~be the maximal cardinality
of generator sets of its terms $X^i$, \,$i\in\boZ$.
 Let $\kappa>\lambda$ be a cardinal with the property described in
Lemma~\ref{locally-small-localization-lemma}.
 Then, for any complex $Y^\bu$ in $\sG$, it suffices, for the purposes of
constructing morphisms $X^\bu\rarrow Y^\bu$ in $\sD^\star(\sG)$,
to consider cocones $X_1^\bu\rarrow X^\bu$ of morphisms $X^\bu\rarrow
F^\bu$, where $F^\bu$ ranges over the ($\star$\+bounded) exact complexes
in $\sG$ whose terms are left $R$\+modules with
less than~$\kappa$ generators.
\end{proof}

 Let $L^\bu$ be a pseudo-dualizing complex for associative rings
$A$ and~$B$.
 As in Sections~\ref{two-rings-auslander-bass-subsecn}
and~\ref{deconstructibility-secn}, we assume that the finite complex
$L^\bu$ is situated in the cohomological degrees $-d_1\le m\le d_2$,
and choose an integer $l_1\ge d_1$.

\begin{cor} \label{lower-pseudo-derived-have-hom-sets}
 For any pseudo-dualizing complex $L^\bu$, the lower pseudo-coderived
category of left $A$\+modules\/ $\sD'_{L^\bu}(A\modl)=\sD(\sE_{l_1})$
and the lower pseudo-contraderived category of left $B$\+modules
$\sD''_{L^\bu}(B\modl)=\sD(\sF_{l_1})$ have Hom sets.
\end{cor}

\begin{proof}
 Follows from Theorem~\ref{derived-deconstructible-class-has-hom-sets}
and Corollary~\ref{auslander-bass-deconstructible}.
\end{proof}

\Section{Existence of Fully Faithful Adjoints}
\label{existence-of-adjoints-secn}

 The aim of this section is to prove existence of the adjoint functors
promised in Section~\ref{introd-adjoints}.
 We start with the adjoint functors on
the diagram~\eqref{easy-adjoints-diagram} before passing to
the ones on the diagram~\eqref{hard-adjoints-diagram}.

\begin{lem} \label{quotients-adjoints-compositions}
 Let $q\:\sD_2\rarrow\sD_1$ and $p\:\sD_1\rarrow\sD_0$ be Verdier
quotient functors between triangulated categories.
 Let $r\:\sD_0\rarrow\sD_2$ be a right adjoint functor to the composition
$pq\:\sD_2\rarrow\sD_0$.
 Then the composition $qr\:\sD_0\rarrow\sD_1$ is a right adjoint functor
to the functor $p\:\sD_1\rarrow\sD_0$.
\end{lem}

\begin{proof}
 Let $\sK^0\subset\sD_2$ denote the kernel of the functor~$pq$ and
$\sK^1\subset\sD_2$ the kernel of the functor~$q$; so
$\sK^1\subset\sK^0$.
 Then $\sK^0/\sK^1\subset\sD_2/\sK^1=\sD_1$ is the kernel of
the functor~$p$.
 The functor $r\:\sD_0\rarrow\sD_2$ is fully faithful (as an adjoint
to a Verdier quotient functor) and its essential image $r(\sD_0)$
together with the full subcategory $\sK^0$ form a semi-orthogonal
decomposition of the triangulated category~$\sD_2$.
 The composition $(pq)r$ is the identity functor $\sD_0\rarrow\sD_0$.

 Passing to the quotient category by the full subcategory $\sK^1
\subset\sK^0\subset\sD_2$, we conclude that the functor~$qr$ is fully
faithful, the two full subcategories $\sK^0/\sK^1$ and $qr(\sD_0)$
form a semi-orthogonal decomposition of the triangulated
category~$\sD_1$, and the composition $p(qr)$ is the identity functor.
 Thus the functor~$qr$ is right adjoint to~$p$.
\end{proof}

 Now we can construct the adjoint functors shown on
the diagram~\eqref{easy-adjoints-diagram}.
 More generally, let $\sE\subset A\modl$ and $\sF\subset B\modl$ be
a pair of full subcategories satisfying the conditions~(I\+-IV) of
Section~\ref{two-rings-abstract-classes} for a given pseudo-dualizing
complex $L^\bu$ and two integers $l_1$ and~$l_2$.
 Assume that the full subcategory $\sE$ is closed under infinite
direct sums in $A\modl$, while the full subcategory $\sF$ is closed
under infinite products in $B\modl$.
 Then there exist fully faithful adjoint functors shown by curvilinear
arrows on the diagram
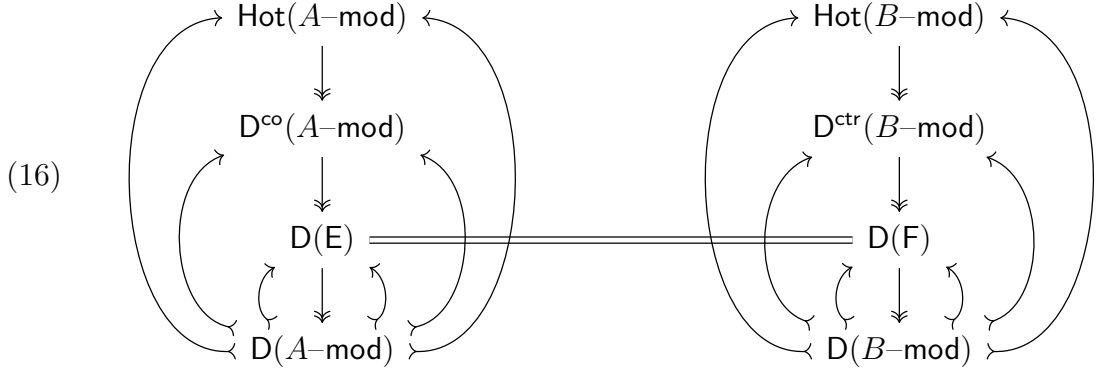
\begin{equation} \label{abstract-easy-adjoints}
\begin{tikzcd}
\Hot(A\modl) \arrow[d, two heads] &&&&&
\Hot(B\modl) \arrow[d, two heads] \\
\sD^\co(A\modl) \arrow[d, two heads] &&&&&
\sD^\ctr(B\modl) \arrow[d, two heads] \\
\sD(\sE) \arrow[d, two heads]
\arrow[rrrrr, Leftrightarrow, no head, no tail] &&&&&
\sD(\sF) \arrow[d, two heads] \\
\sD(A\modl)
\arrow[uuu, tail, bend left=90] \arrow[uuu, tail, bend right=90]
\arrow[uu, tail, bend left=75] \arrow[uu, tail, bend right=75]
\arrow[u, tail, bend left=60] \arrow[u, tail, bend right=60]
&&&&& \sD(B\modl)
\arrow[uuu, tail, bend left=90] \arrow[uuu, tail, bend right=90]
\arrow[uu, tail, bend left=72] \arrow[uu, tail, bend right=72]
\arrow[u, tail, bend left=60] \arrow[u, tail, bend right=60]
\end{tikzcd}
\end{equation}

 Indeed, for any associative ring $R$ the natural Verdier quotient
functor $K\:\Hot(R\modl)\rarrow\sD(R\modl)$ has a right and
a left adjoint.
 The fully faithful left adjoint functor $K_\lambda\:\sD(R\modl)\rarrow
\Hot(R\modl)$ to the functor~$K$ assigns to a complex of left
$R$\+modules its homotopy projective resolution.
 The fully faithful right adjoint functor $K_\rho\:\sD(R\modl)\rarrow
\Hot(R\modl)$ to the functor~$K$ assigns to a complex of right
$R$\+modules its homotopy injective resolution.
{\emergencystretch=3em\hbadness=2325\par}

 Lemma~\ref{quotients-adjoints-compositions} tells that the other
fully faithful adjoint functors on
the diagram~\eqref{abstract-easy-adjoints} can be obtained as
the compositions of the functors $K_\lambda$ and $K_\rho$ with
the Verdier quotient functors on the diagram.
 We refer to the discussion in Section~\ref{pseudo-derived-introd},
based on~\cite[Proposition~A.3.1(b)]{Pcosh},
for the constructions of the Verdier quotient functors
$\sD^\co(A\modl)\rarrow\sD(\sE)$ and $\sD^\ctr(B\modl)
\rarrow\sD(\sF)$.

 In particular, the right adjoint functor $P_\rho\:\sD(A\modl)\rarrow
\sD(\sE)$ to the Verdier quotient functor $P\:\sD(\sE)\rarrow
\sD(A\modl)$ assigns to a complex of left $A$\+modules $M^\bu$
a homotopy injective complex of injective left $A$\+modules $J^\bu$
quasi-isomorphic to $M^\bu$, viewed as an object
$P_\rho(M^\bu)=J^\bu\in\sD(\sE)$ of the conventional derived category
of the exact category~$\sE$.
 Here it is important that, according to the condition~(I), all
the injective left $A$\+modules belong to~$\sE$.

 To construct the image $P_\lambda(M^\bu)$ of the complex $M^\bu$
under the left adjoint functor $P_\lambda\:\sD(A\modl)\rarrow\sD(\sE)$
to the functor $P$, consider a homotopy projective complex
of left $A$\+modules $G^\bu$ quasi-isomorphic to~$M^\bu$.
 Let $G^\bu\rarrow E^\bu$ be a morphism of complexes of left
$A$\+modules with a coacyclic cone acting from the complex $G^\bu$ to
a complex $E^\bu$ with the terms $E^i$ belonging to the full subcategory
$\sE\subset A\modl$.
 Then $P_\lambda(M^\bu)=E^\bu\in\sD(\sE)$.

 Similarly, the left adjoint functor $Q_\lambda\:\sD(B\modl)\rarrow
\sD(\sF)$ to the Verdier quotient functor $Q\:\sD(\sF)\rarrow
\sD(B\modl)$ assigns to a complex of left $B$\+modules $N^\bu$
a homotopy projective complex of projective left $B$\+modules $G^\bu$
quasi-isomorphic to $N^\bu$, viewed as an object
$Q_\lambda(N^\bu)=G^\bu\in\sD(\sF)$ of the conventional derived
category of the exact category~$\sF$.
 Here it is important that, according to the condition~(II), all
the projective left $B$\+modules belong to~$\sF$.

 To construct the image $Q_\rho(N^\bu)$ of the complex $N^\bu$
under the right adjoint functor $Q_\rho\:\sD(B\modl)\rarrow\sD(\sF)$
to the functor $Q$, consider a homotopy injective complex of
left $B$\+modules $J^\bu$ quasi-isomorphic to~$N^\bu$.
 Let $F^\bu\rarrow J^\bu$ be a morphism of complexes of left
$B$\+modules with a contraacyclic cone acting into the complex $J^\bu$
from a complex $F^\bu$ with the terms $F^i$ belonging to the full
subcategory $\sF\subset B\modl$.
 Then $Q_\rho(N^\bu)=F^\bu\in\sD(\sF)$.

 To end, let us prove existence of the adjoint functors on
the diagram~\eqref{hard-adjoints-diagram} from
Section~\ref{introd-adjoints}.

\begin{thm}
 Let $L^\bu$ be a pseudo-dualizing complex for associative rings $A$
and~$B$.
 Then \par
\textup{(a)} assuming that all sfp\+injective left $A$\+modules have
finite injective dimensions, the natural Verdier quotient functor\/
$\sD^\co(A\modl)\rarrow\sD'_{L^\bu}(A\modl)$ has a right adjoint; \par
\textup{(b)} assuming that the ring $A$ is left coherent and all
fp\+injective left $A$\+modules have finite injective dimensions,
the natural Verdier quotient functor\/ $\sD^\co(A\modl)\rarrow
\sD'_{L^\bu}(A\modl)$ has a left adjoint; \par
\textup{(c)} assuming that all sfp\+flat left $B$\+modules have finite
projective dimensions, the natural Verdier quotient functor\/
$\sD^\ctr(B\modl)\rarrow\sD''_{L^\bu}(B\modl)$ has a right adjoint; \par
\textup{(d)} assuming that the ring $B$ is right coherent and all
flat left $B$\+modules have finite projective dimensions, the natural
Verdier quotient functor\/ $\sD^\ctr(B\modl)\rarrow
\sD''_{L^\bu}(B\modl)$ has a left adjoint.
\end{thm}

\begin{proof}
 First of all we show that all the four exotic derived categories in
question have infinite direct sums and products, and both the Verdier
quotient functors preserve both the infinite direct sums and products.
 It is helpful to keep in mind that the Verdier quotient category of
a triangulated category with infinite direct sums by a triangulated
subcategory closed under infinite direct sums has infinite sums, and
the Verdier quotient functor in such a situation preserves infinite direct
sums~\cite[Lemma~3.2.10]{N-book}.
 Dually, the same assertions apply to infinite products.

 Assume that the finite complex $L^\bu$ is situated in
the cohomologiclal degrees $-d_1\le m\le d_2$, and choose an integer
$l_1\ge d_1$.
 By Lemma~\ref{two-rings-closed-sums-products}, both the full
subcategories $\sE_{l_1}\subset A\modl$ and $\sF_{l_1}\subset B\modl$
are closed under both the infinite direct sums and products.
 Thus $\sE_{l_1}$ and $\sF_{l_1}$ are exact categories with exact
functors of infinite direct sums and products, and it follows that
their derived categories $\sD'_{L^\bu}(A\modl)=\sD(\sE_{l_1})$
and $\sD''_{L^\bu}(B\modl)=\sD(\sF_{l_1})$ have infinite direct sums
and products.

 The coderived category $\sD^\co(A\modl)$ is the Verdier quotient
category of the homotopy category $\Hot(A\modl)$ by the full subcategory
of coacyclic complexes, which is, by the definition, closed under infinite
direct sums in $\Hot(A\modl)$.
 Hence the triangulated category $\sD^\co(A\modl)$ has infinite
direct sums.
 Similarly, the contraderived category $\sD^\ctr(B\modl)$ is the Verdier
quotient category of the homotopy category $\Hot(B\modl)$ by the full
subcategory of contraacyclic complexes, which is, by the definition, closed
under infinite products in $\Hot(B\modl)$.
 Hence the triangulated category $\sD^\ctr(B\modl)$ has infinite
products.

 The Verdier quotient functor $\sD^\co(A\modl)=\sD^\co(\sE_{l_1})
\rarrow\sD(\sE_{l_1})=\sD'_{L^\bu}(A\modl)$ preserves infinite
direct sums, since $\sE_{l_1}$ is an exact category with exact functors
of infinite direct sums.
 The Verdier quotient functor $\sD^\ctr(B\modl)=\sD^\ctr(\sF_{l_1})
\rarrow\sD(\sF_{l_1})=\sD''_{L^\bu}(B\modl)$ preserves infinite
products, since $\sF_{l_1}$ is an exact category with exact fuctors
of infinite products.

 Assuming that all sfp\+injective left $A$\+modules have finite
injective dimensions, the triangulated functor $\Hot(A\modl_\inj)
\rarrow\sD^\co(A\modl)$ induced by the inclusion
$A\modl_\inj\rarrow A\modl$ is an equivalence by
Proposition~\ref{sfp-inj-homotopy-coderived-comparison}(b).
 The additive category of injective left $A$\+modules $A\modl_\inj$
has infinite products, hence so does its homotopy category.
 Similarly, assuming that all sfp\+flat left $B$\+modules have finite
projective dimensions, the triangulated functor $\Hot(B\modl_\proj)
\rarrow\sD^\ctr(B\modl)$ induced by the inclusion
$B\modl_\proj\rarrow B\modl$ is an equivalence by
Proposition~\ref{sfp-flat-homotopy-contraderived-comparison}(b).
 The additive category of projective left $B$\+modules $B\modl_\proj$
has infinite direct sums, hence so does its homotopy category.

 The functor $\sD^\co(A\modl)=\Hot(A\modl_\inj)\rarrow\sD(\sE_{l_1})$
preserves infinite products, since $\sE_{l_1}$ is an exact category with
exact functors of infinite products and $A\modl_\inj\subset\sE_{l_1}$
is a split exact subcategory closed under infinite products.
 The functor $\sD^\ctr(B\modl)=\Hot(B\modl_\proj)\rarrow\sD(\sF_{l_1})$
preserves infinite direct sums, since $\sF_{l_1}$ is an exact category with
exact functors of infinite direct sums and $B\modl_\proj\subset\sF_{l_1}$
is a split exact subcategory closed under infinite direct sums.

 For the rest of our argument, it is important that the lower
pseudo-derived categories $\sD'_{L^\bu}(A\modl)$ and
$\sD''_{L^\bu}(B\modl)$ have Hom sets by
Corollary~\ref{lower-pseudo-derived-have-hom-sets}.

 By~\cite[Theorem~5.9]{Neem}, for any associative ring $B$
the homotopy category of projective modules $\Hot(B\modl_\proj)$
is \emph{well-generated} in the sense of~\cite{N-book,Kra}.
 By~\cite[Theorem~3.13]{Neem2}, for any associative ring $A$
the homotopy category of injective modules $\Hot(A\modl_\inj)$
is well-generated.
 According to the Brown representability theorem for well-generated
triangulated categories (see~\cite[Proposition~8.4.2]{N-book}
or~\cite[Theorem~5.1.1]{Kra2}), a triangulated functor from
a well-generated triangulated category to a triangulated category with
small Hom sets has a right adjoint if and only if it preserves infinite
direct sums.
 This proves parts~(a) and~(c).

 By~\cite[Proposition~7.14]{Neem}, for any right coherent ring $B$
the homotopy category of projective left $B$\+modules
$\Hot(B\modl_\proj)$ is compactly generated.
 By~\cite[Corollary~6.13]{St2} (see also~\cite[Corollary~2.6(b)]{Pfp}),
for any left coherent ring $A$ the homotopy category of injective
left $A$\+modules $\Hot(A\modl_\inj)$ is compactly generated.
 According to the covariant Brown representability theorem for
compactly generated triangulated categories
(see~\cite[Theorem~8.6.1]{N-book}, \cite[Section~2]{Kra1},
or~\cite[Proposition~5.3.1(2)]{Kra2}), a triangulated functor from
a compactly triangulated category to a triangulated category with
small Hom sets has a left adjoint if and only if it preserves infinite
products.
 This proves parts~(b) and~(d).
\end{proof}

\appendix
\bigskip
\section*{Appendix.  Derived Functors of Finite Homological
Dimension~II}
\medskip
\setcounter{section}{1}
\setcounter{thm}{0}

 The aim of this appendix is to work out a generalization of
the constructions of~\cite[Appendix~B]{Pmgm} that is needed for
the purposes of the present paper.
 We use an idea borrowed from~\cite[Appendix~A]{Ef} in order to
simplify and clarify the exposition.

\subsection{Posing the problem} \label{posing-problem-appx}
 First we need to recall some notation from~\cite{Pmgm}.
 Given an additive category $\sA$, we denote by $\sC^+(\sA)$
the category of bounded below complexes in $\sA$, viewed either
as a DG\+category (with complexes of morphisms), or simply as
an additive category, with closed morphisms of degree~$0$.
 When $\sA$ is an exact category, the full subcategory $\sC^{\ge0}(\sA)
\subset\sC^+(\sA)$ of nonnegatively cohomologically graded complexes
in $\sA$ and closed morphisms of degree~$0$ between them has
a natural exact category structure, with termwise exact short exact
sequences of complexes.

 Let $\sE$ be an exact category and $\sJ\subset\sE$ be a coresolving
subcategory (in the sense of Section~\ref{pseudo-derived-introd}),
endowed with the exact category structure inherited from~$\sE$.
 As it was pointed out in~\cite{Pmgm}, a closed morphism in
$\sC^+(\sJ)$ is a quasi-isomorphism of complexes in $\sJ$ if and only
if it is a quasi-isomorphism of complexes in~$\sE$.
 A short sequence in $\sC^{\ge0}(\sJ)$ is exact in $\sC^{\ge0}(\sJ)$ if
and only if it is exact in $\sC^{\ge0}(\sE)$.

 Modifying slightly the notation in~\cite{Pmgm}, we denote by
${}_\sE\sC^{\ge0}(\sJ)$ the full subcategory in the exact category
$\sC^{\ge0}(\sJ)$ consisting of all the complexes $0\rarrow J^0\rarrow
J^1\rarrow J^2\rarrow\dotsb$ in $\sJ$ for which there exists an object
$E\in\sE$ together with a morphism $E\rarrow J^0$ such that the sequence
$0\rarrow E\rarrow J^0\rarrow J^1\rarrow\dotsb$ is exact in~$\sE$.
 By the definition, one has ${}_\sE\sC^{\ge0}(\sJ)=\sC^{\ge0}(\sJ)\cap
{}_\sE\sC^{\ge0}(\sE)\subset\sC^{\ge0}(\sE)$.
 The full subcategory ${}_\sE\sC^{\ge0}(\sJ)$ is closed under extensions
and the cokernels of admissible monomorphisms in $\sC^{\ge0}(\sJ)$;
so it inherits an exact category structure.

\medskip

 Let $\sB$ be another exact category and $\sF\subset\sB$ be
a resolving subcategory.
 We will suppose that the additive category $\sB$ contains the images
of idempotent endomorphisms of its objects.
 Let $-l_2\le l_1$ be two integers.
 Denote by $\sC^{\ge-l_2}(\sB)$ the exact category
$\sC^{\ge 0}(\sB)[l_2]\subset\sC^+(\sB)$ of complexes in $\sB$
concentrated in the cohomological degrees~$\ge-l_2$, and
by $\sC^{\ge -l_2}(\sB)^{\le l_1}\subset\sC^{\ge -l_2}(\sB)$ the full
subcategory consisting of all complexes $0\rarrow B^{-l_2}\rarrow
\dotsb\rarrow B^{l_1}\rarrow\dotsb$ such that the sequence $B^{l_1}
\rarrow B^{l_1+1}\rarrow B^{l_1+2}\rarrow\dotsb$ is exact in~$\sB$.
 Furthermore, let $\sC_\sF^{\ge-l_2}(\sB)^{\le l_1}\subset
\sC^{\ge -l_2}(\sB)^{\le l_1}$ be the full subcategory of all complexes
that are isomorphic in the derived category $\sD(\sB)$ to complexes
of the form $0\rarrow F^{-l_2}\rarrow\dotsb\rarrow F^{l_1}\rarrow0$,
with the terms belonging to $\sF$ and concentrated in the cohomological
degrees $-l_2\le m\le l_1$.

 For example, one has $\sC^{\ge0}(\sB)^{\le 0}={}_\sB\sC^{\ge0}(\sB)$.
 The full subcategory $\sC^{\ge-l_2}(\sB)^{\le l_1}$ is closed under
extensions and the cokernels of admissible monomorphisms in
the exact category $\sC^{\ge-l_2}(\sB)$, while (essentially
by~\cite[Proposition~2.3(2)]{St} or~\cite[Lemma~A.5.4(a\+b)]{Pcosh})
the full subcategory $\sC_\sF^{\ge-l_2}(\sB)^{\le l_1}$ is closed
under extensions and the kernels of admissible epimorphisms in
$\sC^{\ge-l_2}(\sB)^{\le l_1}$.
 So the full subcategory $\sC_\sF^{\ge-l_2}(\sB)^{\le l_1}$ inherits
an exact category structure from $\sC^{\ge-l_2}(\sB)$.

\medskip

 Suppose that we are given a DG\+functor $\Psi\:\sC^+(\sJ)\rarrow
\sC^+(\sB)$ taking acyclic complexes in the exact category $\sJ$ to
acyclic complexes in the exact category~$\sB$.
 Suppose further that the restriction of $\Psi$ to the subcategory
${}_\sE\sC^{\ge0}(\sJ)\subset\sC^+(\sJ)$ is an exact functor between
exact categories
\begin{equation} \label{psi-exact-between-exact-cat}
 \Psi\:{}_\sE\sC^{\ge0}(\sJ)\lrarrow\sC_\sF^{\ge-l_2}(\sB)^{\le l_1}.
\end{equation}
 Our aim is to construct the right derived functor
\begin{equation} \label{R-Psi-D-star}
 \boR\Psi\:\sD^\star(\sE)\lrarrow\sD^\star(\sF)
\end{equation}
acting between any bounded or unbounded, conventional or absolute
derived categories $\sD^\star$ with the symbols $\star=\b$, $+$, $-$,
$\varnothing$, $\abs+$, $\abs-$, or~$\abs$.

 Under certain conditions, one can also have the derived functor
$\boR\Psi$ acting between the coderived or contraderived categories,
$\star=\co$ or~$\ctr$, of the exact categories $\sE$ and~$\sF$.
 When the exact categories $\sE$ and $\sB$ have exact functors of
infinite product, the full subcategories $\sJ\subset\sE$
and $\sF\subset\sB$ are closed under infinite products, and
the functor $\Psi$ preserves infinite products, there will be
the derived functor $\boR\Psi$ acting between the contraderived
categories, $\boR\Psi\:\sD^\ctr(\sE)\rarrow\sD^\ctr(\sF)$.

 When the exact categories $\sE$ and $\sB$ have exact functors of
infinite direct sum, the full subcategory $\sF\subset\sB$ is closed
under infinite direct sums, and for any family of complexes
$J^\bu_\alpha\in\sC^{\ge0}(\sJ)$ and a complex $I^\bu\in\sC^{\ge0}(\sJ)$
endowed with a quasi-isomorphism $\bigoplus_\alpha J^\bu_\alpha\rarrow
I^\bu$ of complexes in the exact category $\sE$, the induced morphism
$$
 \bigoplus\nolimits_\alpha\Psi(J^\bu_\alpha)\lrarrow\Psi(I^\bu)
$$
is a quasi-isomorphism of complexes in the exact category $\sB$, there
will be the derived functor $\boR\Psi$ acting between
the coderived categories, $\boR\Psi\:\sD^\co(\sE)\rarrow\sD^\co(\sF)$.

 The construction of the derived functor $\boR\Psi$
in~\cite[Appendix~B]{Pmgm} is the particular case of the construction
below corresponding to the situation with $\sF=\sB$.

\subsection{The construction of derived functor}
\label{derived-functor-construction-appx}
 The following construction of the derived
functor~\eqref{R-Psi-D-star} is based on a version of
the result of~\cite[Proposition~A.3]{Ef}.

 Since the DG\+functor $\Psi\:\sC^+(\sJ)\rarrow\sC^+(\sB)$ preserves
quasi-isomorphisms, it induces a triangulated functor
$$
 \Psi\:\sD^+(\sJ)\lrarrow\sD^+(\sB).
$$
 Taking into account the triangulated equivalence $\sD^+(\sJ)\simeq
\sD^+(\sE)$ (provided by the dual version
of~\cite[Proposition~A.3.1(a)]{Pcosh}), we obtain the derived functor
$$
 \boR\Psi\:\sD^+(\sE)\lrarrow\sD^+(\sB).
$$
 Now our assumptions on $\Psi$ imply that the functor $\boR\Psi$ takes
the full subcategory $\sD^\b(\sE)\subset\sD^+(\sE)$ into the full
subcategory $\sD^\b(\sF)\subset\sD^\b(\sB)\subset\sD^+(\sB)$; hence
the triangulated functor
\begin{equation} \label{R-Psi-D-b}
 \boR\Psi\:\sD^\b(\sE)\lrarrow\sD^\b(\sF).
\end{equation}

 For any exact category $\sA$, we denote by $\sC(\sA)$ the exact
category of unbounded complexes in $\sA$, with termwise exact short
exact sequences of complexes.
 In order to construct the derived functor $\boR\Psi$ for the derived
categories with the symbols other than $\star=\b$, we are going
to substitute into~\eqref{R-Psi-D-b} the exact category $\sC(\sE)$
in place of $\sE$ and the exact category $\sC(\sF)$ in place of~$\sF$.

 For any category $\Gamma$ and DG\+category $\DG$, there is
a DG\+category whose objects are all the functors $\Gamma\rarrow\DG$
taking morphisms in $\Gamma$ to closed morphisms of degree~$0$
in~$\DG$, and whose complexes of morphisms are constructed as
the complexes of morphisms of functors.
 We denote this DG\+category by $\DG^\Gamma$.
 So diagrams of any fixed shape in a given DG\+category form
a DG\+category.
 Given a DG\+functor $F\:{}'\DG\rarrow{}''\DG$, there is the induced
DG\+functor between the categories of diagrams $F^\Gamma\:{}'\DG^\Gamma
\rarrow{}''\DG^\Gamma$.
 In particular, the DG\+category of complexes $\sC(\DG)$ in a given
DG\+category $\DG$ can be constructed as a full DG\+subcategory of
the DG\+category of diagrams of the corresponding shape in~$\DG$.

 Applying this construction to the DG\+functor $\Psi$ and restricting
to the full DG\+subcategories of bicomplexes that are uniformly
bounded on the relevant side, we obtain a DG\+functor
$$
 \Psi_\sC\:\sC^+(\sC(\sJ))\lrarrow\sC^+(\sC(\sB)).
$$
 Here the categories of unbounded complexes $\sC(\sJ)$ and $\sC(\sB)$
are simply viewed as additive/exact categories of complexes and
closed morphisms of degree~$0$ between them.
 The DG\+structures come from the differentials raising the degree
in which the bicomplexes are bounded below.

 The functor $\Psi_\sC$ takes acyclic complexes in the exact category
$\sC(\sJ)$ to acyclic complexes in the exact category~$\sC(\sB)$.
 In view of the standard properties of the resolution
dimension~\cite[Corollary~A.5.2]{Pcosh}, the functor $\Psi_\sC$ takes
the full subcategory ${}_{\sC(\sE)}\sC^{\ge0}(\sC(\sJ))\subset
\sC^+(\sC(\sJ))$ into the full subcategory
$\sC_{\sC(\sF)}^{\ge-l_2}(\sC(\sB))^{\le l_1}\subset\sC^+(\sC(\sB))$,
$$
 \Psi_\sC\:{}_{\sC(\sE)}\sC^{\ge0}(\sC(\sJ))\lrarrow
 \sC_{\sC(\sF)}^{\ge-l_2}(\sC(\sB))^{\le l_1}.
$$
 Finally, the functor $\Psi_\sC$ is exact in restriction to
the exact category ${}_{\sC(\sE)}\sC^{\ge0}(\sC(\sJ))$, since the functor
$\Psi$ is exact in restriction to the exact category
${}_\sE\sC^{\ge0}(\sJ)$.

 Applying the construction of the derived
functor~\eqref{R-Psi-D-b} to the DG\+functor $\Psi_\sC$ in place
of~$\Psi$, we obtain a triangulated functor
\begin{equation} \label{R-Psi-C-D-b}
 \boR\Psi_\sC\:\sD^\b(\sC(\sE))\lrarrow\sD^\b(\sC(\sF)). 
\end{equation}
 Similarly one can construct the derived functors
$\boR\Psi_{\sC^{\le0}}\:\sD^\b(\sC^{\le0}(\sE))\rarrow
\sD^\b(\sC^{\le0}(\sF))$ and $\boR\Psi_{\sC^{\ge0}}\:
\sD^\b(\sC^{\ge0}(\sE))\rarrow\sD^\b(\sC^{\ge0}(\sF))$ acting between
the bounded derived categories of the exact categories of nonpositively
or nonnegatively cohomologically graded complexes.
 Shifting and passing to the direct limits of fully faithful embeddings,
one can obtain the derived functors 
$\boR\Psi_{\sC^-}\:\sD^\b(\sC^-(\sE))\rarrow\sD^\b(\sC^-(\sF))$ and
$\boR\Psi_{\sC^+}\:\sD^\b(\sC^+(\sE))\rarrow\sD^\b(\sC^+(\sF))$ acting
between the bounded derived categories of the exact categories of
bounded above or bounded below complexes, etc.

\medskip

 In order to pass from~\eqref{R-Psi-C-D-b} to~\eqref{R-Psi-D-star}
with $\star=\abs$, we will apply the following version
of~\cite[Proposition~A.3(2)]{Ef}.
 Clearly, for any exact category $\sA$ the totalization of bounded
complexes of complexes in $\sA$ is a triangulated functor
\begin{equation}\label{abs-totalization}
 \sD^\b(\sC(\sA))\lrarrow\sD^\abs(\sA).
\end{equation}

\begin{prop} \label{d-b-c-d-abs-efimov}
 For any exact category\/ $\sA$, the totalization
functor~\eqref{abs-totalization} is a Verdier quotient functor.
 Its kernel is the thick subcategory generated by the contractible
complexes in\/ $\sA$, viewed as objects of\/ $\sC(\sA)$.
\end{prop}

\begin{proof}
 Denote by $\sA_\spl$ the additive category $\sA$ endowed with
the split exact category structure (i.~e., all the short exact
sequences are split).
 Following~\cite{Ef}, one first checks the assertion of proposition
for the exact category~$\sA_\spl$.

 In this case, $\sC(\sA_\spl)$ is a Frobenius exact category
whose projective-injective objects are the contractible complexes,
and $\sD^\abs(\sA_\spl)=\Hot(\sA_\spl)$ is the stable category of the
Frobenius exact category $\sC(\sA_\spl)$.
 The quotient category of the bounded derived category
$\sD^\b(\sC(\sA_\spl))$ by the bounded homotopy category of complexes
of projective-injective objects in $\sC(\sA_\spl)$ is just another
construction of the stable category of a Frobenius exact
category, and the totalization functor is the inverse equivalence to
the comparison functor between the two constructions of
the stable category.

 Then, in order to pass from the functor~\eqref{abs-totalization}
for the exact category $\sA_\spl$ to the similar functor for the exact
category $\sA$, one takes the quotient category by the acyclic
bounded complexes of complexes on the left-hand side, transforming
$\sD^\b(\sC(\sA_\spl))$ into $\sD^\b(\sC(\sA))$, and the quotient
category by the totalizations of such bicomplexes on the right-hand
side, transforming $\Hot(\sA)$ into $\sD^\abs(\sA)$.
\end{proof}

 It remains to notice that the contractible complexes in $\sA$ are
the direct summands of the cones of identity endomorphisms of
complexes in $\sA$, and the functor~\eqref{R-Psi-C-D-b}
obviously takes the cones of identity endomorphisms of complexes
in $\sE$ (viewed as objects of $\sC(\sE)$) to bicomplexes whose
totalizations are contractible complexes in~$\sF$.
 This provides the desired derived functor~\eqref{R-Psi-D-star}
for $\star=\abs$.

 In order to pass from~\eqref{R-Psi-C-D-b} to~\eqref{R-Psi-D-star}
with $\star=\varnothing$, the following corollary of
Proposition~\ref{d-b-c-d-abs-efimov} can be applied.
 Consider the totalization functor
\begin{equation} \label{conv-derived-totalization}
 \sD^\b(\sC(\sA))\lrarrow\sD(\sA).
\end{equation}

\begin{cor} \label{d-b-c-d-conv}
 For any exact category\/ $\sA$, the totalization
functor~\eqref{conv-derived-totalization} is a Verdier quotient functor.
 Its kernel is the thick subcategory generated by the acyclic complexes
in\/ $\sA$, viewed as objects of\/ $\sC(\sA)$.  \qed
\end{cor}

 Using the condition that
the functor~\eqref{psi-exact-between-exact-cat} takes short exact
sequences to short exact sequences together
with~\cite[Lemma~B.2(e)]{Pmgm}, one shows that
the functor~\eqref{R-Psi-C-D-b} takes acyclic complexes in~$\sE$ (viewed
as objects of $\sC(\sE)$) to bicomplexes with acyclic totalizations.
 This provides the derived functor~\eqref{R-Psi-D-star} for
$\star=\varnothing$.

\medskip

 To construct the derived functors $\boR\Psi$ acting between
the bounded above and bounded below versions of the conventional and
absolute derived categories (with $\star=+$, $-$, $\abs+$, or~$\abs-$),
one can notice that the functors $\boR\Psi$ for $\star=\varnothing$
or~$\abs$ take bounded above/below complexes to (objects representable
by) bounded above/below complexes, and use the fact that the embedding
functors from the bounded above/below conventional/absolute derived
categories into the unbounded ones are fully
faithful~\cite[Lemma~A.1.1]{Pcosh}.
 Alternatively, one can repeat the above arguments with the categories
of unbounded complexes $\sC(\sA)$ replaced with the bounded above/below
ones $\sC^-(\sA)$ or $\sC^+(\sA)$.
 The derived functor $\boR\Psi$ with $\star=\b$ constructed in such
a way agrees with the functor~\eqref{R-Psi-D-b}.

 To construct the derived functor $\boR\Psi$ acting between
the coderived or contraderived categories (under the respective
assumptions in Section~\ref{posing-problem-appx}), one considers
the derived functor $\boR\Psi$ for $\star=\abs$, and checks that
the kernel of the composition $\sC(\sE)\rarrow\sD^\abs(\sE)\rarrow
\sD^\abs(\sF)\rarrow\sD^\co(\sF)$ or $\sC(\sE)\rarrow\sD^\abs(\sE)\rarrow
\sD^\abs(\sF)\rarrow\sD^\ctr(\sF)$ is closed under the infinite direct
sums or infinite products, respectively.
 The facts that the kernels of the additive functors $\sC(\sF)\rarrow
\sD^{\co/\ctr}(\sF)$ are closed under the infinite direct sums/products
and the total complex of a finite acyclic complex of unbounded
complexes in $\sF$ is absolutely acyclic need to be used.

\subsection{The dual setting} \label{dual-setting-appx}
 The notation $\sC^{\le0}(\sB)\subset\sC^-(\sB)$ for an additive or
exact category $\sB$ has the similar or dual meaning to the one in
Section~\ref{posing-problem-appx}.

 Let $\sF$ be an exact category and $\sP\subset\sF$ be a resolving
subcategory, endowed with the inherited exact category structure.
 A closed morphism in $\sC^-(\sP)$ is a quasi-isomorphism of complexes
in $\sP$ if and only if it is a quasi-isomorphism of complexes in~$\sF$.
 A short sequence in $\sC^{\le0}(\sP)$ is exact in $\sC^{\le0}(\sP)$ if
and only if it is exact in~$\sC^{\le0}(\sF)$.

 Following the notation in Section~\ref{posing-problem-appx}, denote
by ${}_\sF\sC^{\le0}(\sP)$ the full subcategory in the exact category
$\sC^{\le0}(\sP)$ consisting of all the complexes
$\dotsb\rarrow P^{-2}\rarrow P^{-1}\rarrow P^0\rarrow0$ in $\sP$ for
which there exists an object $F\in\sF$ together with a morphism
$P^0\rarrow F$ such that the sequence $\dotsb\rarrow P^{-1}\rarrow
P^0\rarrow F\rarrow0$ is exact in~$\sF$.
 By the definition, one has ${}_\sF\sC^{\le0}(\sP)=\sC^{\le0}(\sP)\cap
{}_\sF\sC^{\le0}(\sF)\subset\sC^{\le0}(\sF)$.
 The full subcategory ${}_\sF\sC^{\le0}(\sP)$ is closed under extensions
and the kernels of admissible epimorphisms in $\sC^{\le0}(\sP)$; so
it inherits an exact category structure.

 Let $\sA$ be another exact category and $\sE\subset\sA$ be
a coresolving subcategory.
 Suppose that the additive category $\sA$ contains the images of
idempotent endomorphisms of its objects.
 Let $-l_1\le l_2$ be two integers.
 Denote by $\sC^{\le l_2}(\sA)$ the exact category
$\sC^{\le0}(\sA)[-l_2]\subset\sC^-(\sA)$ of complexes in $\sA$
concentrated in the cohomological degrees~$\le\nobreak l_2$, and by
$\sC^{\le l_2}(\sA)^{\ge-l_1}\subset\sC^{\le l_2}(\sA)$ the full subcategory
consisting of all complexes $\dotsb\rarrow A^{-l_1}\rarrow\dotsb\rarrow
A^{l_2}\rarrow0$ such that the sequence $\dotsb\rarrow A^{-l_1-2}\rarrow
A^{-l_1-1}\rarrow A^{-l_1}$ is exact in~$\sA$.
 Furthermore, let $\sC_\sE^{\le l_2}(\sA)^{\ge-l_1}\subset
\sC^{\le l_2}(\sA)^{\ge-l_1}$ be the full subcategory of all complexes
that are isomorphic in the derived category $\sD(\sA)$ to complexes
of the form $0\rarrow E^{-l_1}\rarrow\dotsb\rarrow E^{l_2}\rarrow0$,
with the terms belonging to~$\sE$ and concentrated in the cohomological
degrees $-l_1\le m\le l_2$.

 For example, one has $\sC^{\le0}(\sA)^{\ge0}={}_\sA\sC^{\le0}(\sA)$.
 The full subcategory $\sC^{\le l_2}(\sA)^{\ge -l_1}$ is closed under
extensions and the kernels of admissible epimorphisms in the exact
category $\sC^{\le l_2}(\sA)$, while the full subcategory
$\sC_\sE^{\le l_2}(\sA)^{\ge-l_1}$ is closed under extension and
the cokernels of admissible monomorphisms in
$\sC^{\le l_2}(\sA)^{\ge-l_1}$.
 So the full subcategory $\sC_\sE^{\le l_2}(\sA)^{\ge-l_1}$ inherits
an exact category structure from $\sC^{\le l_2}(\sA)$.

\medskip

 Suppose that we are given a DG\+functor $\Phi\:\sC^-(\sP)\rarrow
\sC^-(\sA)$ taking acyclic complexes in the exact category $\sP$ to
acyclic complexes in the exact category~$\sA$.
 Suppose further that the restriction of $\Phi$ to the subcategory
${}_\sF\sC^{\le0}(\sP)\subset\sC^-(\sP)$ is an exact functor between
exact categories
\begin{equation}
 {}_\sF\sC^{\le0}(\sP)\lrarrow \sC_\sE^{\le l_2}(\sA)^{\ge-l_1}.
\end{equation}
 Then the construction dual to that in
Section~\ref{derived-functor-construction-appx} provides
the left derived functor
\begin{equation} \label{L-Phi-D-star}
 \boL\Phi\:\sD^\star(\sF)\rarrow\sD^\star(\sE)
\end{equation}
acting between any bounded or unbounded, conventional or absolute
derived categories $\sD^\star$ with the symbols $\star=\b$, $+$, $-$,
$\varnothing$, $\abs+$, $\abs-$, or~$\abs$.

 Under certain conditions, one can also have the derived functor
$\boL\Phi$ acting between the coderived or contraderived categories.
 When the exact categories $\sF$ and $\sA$ have exact functors of
infinite direct sum, the full subcategories $\sP\subset\sF$ and
$\sE\subset\sA$ are closed under infinite direct sums, and the functor
$\Phi$ preserves infinite direct sums, there is the derived
functor $\boL\Phi\:\sD^\co(\sF)\rarrow\sD^\co(\sE)$.

 When the exact categories $\sF$ and $\sA$ have exact functors of
infinite product, the full subcategory $\sE\subset\sA$ is closed under
infinite products, and for any family of complexes $P^\bu_\alpha\in
\sC^{\le0}(\sP)$ and a complex $Q^\bu\in\sC^{\le0}(\sP)$ endowed with
a quasi-isomorphism $Q^\bu\rarrow\prod_\alpha P^\bu_\alpha$ of
complexes in the exact category $\sF$, the induced morphism
$$
 \Phi(Q^\bu)\lrarrow\prod\nolimits_\alpha\Phi(P^\bu_\alpha)
$$
is a quasi-isomorphism of complexes in the exact category $\sA$,
there is the derived functor $\boL\Phi\:\sD^\ctr(\sF)\rarrow
\sD^\ctr(\sE)$.

\medskip

 Let us spell out the major steps of the construction of
the derived functor~\eqref{L-Phi-D-star}.
 Since the DG\+functor $\Phi\:\sC^-(\sP)\rarrow\sC^-(\sA)$ preserves
quasi-isomorphisms, it induces a triangulated functor
$\Phi\:\sD^-(\sP)\rarrow\sD^-(\sA)$.
 Taking into account the triangulated equivalence $\sD^-(\sP)
\simeq\sD^-(\sF)$ provided by~\cite[Proposition~A.3.1(a)]{Pcosh}, we
obtain the derived functor $\boL\Phi\:\sD^-(\sF)\rarrow\sD^-(\sA)$.
 Our assumptions on $\Phi$ imply that this functor $\boL\Phi$
takes the full subcategory $\sD^\b(\sF)\subset\sD^-(\sF)$ into
the full subcategory $\sD^\b(\sE)\subset\sD^\b(\sA)\subset\sD^-(\sA)$;
hence the triangulated functor
\begin{equation} \label{L-Phi-D-b}
 \boL\Phi\:\sD^\b(\sF)\rarrow\sD^\b(\sE).
\end{equation}

 Passing from the DG\+functor $\Phi\:\sC^-(\sP)\rarrow\sC^-(\sA)$
to the induced DG\+functor between the DG\+categories of unbounded
complexes in the given DG\+categories, as explained in
Section~\ref{derived-functor-construction-appx}, and restricting to
the full DG\+subcategories of uniformly bounded bicomplexes, one
obtains the DG\+functor
$$
 \Phi_\sC\:\sC^-(\sC(\sP))\lrarrow\sC^-(\sC(\sA)).
$$
 The functor $\Phi_\sC$ takes acyclic complexes in the exact category
$\sC(\sP)$ to acyclic complexes in the exact category $\sC(\sA)$.
 It also takes the full subcategory ${}_{\sC(\sF)}\sC^{\le0}(\sC(\sP))
\subset\sC^-(\sC(\sP))$ into the full subcategory
$\sC_{\sC(\sE)}^{\le l_2}(\sC(\sA))^{\ge-l_1}\subset\sC^-(\sC(\sA))$.
 So we can apply the construction of the derived
functor~\eqref{L-Phi-D-b} to the DG\+functor $\Phi_\sC$ in place of
$\Phi$, and produce a triangulated functor
\begin{equation} \label{L-Phi-C-D-b}
 \boL\Phi_\sC\:\sD^\b(\sC(\sF))\rarrow\sD^\b(\sC(\sE)).
\end{equation}

 Using Proposition~\ref{d-b-c-d-abs-efimov} and
Corollary~\ref{d-b-c-d-conv}, one shows that the triangulated
functor~\eqref{L-Phi-C-D-b} descends to a triangulated
functor~\eqref{L-Phi-D-star} between the absolute or conventional
derived categories, $\star=\abs$ or~$\varnothing$.
 The cases of bounded above or below absolute or conventional
derived categories, $\star=+$, $-$, $\abs+$, or~$\abs-$ can be
treated as explained in Section~\ref{derived-functor-construction-appx}.
 Under the respective assumptions, one can also descend from
the absolute derived categories to the coderived or contraderived
categories, producing the derived functor~\eqref{L-Phi-D-star}
for $\star=\co$ or~$\ctr$.

\subsection{Deriving adjoint functors}  \label{deriving-adjoints-appx}
 Let $\sA$ and $\sB$ be exact categories containing the images of
idempotent endomorphisms of its objects, let $\sJ\subset\sE\subset\sA$
be coresolving subcategories in $\sA$, and let $\sP\subset\sF\subset\sB$
be resolving subcategories in~$\sB$.

 Let $\Psi\:\sC^+(\sJ)\rarrow\sC^+(\sB)$ be a DG\+functor satisfying
the conditions of Section~\ref{posing-problem-appx}, and let
$\Phi\:\sC^-(\sP)\rarrow\sC^-(\sA)$ be a DG\+functor satisfying
the conditions of Section~\ref{dual-setting-appx}.
 Suppose that the DG\+functors $\Phi$ and $\Psi$ are partially adjoint,
in the sense that for any two complexes $J^\bu\in\sC^+(\sJ)$ and
$P^\bu\in\sC^-(\sP)$ there is a natural isomorphism of complexes
of abelian groups
\begin{equation} \label{dg-adjunction}
 \Hom_\sA(\Phi(P^\bu),J^\bu)\simeq\Hom_\sB(P^\bu,\Psi(J^\bu)),
\end{equation}
where $\Hom_\sA$ and $\Hom_\sB$ denote the complexes of morphisms
in the DG\+categories of unbounded complexes $\sC(\sA)$ and
$\sC(\sB)$.

 Our aim is to show that the triangulated functor $\boL\Phi$
\eqref{L-Phi-D-star} is left adjoint to the triangulated functor
$\boR\Phi$ \eqref{R-Psi-D-star}, for any symbol $\star=\b$, $+$,
$-$, $\varnothing$, $\abs+$, $\abs-$, or~$\abs$.
 When the functors $\boL\Phi$ and $\boR\Psi$ acting between
the categories $\sD^\co$ or $\sD^\ctr$ are defined (i.~e., the related
conditions in Sections~\ref{posing-problem-appx}
and~\ref{dual-setting-appx} are satisfied), the former of them is
also left adjoint to the latter one.

 Our first step is the following lemma.

\begin{lem}
 In the assumptions above, the induced triangulated functors\/
$\Phi\:\sD^-(\sP)\allowbreak\rarrow\sD^-(\sA)$ and\/
$\Psi\:\sD^+(\sJ)\rarrow\sD^+(\sB)$ are partially adjoint, in
the sense that for any complexes $J^\bu\in\sC^+(\sJ)$ and
$P^\bu\in\sC^-(\sP)$ there is a natural isomorphism of abelian groups
of morphisms in the unbounded derived categories
$$
 \Hom_{\sD(\sA)}(\Phi(P^\bu),J^\bu)\simeq
 \Hom_{\sD(\sB)}(P^\bu,\Psi(J^\bu)).
$$
\end{lem}

\begin{proof}
 Passing to the cohomology groups in the DG\+adjunction
isomorphism~\eqref{dg-adjunction}, one obtains an isomorphism
of the groups of morphisms in the homotopy categories
$$
 \Hom_{\Hot(\sA)}(\Phi(P^\bu),J^\bu)\simeq
 \Hom_{\Hot(\sB)}(P^\bu,\Psi(J^\bu)).
$$

 In order to pass from this to the desired isomorphism of the groups
of morphisms in the unbounded derived categories, one can notice that
for any (unbounded) complex $A^\bu\in\sC(\sA)$ endowed with
a quasi-isomorphism $J^\bu\rarrow A^\bu$ of complexes in $\sA$ there
exists a bounded below complex $I^\bu\in\sC^+(\sJ)$ together with
a quasi-isomorphism $A^\bu\rarrow I^\bu$ of complexes in~$\sA$.
 The composition $J^\bu\rarrow A^\bu\rarrow I^\bu$ is then
a quasi-isomorphism of bounded below complexes in~$\sJ$.
 Similarly, for any (unbounded) complex $B^\bu\in\sC(\sB)$ endowed
with a quasi-isomorphism $B^\bu\rarrow P^\bu$ of complexes in $\sB$
there exists a bounded above complex $Q^\bu\in\sC^-(\sP)$ together with
a quasi-isomorphism $Q^\bu\rarrow B^\bu$ of complexes in~$\sB$.
 The composition $Q^\bu\rarrow B^\bu\rarrow P^\bu$ is then
a quasi-isomorphism of bounded above complexes in~$\sP$.
\end{proof}

 Restricting to the full subcategories $\sD^\b(\sE)\subset\sD^+(\sJ)
\subset\sD(\sA)$ and $\sD^\b(\sF)\subset\sD^-(\sP)\subset\sD(\sB)$,
we conclude that the derived functor
$\boL\Phi\:\sD^\b(\sF)\rarrow\sD^\b(\sE)$ \eqref{L-Phi-D-b} is
left adjoint to the derived functor
$\boR\Psi\:\sD^\b(\sE)\rarrow\sD^\b(\sF)$~\eqref{R-Psi-D-b}.
 Replacing all the exact categories with the categories of unbounded
complexes in them, we see that the derived functor
$\boL\Phi_\sC\:\sD^\b(\sC(\sF))\rarrow\sD^\b(\sC(\sE))$
\eqref{L-Phi-C-D-b} is left adjoint to the derived functor
$\boR\Psi_\sC\:\sD^\b(\sC(\sE))\rarrow
\sD^\b(\sC(\sF))$~\eqref{R-Psi-C-D-b}.

 In order to pass to the desired adjunction between the derived
functors $\boR\Psi\:\sD^\star(\sE)\rarrow\sD^\star(\sF)$
\eqref{R-Psi-D-star} and $\boL\Phi\:\sD^\star(\sF)\rarrow\sD^\star(\sE)$
\eqref{L-Phi-D-star}, it remains to apply the next (well-known) lemma.
{\hbadness=1600\par}

\begin{lem} \label{descent-adjunction-lem}
 Suppose that we are given two commutative diagrams of triangulated
functors
$$\dgARROWLENGTH=2.5em
\begin{diagram} 
\node{\sD_1}\arrow[3]{e,t}{G}\arrow[2]{s,A}
\node[3]{\sD_2}\arrow[2]{s,A} \\ \\
\node{\overline\sD_1}\arrow[3]{e,t}{\overline G}
\node[3]{\overline\sD_2}
\end{diagram}
\qquad\quad
\begin{diagram} 
\node{\sD_1}\arrow[2]{s,A}
\node[3]{\sD_2}\arrow[2]{s,A}\arrow[3]{w,t}{F} \\ \\
\node{\overline\sD_1}
\node[3]{\overline\sD_2}\arrow[3]{w,t}{\overline F}
\end{diagram}
$$
where the vertical arrows are Verdier quotient functors.
 Suppose further that the functor $F\:\sD_2\rarrow\sD_1$ is left
adjoint to the functor $G\:\sD_1\rarrow\sD_2$.
 Then the functor $\overline F\:\overline\sD_2\rarrow\overline\sD_1$
is also naturally left adjoint to the functor
$\overline G\:\overline\sD_1\rarrow\overline\sD_2$.
\end{lem}

\begin{proof}
 The adjunction morphisms $F\circ G\rarrow\Id_{\sD_1}$ and
$\Id_{\sD_2}\rarrow G\circ F$ induce adjunction morphisms
$\overline F\circ \overline G\rarrow\Id_{\overline\sD_1}$ and
$\Id_{\overline\sD_2}\rarrow \overline G\circ \overline F$.
\end{proof}

\subsection{Triangulated equivalences}
 The following theorem describes the situation in which the adjoint
triangulated functors $\boR\Psi$ and $\boL\Phi$ turn out to be
triangulated equivalences (cf.\ the proofs of~\cite[Theorems~4.9
and~5.10]{Pmgm}, \cite[Theorems~3.6 and~4.3]{Pmc}, and
\cite[Theorem~7.6]{PMat}, where this technique was used).

\begin{thm} \label{appx-triangulated-equivalence-thm}
 In the context of Section~\ref{deriving-adjoints-appx},
suppose that the adjoint derived functors\/ $\boR\Psi\:\sD^\b(\sE)
\rarrow\sD^\b(\sF)$ \eqref{R-Psi-D-b} and\/
$\boL\Phi\:\sD^\b(\sF)\rarrow\sD^\b(\sE)$ \eqref{L-Phi-D-b}
are mutually inverse triangulated equivalences.
 Then so are the adjoint derived functors\/ $\boR\Psi\:\sD^\star(\sE)
\rarrow\sD^\star(\sF)$ \eqref{R-Psi-D-star} and\/
$\boL\Phi\:\sD^\star(\sF)\rarrow\sD^\star(\sE)$ \eqref{L-Phi-D-star}
for all the symbols\/ $\star=\b$, $+$, $-$, $\varnothing$, $\abs+$,
$\abs-$, or~$\abs$, and also for any one of the symbols\/ $\star=\co$
or\/~$\ctr$ for which these two functors are defined by
the constructions of Sections~\ref{derived-functor-construction-appx}%
\+-\ref{dual-setting-appx}.

 Moreover, assume that the adjunction morphisms\/
$\boL\Phi(\Psi(J))\rarrow J$ and $P\rarrow\boR\Psi(\Phi(P))$ are
isomorphisms in\/ $\sD^\b(\sE)$ and\/ $\sD^\b(\sF)$ for all
objects $J\in\sJ$ and $P\in\sP$.
 Then the adjoint derived functors~\eqref{R-Psi-D-star}
and~\eqref{L-Phi-D-star} are mutually inverse triangulated
equivalences for all the symbols\/~$\star$ for which they are defined.
\end{thm}

\begin{proof}
 A complex of complexes in an exact category $\sG$ is acyclic if and
only if it is termwise acyclic.
 In other words, one can consider the family of functors
$\Theta^n_\sG\:\sC(\sG))\rarrow\sG$, indexed by the integers~$n$,
assigning to a complex $G^\bu$ its $n$\+th term~$G^n$.
 Then the family of induced triangulated functors $\Theta^n_\sG\:
\sD(\sC(\sG))\rarrow\sD(\sG)$ is conservative in total.
 This means that for any nonzero object $G^{\bu,\bu}\in\sD(\sC(\sG))$
there exists $n\in\boZ$ such that $\Theta^n_\sG(G^{\bu,\bu})\ne0$
in $\sD(\sG)$.

 Now the two such functors $\Theta^n_\sE\:\sD^\b(\sC(\sE))\rarrow
\sD^\b(\sE)$ and $\Theta^n_\sF\:\sD^\b(\sC(\sF))\rarrow\sD^\b(\sF)$
form commutative diagrams with the adjoint derived
functors~(\ref{R-Psi-D-b}\+-\ref{R-Psi-C-D-b})
and~(\ref{L-Phi-D-b}\+-\ref{L-Phi-C-D-b}).
 Therefore, the adjoint functors~\eqref{R-Psi-C-D-b}
and~\eqref{L-Phi-C-D-b} are mutually inverse equivalences whenever
so are the adjoint functors~\eqref{R-Psi-D-b} and~\eqref{L-Phi-D-b}.
 It remains to point out that, in the context of
Lemma~\ref{descent-adjunction-lem}, the two adjoint functors
$\overline F$ and $\overline G$ are mutually inverse equivalences
whenever so are the two adjoint functors $F$ and~$G$.

 This proves the first assertion of the theorem, and in fact somewhat
more than that.
 We have shown that the adjunction morphism
$\boL\Phi(\boR\Psi(E^\bu))\rarrow E^\bu$ is an isomorphism in
$\sD^\star(\sE)$ whenever for every $n\in\boZ$ the adjunction morphism
$\boL\Phi(\boR\Psi(E^n))\rarrow E^n$ is an isomorphism in $\sD^\b(\sE)$.
 Now, replacing an object $E\in\sE$ by its coresolution $J^\bu$ by
objects from $\sJ$, viewed as an object in $\sD^\star(\sE)$ with
$\star=+$, we see that it suffices to check that the adjunction
morphism is an isomorphism for an object $J\in\sJ$.
 Similarly, the adjunction morphism $F^\bu\rarrow
\boR\Psi(\boL\Phi(F^\bu))$ is an isomorphism in $\sD^\star(\sF)$
whenever for every $n\in\boZ$ the adjunction morphism
$F^n\rarrow\boR\Psi(\boL\Phi(F^n))$ is an isomorphism in $\sD^\b(\sF)$.
 Replacing an object $F\in\sF$ by its resolution $P^\bu$ by objects
from $\sP$, viewed as an object in $\sD^\star(\sF)$ with $\star=-$,
we see that it suffices to check that the adjunction morphism is
an isomorphism for an object $P\in\sP$.
\end{proof}


\bigskip
\begin{thebibliography}{99}
\smallskip

\bibitem{AJL}
 L.~Alonso, A.~Jerem\'\i as, J.~Lipman.
   Duality and flat base change on formal schemes.
In: L.~Alonso Tarr\'\i o, A.~Jerem\'\i as L\'opez, J.~Lipman,
Studies in Duality on Noetherian Formal Schemes and Non-Noetherian
Ordinary Schemes, Contemporary Math., 244, AMS, Providence, 1999, 
p.~3--90.  Correction: \textit{Proc.\ Amer.\ Math.\ Soc.}\
\textbf{131}, \#2, p.~351--357, 2003.

\bibitem{AF}
 L.~L.~Avramov, H.-B.~Foxby.
   Ring homomorphisms and finite Gorenstein dimension.
\textit{Proceedings of the London Math.\ Soc.}\ \textbf{75}, \#2,
p.~241--270, 1997.

\bibitem{BT}
 S.~Bazzoni, M.~Tarantino.
   Recollements from cotorsion pairs.
\textit{Journ.\ of Pure and Appl.\ Algebra} \textbf{223}, \#5,
p.~1833--1855, 2019.  \texttt{arXiv:1712.04781 [math.CT]}

\bibitem{Bec}
 H.~Becker.
   Models for singularity categories.
\textit{Advances in Math.}\ \textbf{254}, p.~187--232, 2014.
\texttt{arXiv:1205.4473 [math.CT]}

\bibitem{Bier}
 R.~Bieri.
   Homological dimension of discrete groups.
Queen Mary College Mathematical Notes, Mathematics Department,
Queen Mary College, London, 1976.

\bibitem{Bour}
 N.~Bourbaki.
   Elements of mathematics.  Commutative algebra.
Addison-Wesley, Reading, Mass., 1972.
Translated from: \'El\'ements de math\'ematique.
Alg\`ebre commutative.  Hermann, Paris, 1964--1969.

\bibitem{BG}
 D.~Bravo, J.~Gillespie.
   Absolutely clean, level, and Gorenstein AC\+injective
complexes.
\textit{Communicat.\ in Algebra} \textbf{44}, \#5,
p.~2213--2233, 2016.  \texttt{arXiv:1408.7089 [math.AT]}

\bibitem{BGH}
 D.~Bravo, J.~Gillespie, M.~Hovey.
   The stable module category of a general ring.
Electronic preprint \texttt{arXiv:1405.5768 [math.RA]}.

\bibitem{BP}
 D.~Bravo, M.~A.~P\'erez.
   Finiteness conditions and cotorsion pairs.
\textit{Journ.\ of Pure and Appl.\ Algebra} \textbf{221}, \#6,
p.~1249--1267, 2017.  \texttt{arXiv:1510.08966 [math.RA]}

\bibitem{Bueh}
 T.~B\"uhler.
   Exact categories.
\textit{Expositiones Math.}\ \textbf{28}, \#1, p.~1--69, 2010.
\texttt{arXiv:0811.1480 [math.HO]}

\bibitem{Chr}
 L.~W.~Christensen.
   Semi-dualizing complexes and their Auslander categories.
\textit{Trans.\ of the Amer.\ Math.\ Soc.}\ \textbf{353}, \#5,
p.~1839--1883, 2001.

\bibitem{CFH}
 L.~W.~Christensen, A.~Frankild, H.~Holm.
   On Gorenstein projective, injective, and flat
dimensions---A functorial description with applications.
\textit{Journ.\ of Algebra} \textbf{302}, \#1, p.~231--279, 2006.
\texttt{arXiv:math.AC/0403156}

\bibitem{CPS}
 E.~Cline, B.~Parshall, L.~Scott.
    Derived categories and Morita theory.
\textit{Journ.\ of Algebra} \textbf{104}, \#2, p.~397--409, 1986.

\bibitem{Co}
 R.~Colpi.
    Tilting in Grothendieck categories.
\textit{Forum Mathematicum} \textbf{11}, \#6, p.~735--759, 1999.

\bibitem{CT}
 R.~Colpi, J.~Trlifaj.
    Tilting modules and tilting torsion theories.
\textit{Journ.\ of Algebra} \textbf{178}, \#2, p.~614--634, 1995.

\bibitem{Del}
 P.~Deligne.
   Cohomologie \`a supports propres.
SGA4, Tome 3. \textit{Lecture Notes in Math.}\ \textbf{305},
Springer-Verlag, Berlin--Heidelberg--New York, 1973,
p.~250--480.

\bibitem{DCH}
 G.~C.~Drummond-Cole, J.~Hirsh.
   Model structures for coalgebras.
\textit{Proc.\ of the Amer.\ Math.\ Soc.}\ \textbf{144}, \#4,
p.~1467--1481, 2016.  \texttt{arXiv:1411.5526 [math.AT]}

\bibitem{DS}
 D.~Dugger, B.~Shipley.
   $K$\+theory and derived equivalences.
\textit{Duke Math.\ Journ.}\ \textbf{124}, \#3, p.~587--617, 2004.
\texttt{arXiv:math.KT/0209084}

\bibitem{Ef}
 A.~I.~Efimov.
   Homotopy finiteness of some DG categories from algebraic geometry.
Electronic preprint \texttt{arXiv:1308.0135 [math.AG]}, to appear in
\textit{Journ.\ of the European Math.\ Soc.}

\bibitem{ET}
 P.~C.~Eklof, J.~Trlifaj.
   How to make Ext vanish.
\textit{Bulletin of the London Math.\ Society} \textbf{33}, \#1,
p.~41--51, 2001.

\bibitem{EJLR}
 E.~E.~Enochs, O.~M.~G.~Jenda, J.~A.~L\'opez-Ramos.
   Dualizing modules and $n$\+perfect rings.
\textit{Proceedings of the Edinburgh Math. Society} \textbf{48},
\#1, p.~75--90, 2005.

\bibitem{EH}
 E.~E.~Enochs, H.~Holm.
   Cotorsion pairs associated with Auslander categories.
\textit{Israel Journ.\ of Math.} \textbf{174}, \#1, p.~253--268,
2009.  \texttt{arXiv:math.AC/0609291}

\bibitem{Fo}
 H.-B.~Foxby.
   Gorenstein modules and related modules.
\textit{Math.\ Scand.}\ \textbf{31}, p.~267--284, 1972.

\bibitem{FJ}
 A.~Frankild, P.~J\o rgensen.
   Foxby equivalence, complete modules, and torsion modules.
\textit{Journ.\ of Pure and Appl.\ Algebra} \textbf{174}, \#2,
p.~135--147, 2002.

\bibitem{HT}
 D.~Herbera, J.~Trlifaj. 
   Almost free modules and Mittag-Leffler conditions.
\textit{Advances in Math.}\ \textbf{229}, \#6, p.~3436--3467, 2012.

\bibitem{HW}
 H.~Holm, D.~White.
   Foxby equivalence over associative rings.
\textit{Journ.\ of Math.\ of Kyoto Univ.}\ \textbf{47}, \#4,
p.~781--808, 2007.  \texttt{arXiv:math.AC/0611838}

\bibitem{Il}
 L.~Illusie.
   G\'en\'eralit\'es sur les conditions de finitude dans les
cat\'egories d\'eriv\'ees.
SGA6, \textit{Lecture Notes in Math.}\ \textup{225},
Springer-Verlag, Berlin--Heidelberg--New York, 1971,
p.~78--159.

\bibitem{IK}
 S.~Iyengar, H.~Krause.
   Acyclicity versus total acyclicity for complexes over
noetherian rings.
\textit{Documenta Math.}\ \textbf{11}, p.~207--240, 2006.
\texttt{arXiv:math.AC/0506265}

\bibitem{SP}
 A.~J.~de~Jong et al.
   The Stacks Project.
Available from \texttt{https://stacks.math.columbia.edu/}
 
\bibitem{Kel}
 B.~Keller.
   Derived categories and their uses.
In: M.~Hazewinkel, Ed., \textit{Handbook of algebra}, vol.~1,
1996, p.~671--701.

\bibitem{Kra}
 H.~Krause.
    On Neeman's well-generated triangulated categories.
\textit{Documenta Math.}\ \textbf{6}, p.~121--126, 2001.

\bibitem{Kra1}
 H.~Krause.
   A Brown representability theorem via coherent functors.
\textit{Topology} \textbf{41}, \#4, p.~853--861, 2002.

\bibitem{Kra2}
 H.~Krause.
    Localization theory for triangulated categories.
In: Th.~Holm, P.~J\o rgensen and R.~Rouquier, Eds., 
\textit{Triangulated categories}, London Math.\ Lecture Note Ser.~375,
Cambridge University Press, 2010, p.~161--235.
\texttt{arXiv:0806.1324 [math.CT]}

\bibitem{Miy}
 J.~Miyachi.
   Derived categories and Morita duality theory.
\textit{Journ.\ of Pure and Appl.\ Algebra} \textbf{128}, \#2,
p.~153--170, 1998.

\bibitem{Mi}
 Y.~Miyashita.
    Tilting modules of finite projective dimension.
\textit{Math.\ Zeitschrift} \textbf{193}, \#1, p.~113--146, 1986.

\bibitem{Neem0}
 A.~Neeman.
   The derived category of an exact category.
\textit{Journ.\ of Algebra} \textbf{135}, \#2, p.~388--394, 1990.

\bibitem{N-book}
 A.~Neeman.
   Triangulated categories.
Annals of Math.\ Studies, Princeton Univ.\ Press, 2001. 449~pp. 

\bibitem{Neem}
 A.~Neeman.
   The homotopy category of flat modules, and Grothendieck duality.
\textit{Inventiones Math.}\ \textbf{174}, \#2, p.~255--308, 2008.

\bibitem{Neem2}
 A.~Neeman.
   The homotopy category of injectives.
\textit{Algebra and Number Theory} \textbf{8}, \#2, p.~429--456,
2014.

\bibitem{PSY}
 M.~Porta, L.~Shaul, A.~Yekutieli.
   On the homology of completion and torsion.
\textit{Algebras and Represent.\ Theory} \textbf{17}, \#1,
p.~31--67, 2014.  \texttt{arXiv:1010.4386 [math.AC]}\,.
Erratum in \textit{Algebras and Represent.\ Theory} \textbf{18},
\#5, p.~1401--1405, 2015.  \texttt{arXiv:1506.07765 [math.AC]}

\bibitem{PSY2}
 M.~Porta, L.~Shaul, A.~Yekutieli.
   Cohomologically cofinite complexes.
\textit{Communicat.\ in Algebra} \textbf{43}, \#2, p.~597--615, 2015.
\texttt{arXiv:1208.4064 [math.AC]}

\bibitem{Psemi}
 L.~Positselski.
   Homological algebra of semimodules and semicontramodules:
Semi-infinite homological algebra of associative algebraic structures.
 Appendix~C in collaboration with D.~Rumynin; Appendix~D in
collaboration with S.~Arkhipov.
 Monografie Matematyczne vol.~70, Birkh\"auser/Springer Basel, 2010. 
xxiv+349~pp. \texttt{arXiv:0708.3398 [math.CT]}

\bibitem{Pkoszul}
 L.~Positselski.
   Two kinds of derived categories, Koszul duality, and
comodule-contramodule correspondence.
\textit{Memoirs of the American Math.\ Society} \textbf{212},
\#996, 2011.  vi+133~pp.  \texttt{arXiv:0905.2621 [math.CT]}

\bibitem{Pcosh}
 L.~Positselski.
   Contraherent cosheaves.
Electronic preprint \texttt{arXiv:1209.2995 [math.CT]}.

\bibitem{Pmgm}
 L.~Positselski.
   Dedualizing complexes and MGM duality.
\textit{Journ.\ of Pure and Appl.\ Algebra} \textbf{220}, \#12,
p.~3866--3909, 2016.  \texttt{arXiv:1503.05523 [math.CT]}

\bibitem{Pfp}
 L.~Positselski.
   Coherent rings, fp\+injective modules, dualizing complexes, and
covariant Serre--Grothendieck duality.
\textit{Selecta Math.\ (New Ser.)} \textbf{23}, \#2, p.~1279--1307,
2017.  \texttt{arXiv:1504.00700 [math.CT]}

\bibitem{Pcta}
 L.~Positselski.
   Contraadjusted modules, contramodules, and reduced cotorsion modules.
\textit{Moscow Math.\ Journ.}\ \textbf{17}, \#3, p.~385--455, 2017.
\texttt{arXiv:1605.03934 [math.CT]}

\bibitem{PMat}
 L.~Positselski.
   Triangulated Matlis equivalence.
\textit{Journ.\ of Algebra and its Appl.}\ \textbf{17}, \#4,
article ID~1850067, 2018.  \texttt{arXiv:1605.08018 [math.CT]}

\bibitem{Pmc}
 L.~Positselski.
   Dedualizing complexes of bicomodules and MGM duality over coalgebras.
\textit{Algebras and Represent.\ Theory} \textbf{21}, \#4, p.~737--767,
2018.  \texttt{arXiv:1607.03066 [math.CT]}

\bibitem{PR}
 L.~Positselski, J.~Rosick\'y.
   Covers, envelopes, and cotorsion theories in locally presentable
abelian categories and contramodule categories.
\textit{Journ.\ of Algebra} \textbf{483}, p.~83--128, 2017.
\texttt{arXiv:1512.08119 [math.CT]}

\bibitem{PS0}
 L.~Positselski, J.~\v St\!'ov\'\i\v cek.
   The tilting-cotilting correspondence.
\textit{Internat.\ Math.\ Research Notices}, published online at
\texttt{https://doi.org/10.1093/imrn/rnz116} in July~2019.
\texttt{arXiv:1710.02230 [math.CT]} {\hbadness=2625\par}

\bibitem{PS}
 L.~Positselski, J.~\v St\!'ov\'\i\v cek.
   $\infty$-tilting theory.
\textit{Pacific Journ.\ of Math.} \textbf{301}, \#1,
p.~297--334, 2019.  \texttt{arXiv:1711.06169 [math.CT]}

\bibitem{Ric}
 J.~Rickard.
   Morita theory for derived categories.
\textit{Journ.\ of the London Math.\ Soc.}\ \textbf{39}, \#3,
p.~436--456, 1989.

\bibitem{Ric2}
 J.~Rickard.
   Derived equivalences as derived functors.
\textit{Journ.\ of the London Math.\ Soc.}\ \textbf{43}, \#1,
p.~37--48, 1991.

\bibitem{SS}
 J.~\v Saroch, J.~\v St\!'ov\'\i\v cek.
   Singular compactness and definability for $\Sigma$\+cotorsion
and Gorenstein modules.
\textit{Selecta Math.\ (New Ser.)}\ \textbf{26}, \#2, article~23,
40~pp., 2020.  \texttt{arXiv:1804.09080 [math.RT]}

\bibitem{ST}
 J.~\v Saroch, J.~Trlifaj.
   Kaplansky classes, finite character and $\aleph_1$\+projectivity.
\textit{Forum Mathematicum} \textbf{24}, \#5, p.~1091--1109, 2012.

\bibitem{Sch}
 P.~Schenzel.
   Proregular sequences, local cohomology, and completion.
\textit{Math.\ Scand.}\ \textbf{92}, \#2, p.~161--180, 2003.

\bibitem{St0}
 J.~\v St\!'ov\'\i\v cek.
   Deconstructibility and the Hill Lemma in Grothendieck categories.
\textit{Forum Mathematicum} \textbf{25}, \#1, p.~193--219, 2013.
\texttt{arXiv:1005.3251 [math.CT]}

\bibitem{St}
 J.~\v St\!'ov\'\i\v cek.
   Derived equivalences induced by big cotilting modules. 
\textit{Advances in Math.}\ \textbf{263}, p.~45--87, 2014.
\texttt{arXiv:1308.1804 [math.CT]}

\bibitem{St2}
 J.~\v St\!'ov\'\i\v cek.
   On purity and applications to coderived and singularity categories.
Electronic preprint \texttt{arXiv:1412.1615 [math.CT]}\,.

\bibitem{Wak1}
 T.~Wakamatsu.
   On modules with trivial self-extensions.
\textit{Journ.\ of Algebra} \textbf{114}, \#1, p.~106--114, 1988.

\bibitem{Wak2}
 T.~Wakamatsu.
   Stable equivalence for self-injective algebras and
a generalization of tilting modules.
\textit{Journ.\ of Algebra} \textbf{134}, \#2, p.~298--325, 1990.

\bibitem{Wei}
 C.~A.~Weibel.
      An introduction to homological algebra.
Cambridge Studies in Advanced Mathematics, 38.  Cambridge University
Press, 1994.

\bibitem{Yek}
 A.~Yekutieli.
   A separated cohomologically complete module is complete.
\textit{Communicat.\ in Algebra} \textbf{43}, \#2, p.~616--622, 2015.
\texttt{arXiv:1312.2714 [math.AC]}

\bibitem{YZ}
 A.~Yekutieli, J.~J.~Zhang.
   Rings with Auslander dualizing complexes.
\textit{Journ.\ of Algebra} \textbf{213}, \#1, p.~1--51, 1999.
\texttt{arXiv:math.RA/9804005}

\end{thebibliography}
\end{document}